\documentclass[twoside]{article}
\usepackage[letterpaper, left=3.5cm, right=3.5cm, top=4cm, bottom=2.5cm]{geometry}
\usepackage{graphicx}
\usepackage{amsmath,amssymb,amsthm}
\usepackage{authblk}
\usepackage[utf8]{inputenc}
\usepackage{microtype}
\usepackage{mathrsfs}  
\usepackage{enumitem}
\usepackage{calc}
\usepackage{esint}
\usepackage{fancyhdr}
\usepackage{xcolor}
\usepackage[labelfont = bf]{caption}
\usepackage{doi}
\usepackage{subfigure}
\usepackage[numbers]{natbib}
\usepackage{float}
\usepackage{stmaryrd}
\usepackage{mathtools}
\usepackage{booktabs}
\usepackage{colortbl}
\usepackage{tabulary}
\usepackage{diagbox}
\usepackage{caption}

\newtheorem{theorem}{Theorem}[section]
\numberwithin{equation}{section}
\newtheorem{proposition}[theorem]{Proposition}

\newtheorem{corollary}[theorem]{Corollary}
\newtheorem{remark}[theorem]{Remark}
\newtheorem{lemma}[theorem]{Lemma}

\newtheorem{algorithm}[theorem]{Algorithm}

\usepackage{titlesec}
\titleformat{\section}{\normalfont\scshape\centering}{\thesection.}{0.5em}{}
\titleformat*{\subsection}{\itshape}
\titleformat*{\subsubsection}{\itshape}

\providecommand{\keywords}[1]
{
	{\small\textit{Keywords:} #1}
}

\providecommand{\MSC}[1]
{
	{\small\textit{AMS MSC (2020):~~} #1}
}

\definecolor{denim}{rgb}{0.08, 0.38, 0.74}
\definecolor{byzantium}{rgb}{0.44, 0.16, 0.39} 
\definecolor{shamrockgreen}{rgb}{0.0, 0.62, 0.38} 

\usepackage{hyperref}
\hypersetup{
	colorlinks=true,
	linkcolor=denim,
	citecolor = shamrockgreen,
	filecolor=magenta, 
	urlcolor=byzantium,
} 

\providecommand{\jumptmp}[2]{#1\llbracket{#2}#1\rrbracket}
\providecommand{\jump}[1]{\jumptmp{}{#1}}

\begin{document}
	\setlength{\abovedisplayskip}{5.5pt}
	\setlength{\belowdisplayskip}{5.5pt}
	\setlength{\abovedisplayshortskip}{5.5pt}
	\setlength{\belowdisplayshortskip}{5.5pt}

	\title{Error analysis for a Crouzeix--Raviart approximation of the obstacle problem}
	\author[1]{Sören Bartels\thanks{Email: \texttt{bartels@mathematik.uni-freiburg.de}}}
	\author[2]{Alex Kaltenbach\thanks{Email: \texttt{kaltenbach@math.tu-berlin.de}}}
	\date{\today}
	\affil[1]{\small{Department of Applied Mathematics, University of Freiburg, Hermann--Herder--Stra\ss e~10, 79104 Freiburg}}
	\affil[2]{\small{Institute of Mathematics, Technical University of Berlin, Stra\ss e des 17.\ Juni 135, 10623 Berlin}}
	\maketitle

	\pagestyle{fancy}
	\fancyhf{}
	\fancyheadoffset{0cm}
	\addtolength{\headheight}{-0.25cm}
	\renewcommand{\headrulewidth}{0pt} 
	\renewcommand{\footrulewidth}{0pt}
	\fancyhead[CO]{\textsc{Error analysis for a CR approximation of the obstacle problem}}
	\fancyhead[CE]{\textsc{S. Bartels and A. Kaltenbach}}
	\fancyhead[R]{\thepage}
	\fancyfoot[R]{}
	
	\begin{abstract}
		\hspace{-0.1mm}In \hspace{-0.1mm}the \hspace{-0.1mm}present \hspace{-0.1mm}paper, \hspace{-0.1mm}we \hspace{-0.1mm}study \hspace{-0.1mm}a \hspace{-0.1mm}Crouzeix--Raviart \hspace{-0.1mm}approximation \hspace{-0.1mm}of \hspace{-0.1mm}the \hspace{-0.1mm}obstacle~\hspace{-0.1mm}\mbox{problem}, which imposes the obstacle constraint in the midpoints (\textit{i.e.}, barycenters) of the elements of  a  triangulation.
We establish \textit{a priori} error estimates imposing natural~regularity~assumptions, which \hspace{-0.1mm}are \hspace{-0.1mm}optimal, \hspace{-0.1mm}and \hspace{-0.1mm}the \hspace{-0.1mm}reliability \hspace{-0.1mm}and \hspace{-0.1mm}efficiency \hspace{-0.1mm}of \hspace{-0.1mm}a 
     primal-dual~\hspace{-0.1mm}type~\hspace{-0.1mm}a~\hspace{-0.1mm}posteriori~\hspace{-0.1mm}error estimator for general obstacles and  involving data oscillation terms stemming only from the right-hand side. Numerical experiments are carried out to support the  theoretical findings.
	\end{abstract}

	
	\MSC{35J20; 49J40; 49M29; 65N30; 65N15; 65N50.}

    \section{Introduction}\thispagestyle{empty}

    \hspace{5mm} The \hspace{-0.1mm}obstacle \hspace{-0.1mm}problem \hspace{-0.1mm}is \hspace{-0.1mm}a \hspace{-0.1mm}prototypical \hspace{-0.1mm}example \hspace{-0.1mm}of \hspace{-0.1mm}a \hspace{-0.1mm}non-smooth \hspace{-0.1mm}convex~\hspace{-0.1mm}\mbox{minimization}~\hspace{-0.1mm}\mbox{problem} \hspace{-0.1mm}with \hspace{-0.1mm}an \hspace{-0.1mm}inequality \hspace{-0.1mm}constraint \hspace{-0.1mm}that \hspace{-0.1mm}leads \hspace{-0.1mm}to \hspace{-0.1mm}a \hspace{-0.1mm}variational \hspace{-0.1mm}inequality. \hspace{-0.1mm}It~\hspace{-0.1mm}has~\hspace{-0.1mm}countless~\hspace{-0.1mm}\mbox{applications},~\hspace{-0.1mm}\textit{e.g.}, in the contexts of fluid filtration in porous media, constrained heating, elasto-plasticity,~optimal control, and financial mathematics (\textit{cf}.\ \cite{Caf98,Fried88}). 
    It is deeply related to models in free boundary value problems, the study of minimal surfaces, and the capacity of a set in potential~theory (\textit{cf}.~\cite{Caf98}). The \hspace{-0.1mm}problem \hspace{-0.1mm}is \hspace{-0.1mm}to \hspace{-0.1mm}find \hspace{-0.1mm}the \hspace{-0.1mm}equilibrium \hspace{-0.1mm}position \hspace{-0.1mm}of \hspace{-0.1mm}an \hspace{-0.1mm}elastic \hspace{-0.1mm}membrane~\hspace{-0.1mm}whose~\hspace{-0.1mm}boundary \hspace{-0.1mm}is~\hspace{-0.1mm}held~\hspace{-0.1mm}fixed and which is constrained to lie above a given obstacle.

    More \hspace{-0.15mm}precisely, \hspace{-0.15mm}given \hspace{-0.15mm}a \hspace{-0.15mm}force \hspace{-0.15mm}$f\hspace{-0.17em}\in \hspace{-0.17em} L^2(\Omega)$ \hspace{-0.15mm}and \hspace{-0.15mm}an \hspace{-0.15mm}obstacle \hspace{-0.15mm}$\chi\hspace{-0.17em}\in\hspace{-0.17em} H^1(\Omega)$~\hspace{-0.15mm}with~\hspace{-0.15mm}${\textup{tr}\,\chi\hspace{-0.17em}\leq\hspace{-0.17em} 0}$~\hspace{-0.15mm}a.e.~\hspace{-0.15mm}on~\hspace{-0.15mm}$\Gamma_D$, where $\Gamma_D\subseteq \partial\Omega$ denotes the Dirichlet part of the topological boundary $\partial\Omega$, the obstacle problem seeks for a minimizer $u\in H^1_D(\Omega)\coloneqq\{ v\in H^1(\Omega)\mid \textrm{tr}\,v=0\textrm{ a.e.\ on }\Gamma_D\}$~of~the energy functional $I\colon H^1_D(\Omega)\to \mathbb{R}\cup\{+\infty\}$, for every $v\in H^1_D(\Omega)$ defined by 
    \begin{align}
        I(v)\coloneqq\tfrac{1}{2}\|\nabla v\|_{\Omega}^2-(f,v)_{\Omega}+I_K(v)\,,\label{eq:obstacle}
    \end{align}
    where 
    \begin{align*}
        K\coloneqq \big\{v\in H^1_D(\Omega)\mid v\ge \chi\text{ a.e.\  in }\Omega\big\}\,,
    \end{align*}
    and 
    $I_K\colon H^1_D(\Omega)\to \mathbb{R}\cup\{+\infty\}$ is defined by $I_K(v)\coloneqq 0$ if $v\in K$ and $I_K(v)\coloneqq +\infty$ else.

    \subsection{Related contributions}\vspace{-1mm}

    \hspace{5mm}The numerical approximation of \eqref{eq:obstacle} has already been the subject of numerous contributions:
    Early \hspace{-0.1mm}contributions \hspace{-0.1mm}examing \hspace{-0.1mm}the \hspace{-0.1mm}\textit{a \hspace{-0.1mm}priori} \hspace{-0.1mm}and \hspace{-0.1mm}\textit{a \hspace{-0.1mm}posteriori} \hspace{-0.1mm}error \hspace{-0.1mm}analysis \hspace{-0.1mm}of \hspace{-0.1mm}approximations~\hspace{-0.1mm}of~\hspace{-0.1mm}\eqref{eq:obstacle} \hspace{-0.1mm}using \hspace{-0.1mm}the \hspace{-0.1mm}conforming \hspace{-0.1mm}Lagrange \hspace{-0.1mm}finite \hspace{-0.1mm}element \hspace{-0.1mm}can \hspace{-0.1mm}be \hspace{-0.1mm}found \hspace{-0.1mm}in~\hspace{-0.1mm}\cite{Falk74,AOL93,HK94,Korn96,Korn97,BR99,Veeser99,CN00,Veeser01,Braess05,BHS08,BC08,WW10,FPP14}, imposing the obstacle constraint in the nodes of a  triangulation, and in \cite{John92,FLN01}, enforcing the obstacle constraint in the limit via a penalization~approach.~We~\mbox{refer}~to~\cite{CN00}~for~a~short~review. Contributions addressing the \textit{a priori} and \textit{a posteriori} error analysis of an approximation~of~\eqref{eq:obstacle} deploying Discontinuous Galerkin (DG) type methods
     can be found in \cite{GP14,CEG20,Bar20}, equally imposing the obstacle constraint in the nodes of a  triangulation. 
    The first contribution addressing the \textit{a priori} error analysis of an~approximation~of~\eqref{eq:obstacle} in two dimensions deploying the  Crouzeix--Raviart element \hspace{-0.1mm}can \hspace{-0.1mm}be \hspace{-0.1mm}found \hspace{-0.1mm}in \hspace{-0.1mm}\cite{Wang03} \hspace{-0.1mm}and \hspace{-0.1mm}imposes \hspace{-0.1mm}the \hspace{-0.1mm}obstacle \hspace{-0.1mm}constraint~\hspace{-0.1mm}in~\hspace{-0.1mm}the~\hspace{-0.1mm}midpoints~\hspace{-0.1mm}(\textit{i.e.},~\hspace{-0.1mm}barycenters) of elements~of~a~triangulation.~In~\cite{Bar21},~for~homogeneous~Dirichlet boundary data~and~zero~obstacle, this result was extended to~arbitrary~dimensions.~In~\cite{CK17}, an \textit{a priori}  and \textit{a posteriori} error analysis of an approximation of \eqref{eq:obstacle}  deploying the Crouzeix--Raviart element, 
    which imposes the obstacle constraint in the integral~mean values of element sides of a  triangulation,~was~carried~out, however, only in two and three dimensions.\vspace{-1mm}\enlargethispage{10mm}

    \subsection{New contributions}\vspace{-1mm}

    \hspace{5mm}Inspired by \cite{Wang03} as well as recent contributions \cite{Bar20,Bar21,BKAFEM22}, different from the contribution~\cite{CK17}, we treat an approximation of the obstacle problem \eqref{eq:obstacle} deploying the Crouzeix--Raviart~element~that imposes the obstacle constraint in the midpoints (\textit{i.e.}, barycenters) of elements~of~a~triangulation. More precisely, given a family of regular triangulations $\{\mathcal{T}_h\}_{h>0}$, setting $f_h\coloneqq \Pi_h f\in \mathcal{L}^0(\mathcal{T}_h)$ and for $\chi_h\in \mathcal{L}^0(\mathcal{T}_h)$ approximating $\chi \in H^1(\Omega)$, our discrete obstacle problem~seeks~for~a~minimizer $u_h^{cr}\in \mathcal{S}^{1,cr}_D(\mathcal{T}_h)$ of the 
    functional $I_h^{cr}\colon \mathcal{S}^{1,cr}_D(\mathcal{T}_h)\to \mathbb{R}\cup\{+\infty\}$,~for~every ${v_h\in \mathcal{S}^{1,cr}_D(\mathcal{T}_h)}$~defined~by 
    \begin{align}
        \smash{I_h^{cr}(v_h)\coloneqq\tfrac{1}{2}\|\nabla_h v_h\|_{\Omega}^2-(f_h,\Pi_h v_h)_{\Omega}+I_{K_h^{cr}}(v_h)\,,}\label{eq:discrete_obstacle}
    \end{align}
    where 
    \begin{align*}
        \smash{K_h^{cr}\coloneqq \big\{v_h\in \mathcal{S}^{1,cr}_D(\mathcal{T}_h)\mid \Pi_h v_h\ge \chi_h\text{ a.e.\  in }\Omega\big\}\,,}
    \end{align*}
    and 
    $I_{K_h^{cr}}\colon \mathcal{S}^{1,cr}_D(\mathcal{T}_h)\to \mathbb{R}\cup\{+\infty\}$ is defined by $I_{K_h^{cr}}(v_h)\coloneqq 0$ if $v_h\in K_h^{cr}$ and $I_{K_h^{cr}}(v_h)\coloneqq +\infty$ else. Here, $\mathcal{L}^0(\mathcal{T}_h)$ denotes the space of element-wise constant functions, $\mathcal{S}^{1,cr}_D(\mathcal{T}_h)$ the Crouzeix--Raviart finite element space, \textit{i.e.}, the space of element-wise affine functions that are continuous in the midpoints (\textit{i.e.}, barycenters) of interior element sides and that vanish in the midpoints of element sides that belong to $\Gamma_D$, $\nabla_h\colon \hspace{-0.1em}\mathcal{S}^{1,cr}_D(\mathcal{T}_h)\hspace{-0.1em}\to \hspace{-0.1em}\mathcal{L}^0(\mathcal{T}_h)$ the element-wise gradient and $\Pi_h\colon L^2(\Omega)\to \mathcal{L}^0(\mathcal{T}_h)$ the (local) $L^2$-projection~operator~onto~\mbox{element-wise}~\mbox{constant}~functions.
    Imposing the obstacle constraint in the midpoints of elements follows a systematic approximation procedure for general convex minimization problems deploying the Crouzeix--Raviart~element~introduced in \cite{Bar21,BKAFEM22} and  has the advantage that the resulting discrete convex  minimization problem 
    generates discrete convex duality relations that are analogous to those in the continuous setting --up~to~\mbox{non-conforming} modifications-- and that enable a systematic~\textit{a~priori}~error~analysis~and~\textit{a~posteriori}~error~analysis:
    \begin{itemize}[noitemsep, topsep=4pt]
        \item[$\bullet$] In \cite{Bar21}, a systematic procedure for the derivation of \textit{a priori} error estimates for convex minimi-zation problems deploying the Crouzeix--Raviart element based on (discrete) convex duality relations was proposed. Following this systematic~procedure,~with~\mbox{comparably}~\mbox{little}~effort, we derive \textit{a priori} error estimates, which are optimal for natural regularity assumptions and also apply in~arbitrary~dimensions.
        More precisely, our \textit{a priori} error estimates exploit that the discrete primal-dual gap controls the convexity  measure of \eqref{eq:discrete_obstacle} and the concavity measure of its dual functional, \textit{i.e.}, that for every $v_h\in K_h^{cr}$ and $y_h\in \mathcal{R}T^0_N(\mathcal{T}_h)$,~it~holds~that
        \begin{align}
            \hspace{-5mm}\tfrac{1}{2}\|\nabla_h v_h \hspace{-0.12em}-\hspace{-0.12em}\nabla_h u_h^{cr} \|_{\Omega}^2\hspace{-0.12em}+\hspace{-0.12em}(-\smash{\overline{\lambda}}_h^{cr} ,\Pi_h(v_h\hspace{-0.12em}-\hspace{-0.12em}u_h^{cr}))_{\Omega}\hspace{-0.12em}+\hspace{-0.12em}\tfrac{1}{2}\|\Pi_h y_h\hspace{-0.12em}-\hspace{-0.12em}\Pi_h z_h^{rt} \|_{\Omega}^2\leq I_h^{cr}(v_h)\hspace{-0.12em}-\hspace{-0.12em}D_h^{rt}(y_h)\,,\label{intro:a_priori}
        \end{align}
        where $\mathcal{R}T^0_N(\mathcal{T}_h)$ denotes the Raviart--Thomas finite element space, \textit{i.e.},~the~space~of~element-wise affine vector fields that have continuous constant normal components on element sides that vanish on $\Gamma_N$,
        $z_h^{rt}\in\mathcal{R}T^0_N(\mathcal{T}_h) $ the unique discrete dual solution, \textit{i.e.}, the~maximizer~of the discrete dual energy functional $D_h^{rt}\colon \mathcal{R}T^0_N(\mathcal{T}_h)\to \mathbb{R}\cup \{-\infty\}$, and $\smash{\overline{\lambda}}_h^{cr}\in \smash{\Pi_h(\mathcal{S}^{1,cr}_D(\mathcal{T}_h))}$ the unique discrete Lagrange multiplier satisfying $\smash{\overline{\lambda}}_h^{cr}\leq 0$ a.e.\  in $\Omega$ and~for~all~${v_h\in \smash{\mathcal{S}^{1,cr}_D(\mathcal{T}_h)}}$\vspace{-0.5mm}
        \begin{align*}
            \smash{(\smash{\overline{\lambda}}_h^{cr},\Pi_h v_h)_{\Omega}=(f_h,\Pi_h v_h)_{\Omega}-(\nabla_h u_h^{cr},\nabla_h v_h)_{\Omega}}\,.
        \end{align*}
        If $\chi_h\coloneqq \Pi_h \Pi_h^{cr} \chi\in \mathcal{L}^0(\mathcal{T}_h)$, where 
        $\Pi_h^{cr}\colon H^1(\Omega)\to \mathcal{S}^{1,cr}(\Omega)$ denotes the Crouzeix--Raviart quasi-interpolation operator, 
        then $\Pi_h^{cr} u\in K_h^{cr}$. Thus, under natural regularity~assumptions, \textit{i.e.}, $u,\chi\in H^2(\Omega)$, the choices
        $v_h=\Pi_h^{cr} u\in K_h^{cr}$
        and $y_h=\Pi_h^{rt} z\in \mathcal{R}T^0_N(\mathcal{T}_h)$, where $z=\nabla u\in (H^1(\Omega))^d\cap H^2_N(\textup{div};\Omega)$ denotes the dual solution, \textit{i.e.}, the maximizer~of~the~dual energy functional $D\colon \hspace{-0.15em}(L^2(\Omega))^d\hspace{-0.15em}\to\hspace{-0.15em} \mathbb{R}\cup\{-\infty\}$,  and  ${\Pi_h^{rt}\colon \hspace{-0.15em}(H^1(\Omega))^d\hspace{-0.15em}\cap\hspace{-0.15em} H^2_N(\textup{div};\Omega)\hspace{-0.15em}\to \hspace{-0.15em}\mathcal{R}T^0_N(\mathcal{T}_h)}$ the Raviart--Thomas quasi-interpolation operator, are admissible in \eqref{intro:a_priori} and~lead~to~quasi-optimal \textit{a priori} error estimates.\vspace{1mm}\enlargethispage{5mm}

        \item[$\bullet$] In \cite{BKAFEM22}, a systematic procedure for the derivation of reliable, quasi-constant-free \textit{a posteriori} error estimates for convex minimization problems deploying the Crouzeix--Raviart~\mbox{element}~ba-sed on (discrete) convex duality relations was proposed. Following~this~\mbox{systematic}~\mbox{procedure}, we derive \textit{a posteriori} error estimates, which, by definition, are reliable and  constant-free. Apart from that, we 
        establish the efficiency of these \textit{a posteriori} error estimates for general obstacles $\chi\in H^1(\Omega)$. More precisely, our \textit{a posteriori} error estimates exploit that the primal-dual gap controls the convexity measure of \eqref{eq:obstacle} and the concavity measure of its dual functional, \textit{i.e.}, that for every $v\in K$  and  $y\in (L^2(\Omega))^d$, it holds that
        \begin{align}\label{intro:a_posteriori}        
            \smash{\tfrac{1}{2}\|\nabla v -\nabla u \|_{\Omega}^2+\langle- \Lambda ,v-u\rangle_{\Omega}+\tfrac{1}{2}\|y -z \|_{\Omega}^2\leq I(v)-D(y)}
            \,,
        \end{align}
        where $\Lambda\in (H^1_D(\Omega))^*$ is the unique Lagrange multiplier satisfying $\Lambda\leq 0$ in $(H^1_D(\Omega))^*$ 
        and for all $v\in H^1_D(\Omega)$\vspace{-1mm}
        \begin{align*}
            \smash{\langle \Lambda,v\rangle_{\Omega}=(f,v)_{\Omega}-(\nabla u,\nabla v)_{\Omega}}\,.
        \end{align*}
        For the \textit{a posteriori} error estimate \eqref{intro:a_posteriori} being practicable it is necessary to have a sufficiently accurate and computationally cheap procedure to obtain an approximation $y\in (L^2(\Omega))^d$ of the dual solution $z=\nabla u\in (L^2(\Omega))^d$ at hand. 
        In the case $f=f_h\in \mathcal{L}^0(\mathcal{T}_h)$, the discrete dual solution $z_h^{rt}\in \mathcal{R}T^0_N(\mathcal{T}_h)$ is~admissible~in~\eqref{intro:a_posteriori} and leads to a constant-free reliable and efficient \textit{a posteriori} error estimator $\eta^2_h\coloneqq I(\cdot)-D(z_h^{rt})\colon H^1_D(\Omega)\to \mathbb{R}$, which has similarities to the residual type \textit{a posteriori} estimator derived in \cite{Veeser01} but is simper
        and avoids jump terms of the obstacle that arise in the efficiency analysis in \cite{Veeser01}.   In particular, note that the discrete dual solution can cheaply be computed via the generalized Marini~formula
        \begin{align}
            z_h^{rt}=\nabla_h u_h^{cr}+\frac{\smash{\smash{\overline{\lambda}}_h^{cr}}-f_h}{d}(\textrm{id}_{\mathbb{R}^d}-\Pi_h \textrm{id}_{\mathbb{R}^d})\quad\text{ in }\mathcal{R}T^0_N(\mathcal{T}_h)\,.
        \end{align}
        A typical choice for $v\in K$ is obtained via nodal averaging $u_h^{cr}\in \mathcal{S}^{1,cr}_D(\mathcal{T}_h)$ and truncating to enforce the continuous obstacle constraint. Moreover, any conforming approximation $u_h\in K$ can be used such as a continuous Lagrange approximation $u_h^c\in K_h^c\coloneqq K\cap\mathcal{S}^{1,cr}_D(\mathcal{T}_h) $, so that our analysis also implies the full reliability and efficiency error analysis for continuous Lagrange approximations, even for general obstacles and oscillation terms only stemming from the right-hand side since lumping is not needed in our analysis.\vspace{-0.5mm}\enlargethispage{3.5mm}
    \end{itemize}
    As \hspace{-0.15mm}a \hspace{-0.15mm}whole, \hspace{-0.15mm}our \hspace{-0.15mm}approach \hspace{-0.15mm}brings \hspace{-0.15mm}together \hspace{-0.15mm}and \hspace{-0.15mm}extends \hspace{-0.15mm}ideas \hspace{-0.15mm}and \hspace{-0.15mm}concepts \hspace{-0.15mm}from \hspace{-0.15mm}\cite{BR99,Veeser99,Veeser01,Braess05,BHS08,CK17}, and leads to a full error analysis.\vspace{-1mm}

    \subsection{Outline}\vspace{-1.5mm}
    
    \hspace{5mm}\textit{This article is organized as follows:} In Section \ref{sec:preliminaries}, we introduce the~notation, the~relevant~function spaces and  finite element spaces. In Section \ref{sec:obstacle}, we give a brief review of the continuous~and~the discrete obstacle problem. 
     In Section \ref{sec:apriori}, we prove \textit{a priori} error estimates for the Crouzeix--Raviart approximation \eqref{eq:discrete_obstacle} of \eqref{eq:obstacle}, which are optimal for natural regularity assumptions.~In~Section~\ref{sec:aposteriori}, we introduce a primal-dual \textit{a posteriori} error estimator and establish its  reliability and efficiency.
   In Section~\ref{sec:experiments}, numerical experiments are carried out to confirm the  theoretical findings.~In~the Appendix~\hspace{-0.15mm}\ref{sec:medius}, \hspace{-0.51mm}we \hspace{-0.15mm}derive \hspace{-0.15mm}local \hspace{-0.15mm}efficiency \hspace{-0.15mm}estimates \hspace{-0.15mm}for \hspace{-0.15mm}the \hspace{-0.15mm}Crouzeix--Raviart~\hspace{-0.15mm}approximation~\hspace{-0.15mm}\eqref{eq:discrete_obstacle}~\hspace{-0.15mm}of~\hspace{-0.15mm}\eqref{eq:obstacle}.

    \newpage
    \section{Preliminaries}\label{sec:preliminaries}
	
	\hspace{5mm}Throughout the article, let ${\Omega\subseteq \mathbb{R}^d}$, ${d\in\mathbb{N}}$, be a bounded polyhedral Lipschitz domain~whose boundary $\partial\Omega$ is disjointly divided into a closed Dirichlet part $\Gamma_D$, for which we assume~that~${\vert \Gamma_D\vert\!>\!0}$\footnote{For a (Lebesgue) measurable set $M\subseteq \mathbb{R}^d$, $d\in \mathbb{N}$, we denote by $\vert M\vert $ its $d$-dimensional Lebesgue measure. For a $(d-1)$-dimensional submanifold $M\subseteq \mathbb{R}^d$, $d\in \mathbb{N}$, we denote by $\vert M\vert $ its $(d-1)$-dimensional~Hausdorff~measure.}, and~a~Neumann~part~$\Gamma_N$. 
    For a Lebesgue measurable set $\omega\subseteq \mathbb{R}^d$, $d\in \mathbb{N}$, and  (Lebesgue) measurable functions $u,v\colon \omega\to \mathbb{R}$, we employ the product
    \begin{align*}
        (u,v)_{\omega}\coloneqq \int_{\omega}{u\,v\,\mathrm{d}x}\,,
    \end{align*}
    whenever the right-hand side is well-defined. Analogously, for  (Lebesgue) measurable~vector~fields $z,y\colon \omega\to  \mathbb{R}^d$ and a (Lebesgue) measurable set $\omega\subseteq \Omega$, we write ${(z,y)_{\omega}\coloneqq \int_{\omega}{z\cdot y\,\mathrm{d}x}}$.
    
	\subsection{Standard function spaces}

	\hspace{5mm}We let
	\begin{align*}
		\begin{aligned}
		H^1_D(\Omega)&\coloneqq \big\{v\in L^2(\Omega)&&\hspace*{-3.25mm}\mid \nabla v\in (L^2(\Omega))^d,\, \textup{tr}\,v=0\text{ a.e.\ on }\Gamma_D\big\}\,,\\
		H^{2}_N(\textup{div};\Omega)&\coloneqq \big\{y\in (L^2(\Omega))^d&&\hspace*{-3.25mm}\mid \textup{div}\,y\in L^2(\Omega),\,\langle y\cdot n,v\rangle_{\smash{H^{\smash{\frac{1}{2}}}(\partial\Omega)}}=0\text{ for all }v\in H^1_D(\Omega)\big\}\,,
	\end{aligned}
	\end{align*}
	$\smash{H^1(\Omega)\coloneqq H^1_D(\Omega)}$ in the case $\smash{\Gamma_D=\emptyset}$, and $\smash{H^2(\textup{div};\Omega)\coloneqq H^2_N(\textup{div};\Omega)}$ in the case $\smash{\Gamma_N=\emptyset}$.  Here, 
	  $\textup{tr}\colon\hspace{-0.1em}\smash{H^1(\Omega)}\hspace{-0.1em}\to \hspace{-0.1em}\smash{L^2(\partial\Omega)}$ and  $
	\textup{tr}_n\colon\hspace{-0.1em}\smash{H^2(\textup{div};\Omega)}\hspace{-0.1em}\to\hspace{-0.1em} \smash{H^{-\frac{1}{2}}(\partial\Omega)}$ denote the~trace~and normal~trace~\mbox{operator}, respectively. More precisely, $\textup{tr}_n\,y=y\cdot n$ on $\partial \Omega$ for all $y\in (C^0(\overline{\Omega}))^d$, where $n\colon \partial \Omega\to \mathbb{S}^{d-1}$ denotes the outer unit normal vector field to $\Omega$. We  always omit $\textup{tr}(\cdot)$~and~$\textup{tr}_n(\cdot)$.
 For~a~compact~notation, 
     we  abbreviate  $\|\cdot\|_{\Omega}\coloneqq \|\cdot\|_{\smash{L^2(\Omega)}}$, $\|\cdot\|_{*,\Omega}\coloneqq \|\cdot\|_{\smash{(H^1_D(\Omega))^*}}$,~and~${\langle \cdot,\cdot\rangle_{\Omega}\coloneqq\langle \cdot,\cdot\rangle_{\smash{H^1_D(\Omega)}}}$.

 \subsection{Triangulations and standard finite element spaces}\vspace{-1mm}
	
	\hspace{5mm}Throughout the entire article, we denote by $\{\mathcal{T}_h\}_{h>0}$ a family of triangulations~of $\Omega$ (\textit{cf}.\  \cite{EG21}). Here, the parameter
	$h>0$ refers to the \textit{averaged mesh-size}, \textit{i.e.},~we~define~${h 
	\coloneqq (\vert \Omega\vert/\textup{card}(\mathcal{N}_h))^{\frac{1}{d}}}
	$, where $\mathcal{N}_h$  is the set of vertices of $\mathcal{T}_h$.  We assume that the family  of triangulations $\{\mathcal{T}_h\}_{h>0}$ is shape regular, \textit{i.e.}, denoting for
	every $T \in \mathcal{T}_h$,
	by $h_T\coloneqq \textup{diam}(T)$, the diameter of $T$, and by
    $\rho_T\coloneqq \sup\{r>0\mid \exists x\in T\,:\,B_r^d(x)\subseteq T\}$, the supremum of diameters of~inscribed~balls~in~$T$, we assume that 
    there exists a constant $\omega_0\hspace{-0.1em}>\hspace{-0.1em}0$, which does~not~depend~on~$h\hspace{-0.1em}>\hspace{-0.1em}0$,~such~that~${\max_{T\in \mathcal{T}_h}{\big\{\frac{h_T}{\rho_T}\big\}}\hspace{-0.1em}\le\hspace{-0.1em}\omega_0}$. The smallest such constant $\omega_0>0$ is called the \textit{chunkiness} of  $\{\mathcal{T}_h\}_{h>0}$. The \textit{maximum mesh-size} is defined by $h_{\max}\coloneqq \max_{T\in \mathcal{T}_h}{h_T}$.

    
    We define interior and boundary sides of $\mathcal{T}_h$ in the following way: an interior side is the closure of the 
    non-empty relative interior of $\partial T \cap \partial T'$,~where~${T, T'\in \mathcal{T}_h}$~are~two~adjacent~elements.
    For an interior side $S\coloneqq 
    \partial T \cap \partial T'\in \mathcal{S}_h$, where $T,T'\in \mathcal{T}_h$, the side patch is~defined~by~$\omega_S\coloneqq  T \cup
    T'$. A boundary side is the closure of the non-empty relative interior of
    $\partial T \cap \partial \Omega$, where $T\in \mathcal{T}_h$ denotes a boundary~element~of~$\mathcal{T}_h$.  For a boundary side $S\coloneqq  \partial T \cap \partial
    \Omega$, the side patch is defined by  $\omega_S\coloneqq  T $. 
    Eventually, 
    by $\mathcal{S}_h^{i}$, we denote the set of 
 interior sides,
    and by $\mathcal{S}_h$, we denote~the~set~of~all~sides.
    
    For (Lebesgue) measurable~functions~${u,v\colon\mathcal{S}_h\to \mathbb{R}}$ and $\mathcal{M}_h\subseteq \mathcal{S}_h$, we employ the product
    \begin{align*}
        (u,v)_{\mathcal{M}_h}\coloneqq \sum_{S\in \mathcal{M}_h}{(u,v)_S}\,,\quad\text{ where }(u,v)_S\coloneqq\int_S{uv\,\mathrm{d}s}\,,
    \end{align*}
    whenever all integrals are well-defined. Analogously, for  (Lebesgue) measurable vector fields $z,y\colon \mathcal{S}_h\hspace{-0.1em}\to\hspace{-0.1em} \mathbb{R}^d$ \hspace{-0.1mm}and \hspace{-0.1mm}$\mathcal{M}_h\subseteq \mathcal{S}_h$, we write ${(z,y)_{\mathcal{M}_h}\hspace{-0.15em}\coloneqq\hspace{-0.15em} \sum_{S\in \mathcal{M}_h}{(z,y)_S}}$,~where~${(z,y)_S\coloneqq\int_S{z\cdot y\,\mathrm{d}s}}$.
    
	\pagebreak
	For $k\in \mathbb{N}\cup\{0\}$ and $T\in \mathcal{T}_h$, let $\mathbb{P}_k(T)$ denote the set of polynomials of maximal~degree~$k$~on~$T$. Then, for $k\in \mathbb{N}\cup\{0\}$~and $\ell\in  \mathbb{N}$,  the sets of continuous and~\mbox{element-wise}~polynomial functions, respectively, are defined by
	\begin{align*}
	\begin{aligned}
	\mathcal{S}^k(\mathcal{T}_h)&\coloneqq 	\big\{v_h\in C^0(\overline{\Omega})\hspace*{-3mm}&&\mid v_h|_T\in\mathbb{P}_k(T)\text{ for all }T\in \mathcal{T}_h\big\}\,,\\
	\mathcal{L}^k(\mathcal{T}_h)&\coloneqq    \big\{v_h\in L^\infty(\Omega)\hspace*{-3mm}&&\mid v_h|_T\in \mathbb{P}_k(T)\text{ for all }T\in \mathcal{T}_h\big\}\,.
	\end{aligned}
	\end{align*}
	The element-wise constant mesh-size function $h_\mathcal{T}\in \mathcal{L}^0(\mathcal{T}_h)$ is defined~by~${h_\mathcal{T}|_T\coloneqq h_T}$~for~all~${T\in \mathcal{T}_h}$.
	The side-wise constant mesh-size function $h_\mathcal{S}\in \mathcal{L}^0(\mathcal{S}_h)$ is defined~by~${h_\mathcal{S}|_S\coloneqq h_S}$~for~all~${S\in \mathcal{S}_h}$, where $h_S\coloneqq \textup{diam}(S)$ for all $S\in \mathcal{S}_h$.
	Then, for every~${T\in \mathcal{T}_h}$~and~${S\in \mathcal{S}_h}$,  we denote by $x_T\coloneqq \frac{1}{d+1}\sum_{z\in \mathcal{N}_h\cap T}{z}\in T$  and $x_S\coloneqq \frac{1}{d}\sum_{z\in \mathcal{N}_h\cap S}{z}\in S$,  the barycenters~of~$T$~and~$S$, respectively. Moreover, the (local) $L^2$-projection operator onto element-wise constant functions~or~vector~fields, respectively,  is denoted by\vspace{-0.5mm}
	\begin{align*}
	\smash{\Pi_h\colon (L^1(\Omega))^{\ell}\to (\mathcal{L}^0(\mathcal{T}_h))^\ell\,.}
	\end{align*}
    There exists~a~constant~${c_{\Pi}\hspace{-0.1em}>\hspace{-0.1em}0}$, depending only on the chunkiness $\omega_0>0$, such that for every $v\in (L^2(\Omega))^{\ell}$, $\ell\in  \mathbb{N}$, and $T\in \mathcal{T}_h$, it holds that (\textit{cf}.\ \cite[Thm.\  18.16]{EG21})
    \begin{itemize}[noitemsep,topsep=2pt,leftmargin=!,labelwidth=\widthof{\quad(L0.3)},font=\itshape]
        \item[(L0.1)] \hypertarget{L0.1}{} $\|\Pi_h v\|_T\leq \| v\|_T\,$,\vspace{0.5mm}
        \item[(L0.2)] \hypertarget{L0.2}{} $\|v-\Pi_h v\|_T\leq c_{\Pi}\,h_T\,\|\nabla v\|_T$ if $v\in (H^1(T))^{\ell}\,$.
    \end{itemize}
 	
	\subsubsection{Crouzeix--Raviart element}
 
	\qquad The Crouzeix--Raviart finite element space (\textit{cf}.\ \cite{CR73}) is defined as the space of element-wise affine functions that are continuous in the barycenters of inner element sides, \textit{i.e.},\footnote{Here, for every inner side $S\in\mathcal{S}_h^{i}$, the jump is defined by $\jump{v_h}_S\coloneqq v_h|_{T_+}-v_h|_{T_-}$ on $S$, where $T_+, T_-\in \mathcal{T}_h$ satisfy $\partial T_+\cap\partial  T_-=S$, and for every boundary side $S\in\mathcal{S}_h\cap\partial \Omega$, the jump is defined by $\jump{v_h}_S\coloneqq v_h|_T$ on $S$, where $T\in \mathcal{T}_h$ satisfies $S\subseteq \partial T$.}
	\begin{align*}\mathcal{S}^{1,cr}(\mathcal{T}_h)\coloneqq \big\{v_h\in \mathcal{L}^1(\mathcal{T}_h)\mid \jump{v_h}_S(x_S)=0\text{ for all }S\in \mathcal{S}_h^{i}\big\}\,.
	\end{align*}
    The Crouzeix--Raviart finite element space with  homogeneous Dirichlet boundary condition~on~$\Gamma_D$ is defined as the space of  
	Crouzeix--Raviart finite element functions that vanish in the barycenters of boundary~element~sides that belong to  $\Gamma_D$, \textit{i.e.},
	\begin{align*}
			\smash{\mathcal{S}^{1,cr}_D(\mathcal{T}_h)}\coloneqq \big\{v_h\in\smash{\mathcal{S}^{1,cr}(\mathcal{T}_h)}\mid v_h(x_S)=0\text{ for all }S\in \mathcal{S}_h\cap \Gamma_D\big\}\,.
	\end{align*}
    The functions $\varphi_S\in \smash{\mathcal{S}^{1,cr}(\mathcal{T}_h)}$, $S\in \mathcal{S}_h$, that satisfy the Kronecker~property $\varphi_S(x_{S'})=\delta_{S,S'}$ for all $S,S'\in \mathcal{S}_h$, form a basis of $\smash{\mathcal{S}^{1,cr}(\mathcal{T}_h)}$. Then, 
    the functions 	 $\varphi_S\in \smash{\mathcal{S}^{1,cr}_D(\mathcal{T}_h)}$, $S\in \mathcal{S}_h\setminus\Gamma_D$, form a basis of $\smash{\smash{\mathcal{S}^{1,cr}_D(\mathcal{T}_h)}}$.  There exists a constant $c^{cr}_{P}>0$, depending only on the chunkiness $\omega_0>0$, such that for every $v_h\in \smash{\mathcal{S}^{1,cr}_D(\mathcal{T}_h)}$, there holds the \textit{discrete Poincar\'e inequality}\enlargethispage{5mm}
 \begin{align}\label{discrete_poincare}
     \|v_h\|_{\Omega}\leq c^{cr}_{P}\,\|\nabla_hv_h\|_{\Omega}\,,
 \end{align}
    where
    $\nabla_h\colon \mathcal{L}^1(\mathcal{T}_h) \to (\mathcal{L}^0(\mathcal{T}_h))^d$, defined by $\nabla_hv_h|_T\coloneqq \nabla(v_h|_T)$  for all $v_h\in \mathcal{L}^1(\mathcal{T}_h)$~and $T\in \mathcal{T}_h$, is the element-wise gradient.
 The canonical interpolation operator $\smash{\Pi_h^{cr}\colon H^1_D(\Omega)\to \smash{\mathcal{S}^{1,cr}_D(\mathcal{T}_h)}}$, for every $v\in H^1_D(\Omega)$ defined by
	\begin{align}
		\Pi_h^{cr}v\coloneqq \sum_{S\in \mathcal{S}_h}{v_S\,\varphi_S}\,,\quad\text{ where } v_S\coloneqq \fint_S{v\,\textup{d}s}\,,\label{CR-interpolant}
	\end{align}
	preserves local averages of gradients, \textit{i.e.}, $\nabla_h(\Pi_h^{cr}v)=\Pi_h(\nabla v)$ a.e.\ in  $\Omega$  for all ${v\in H^1_D(\Omega)}$.  There exists a constant $c_{cr}>0$, depending only on the chunkiness $\omega_0>0$, such that for~every $v\in H^1_D(\Omega)$ and $T\in \mathcal{T}_h$, it holds that (\textit{cf}.\  \cite[Rem.\ 4.4 \& Thm.\ 4.6]{DR07})
     \begin{itemize}[noitemsep,topsep=2pt,leftmargin=!,labelwidth=\widthof{\quad(CR.3)},font=\itshape]
        \item[(CR.1)] \hypertarget{CR.1}{} $\|\nabla_h \Pi_h^{cr} v\|_T\leq \|\nabla v\|_T\,$;\vspace{0.5mm}
        \item[(CR.2)] \hypertarget{CR.2}{} $\|v-\Pi_h^{cr} v\|_T\leq c_{cr}\,h_T\,\|\nabla v\|_{\omega_T}\,$;\vspace{0.5mm}
        \item[(CR.3)] $\|v-\Pi_h^{cr} v\|_T+h_T\|\nabla(v-\Pi_h^{cr} v)\|_T\leq c_{cr}\,h_T^2\,\|D^2 v\|_{\omega_T}$ if $v\in H^2(T)\,$.
    \end{itemize}
	
	\subsubsection{Raviart--Thomas element}\enlargethispage{1mm}
 
	\qquad The Raviart--Thomas finite element space (of lowest order) (\textit{cf}.\ \cite{RT75}) is defined as the space of element-wise  affine vector fields that have continuous constant normal components on interior elements sides, \textit{i.e.},\footnote{For every inner side $S\hspace{-0.1em}\in\hspace{-0.1em}\mathcal{S}_h^{i}$, the normal jump is defined by $\jump{y_h\cdot n}_S\hspace{-0.1em}\coloneqq \hspace{-0.1em}\smash{y_h|_{T_+}\cdot n_{T_+}+y_h|_{T_-}\cdot n_{T_-}}$ on $S$, where $T_+, T_-\hspace{-0.1em}\in\hspace{-0.1em} \mathcal{T}_h$~satisfy~$\smash{\partial T_+\cap\partial  T_-\hspace{-0.1em}=\hspace{-0.1em}S}$,  and for every $T\in \mathcal{T}_h$, $\smash{n_T\colon\partial T\to \mathbb{S}^{d-1}}$ denotes the outward unit normal vector field~to~$ T$, 
	and  for every boundary side $\smash{S\in\mathcal{S}_h\cap\partial \Omega}$, the normal jump is defined by $\smash{\jump{y_h\cdot n}_S\coloneqq \smash{y_h|_T\cdot n}}$ on $S$, where $T\in \mathcal{T}_h$ satisfies $S\subseteq \partial T$.}
	\begin{align*}
        \mathcal{R}T^0(\mathcal{T}_h)\coloneqq \left\{y_h\in (\mathcal{L}^1(\mathcal{T}_h))^d\;\left|\;
        \begin{aligned}
           & \smash{y_h|_T\cdot         n_T=\textup{const}\text{ on }\partial T\text{ for  all }T\in \mathcal{T}_h\,,}\\ 
        &	\jump{y_h\cdot n}_S=0\text{ on }S\text{ for all }S\in \mathcal{S}_h^{i}
        \end{aligned}\right.\right\}\,.
	\end{align*}
	The \hspace{-0.1mm}Raviart--Thomas \hspace{-0.1mm}finite \hspace{-0.1mm}element \hspace{-0.1mm}space \hspace{-0.1mm}with \hspace{-0.1mm}homogeneous \hspace{-0.1mm}normal \hspace{-0.1mm}boundary \hspace{-0.1mm}condition~\hspace{-0.1mm}on~\hspace{-0.1mm}$\Gamma_N$ is defined as the space of Raviart--Thomas vector fields whose normal components~vanish~on~$\Gamma_N$,~\textit{i.e.},
	\begin{align*}
		\mathcal{R}T^{0}_N(\mathcal{T}_h)\coloneqq \big\{y_h\in	\mathcal{R}T^0(\mathcal{T}_h)\mid y_h\cdot n=0\text{ on }\Gamma_N\big\}\,.
	\end{align*}
    The vector fields $\psi_S\in \mathcal{R}T^0(\mathcal{T}_h)$,  $S\in  \mathcal{S}_h$, that satisfy the Kronecker property $\psi_S|_{S'}\cdot n_{S'}=\delta_{S,S'}$ on $S'$ for all $S'\in \mathcal{S}_h$, where $n_S$ for all $S\in \mathcal{S}_h$ is the unit normal vector on $S$ pointing from $T_-$ to $T_+$ if $T_+\cap T_-=S\in \mathcal{S}_h$, form a  basis of  $\mathcal{R}T^0(\mathcal{T}_h)$.
    Then, the vector fields  $\psi_S\in \smash{\mathcal{R}T^{0}_N(\mathcal{T}_h)}$, ${S\in \mathcal{S}_h\setminus\Gamma_N}$ form a basis of $\mathcal{R}T^{0}_N(\mathcal{T}_h)$. The canonical interpolation~operator~$\Pi_h^{rt}\colon (H^1(\Omega))^d\cap H^2_N(\textup{div};\Omega)\to  \smash{\mathcal{R}T^{0}_N(\mathcal{T}_h)}$, for every $\smash{y\in (H^1(\Omega))^d\cap H^2_N(\textup{div};\Omega)}$ defined~by
	\begin{align}
		\Pi_h^{rt}  y\coloneqq \sum_{S\in \mathcal{S}_h}{y_S\,\psi_S}\,,\quad\text{ where } y_S\coloneqq \fint_S{y\cdot n_S\,\textup{d}s}\,,\label{RT-interpolant}
	\end{align}
	preserves local averages of divergences, \textit{i.e.}, $\textup{div}(\Pi_h^{rt}y)=\Pi_h(\textup{div}\,y)$ a.e.\ in  $\Omega$
	for all $y\in (H^1(\Omega))^d \cap H^2_N(\textup{div};\Omega)$.  There exists a constant $c_{\textit{rt}}>0$, depending only on the chunkiness $\omega_0>0$,~such~that for every $\smash{y\in (H^1(\Omega))^d\cap H^2_N(\textup{div};\Omega)}$  and $T\in \mathcal{T}_h$, it holds that (\textit{cf}.\ \cite[Thm.\ 16.4]{EG21})
     \begin{itemize}[noitemsep,topsep=2pt,leftmargin=!,labelwidth=\widthof{\quad(RT.3)},font=\itshape]
        \item[(RT.1)] \hypertarget{RT.1}{} $\|\Pi_h^{rt} y\|_T\leq c_{\textit{rt}}\,\{\|\Pi_h^{rt} y\|_T+h_T\,\|\nabla y\|_{T}\}\,$;\vspace{0.5mm}
        \item[(RT.2)] \hypertarget{RT.2}{} $\|y-\Pi_h^{rt} y\|_T\leq c_{\textit{rt}}\,h_T\,\|\nabla y\|_{T}\,$;\vspace{0.5mm}
        \item[(RT.3)] \hypertarget{RT.3}{} $\|\textrm{div}\,( y-\Pi_h^{rt} y)\|_T\leq c_{\textit{rt}}\,h_T\,\|\textrm{div}\, y\|_T\,$.
    \end{itemize}
	
	\subsubsection{Discrete integration-by-parts formula}\enlargethispage{5mm}
 
    	\qquad For every $v_h\in \mathcal{S}^{1,cr}(\mathcal{T}_h)$ and ${y_h\in \mathcal{R}T^0(\mathcal{T}_h)}$, there holds the \textit{discrete integration-by-parts
	formula}
	\begin{align}
	(\nabla_hv_h,\Pi_h y_h)_{\Omega}+(\Pi_h v_h,\,\textup{div}\,y_h)_{\Omega}=(v_h,y_h\cdot n)_{\partial\Omega}\,,\label{eq:pi}
	\end{align}
    which follows from the fact that  for every $y_h\in \mathcal{R}T^0(\mathcal{T}_h)$, it holds that $y_h|_T\cdot n_T=\textrm{const}$~on~$\partial T$ for all $T\in \mathcal{T}_h$ and	$\jump{y_h\cdot n}_S=0$ on $S$ for all $S\in \mathcal{S}_h^{i}$, and for every~${v_h\in \mathcal{S}^{1,cr}(\mathcal{T}_h)}$,~it~holds~that $\int_{S}{\jump{v_h}_S\,\textup{d}s}\hspace{-0.1em}=\hspace{-0.1em}\jump{v_h}_S(x_S)\hspace{-0.1em}=\hspace{-0.1em}0$ for all $S\hspace{-0.1em}\in\hspace{-0.1em} \mathcal{S}_h^{i}$.
	As a result, for~every~$v_h\hspace{-0.1em}\in\hspace{-0.1em} \smash{\mathcal{S}^{1,cr}_D(\mathcal{T}_h)}$~and~${y_h\hspace{-0.1em}\in \hspace{-0.1em}\smash{\mathcal{R}T^0_N(\mathcal{T}_h)}}$, \eqref{eq:pi} reads
	\begin{align}
		(\nabla_hv_h,\Pi_h y_h)_{\Omega}=-(\Pi_h v_h,\,\textup{div}\,y_h)_{\Omega}\,.\label{eq:pi0}
	\end{align}
	In \hspace{-0.1mm}\cite{CP20,Bar20,Bar21,BKAFEM22}, \hspace{-0.1mm}the \hspace{-0.1mm}discrete \hspace{-0.1mm}integration-by-parts \hspace{-0.1mm}formula \hspace{-0.1mm}\eqref{eq:pi0} \hspace{-0.1mm}formed \hspace{-0.1mm}a \hspace{-0.1mm}cornerstone~\hspace{-0.1mm}in~\hspace{-0.1mm}the~\hspace{-0.1mm}\mbox{derivation} of a discrete convex duality theory and, as such, also plays a central 
	role~in~the~\mbox{hereinafter}~\mbox{analysis}. For instance, for every $v\in H^1_D(\Omega)$ and $y\in (H^1(\Omega))^d\cap H^2_N(\textup{div};\Omega)$, \eqref{eq:pi0}  enables to exchange quasi-interpolation operators via (\textit{cf}.\ \cite[Lem.\ 2.1]{Bar21})
    \begin{align}
        (\textup{div}\,y,v-\Pi_h \Pi_h^{cr}v)_{\Omega}=-(\nabla v,y-\Pi_h \Pi_h^{rt}y)_{\Omega}\,.\label{eq:exchange}
    \end{align}
    In addition, there holds the \textit{discrete Helmholtz-Weyl decomposition} (\textit{cf}.\ \cite[Sec.\ 2.4]{BW21})
    \begin{align}
        (\mathcal{L}^0(\mathcal{T}_h))^d=\textup{ker}(\textup{div}|_{\smash{\mathcal{R}T^0_N(\mathcal{T}_h)}})\oplus \nabla_h(\smash{\mathcal{S}^{1,cr}_D(\mathcal{T}_h)})\,.\label{eq:decomposition}
    \end{align}

    \section{Obstacle problem}\label{sec:obstacle}

        \qquad In this section, we discuss the continuous and the discrete obstacle problem.\vspace{-2mm}
		
		\subsection{Continuous obstacle problem}\vspace{-1mm}
		
		\qquad \textit{Primal problem.} Given a  force $f\in L^2(\Omega)$ and  an obstacle $\chi\in H^1(\Omega)$~with~${\chi\leq 0 }$~a.e.~on~$\Gamma_D$, the (continuous) obstacle problem is defined via the minimization of ${I\colon H^1_D(\Omega)\to \mathbb{R}\cup\{+\infty\}}$, for every $v\in \smash{H^1_D(\Omega)}$ defined by 
		\begin{align}
			\smash{I(v)\coloneqq  \tfrac{1}{2}\| \nabla v\|_{\Omega}^2-(f,v)_{\Omega}+I_K(v)}\,,\label{eq:obstacle_primal}
		\end{align}
		where 
  \begin{align*}
      \smash{K\coloneqq \big\{v\in H^1_D(\Omega)\mid v\ge \chi\text{ a.e.\  in }\Omega\big\}}\,,
  \end{align*}
        and $I_K\colon \hspace{-0.15em}H^1_D(\Omega)\hspace{-0.15em}\to\hspace{-0.15em} \mathbb{R}\hspace{-0.1em}\cup\hspace{-0.1em}\{+\infty\}$ is given via $I_K(v)\hspace{-0.15em}\coloneqq\hspace{-0.15em}  0$ if $v\hspace{-0.15em}\in\hspace{-0.15em} K$ and ${I_K(v)\hspace{-0.15em}\coloneqq\hspace{-0.15em} +\infty}$~else.~In~what~follows, we refer to the minimization of the functional \eqref{eq:obstacle_primal} as the \textit{primal problem}. 
        Since the functional \eqref{eq:obstacle_primal} is proper, \hspace{-0.1mm}strictly \hspace{-0.1mm}convex, \hspace{-0.1mm}weakly \hspace{-0.1mm}coercive,~\hspace{-0.1mm}and~\hspace{-0.1mm}lower \hspace{-0.1mm}semi-\hspace{-0.1mm}continuous (\textit{cf}.~\hspace{-0.1mm}\cite[\hspace{-0.1mm}Thm.~\hspace{-0.1mm}5.1]{Bar15}),~\hspace{-0.1mm}the \hspace{-0.1mm}direct \hspace{-0.1mm}method \hspace{-0.1mm}in \hspace{-0.1mm}the calculus~\hspace{-0.1mm}of~\hspace{-0.1mm}variations~\hspace{-0.1mm}(\textit{cf}.~\hspace{-0.1mm}\cite{Dac08})~\hspace{-0.1mm}yields \hspace{-0.1mm}the \hspace{-0.1mm}existence \hspace{-0.1mm}of \hspace{-0.1mm}a~\hspace{-0.1mm}unique~\hspace{-0.1mm}minimizer~\hspace{-0.1mm}${u\!\in\! K}$, 
        called~\textit{primal~solution}.~In~what~follows, we reserve the notation $u\in K$ for the~primal~solution.  
        Since the functional \eqref{eq:obstacle_primal} is not Fr\'echet differentiable, the optimality conditions associated with the primal problem are not given via a variational equality. Instead, they~are~given~via~a~variational inequality. In fact,  $u\in K $ is minimal for \eqref{eq:obstacle_primal} (\textit{cf}.\ \cite[Thm.\ 5.2]{Bar15}) if and only if for every $ v\in K $, it holds that\vspace{-0.5mm}
        \begin{align}
            (\nabla u,\nabla u-\nabla v)_{\Omega}\leq (f,u-v)_{\Omega}\,.\label{eq:variational_ineq}
        \end{align}

        \textit{Dual problem.} Appealing to \cite[Sec.\ 2.4, p.\ 84 ff.]{ET99}, the  \textit{dual problem} to the obstacle problem is defined via the maximization of  $D\colon (L^2(\Omega))^d\to \mathbb{R}\cup\{-\infty\}$, for every $y\in (L^2(\Omega))^d$~defined~by
		\begin{align}
			\smash{D(y)\coloneqq -\tfrac{1}{2}\|y\|_{\Omega}^2-I_K^*(-\nabla^*y+F)}\,,\label{eq:obstacle_dual}
		\end{align}
		where $I_K^*\colon (H^1_D(\Omega))^*\to\mathbb{R}\cup\{+\infty\}$ is defined by $I_K^*(v^*)\hspace{-0.1em}\coloneqq\hspace{-0.1em} 0$ if
		$\langle v^*,v\rangle_{\smash{H^1_D(\Omega)}}\hspace{-0.1em}\le\hspace{-0.1em} 0$~for~\mbox{all}~${v\hspace{-0.1em}\in\hspace{-0.1em} K}$~and $\smash{I_K^*(v^*)\hspace{-0.1em}\coloneqq\hspace{-0.1em}  +\infty}$~else,
  $\nabla^*\colon (L^2(\Omega))^d\hspace{-0.1em}\to \hspace{-0.1em}(H^1_D(\Omega))^*$ is defined by $\langle \nabla^*\,y,v\rangle_{\Omega}\hspace{-0.1em}\coloneqq\hspace{-0.1em}(y,\nabla v)_{\Omega}$~for~all~$y\hspace{-0.15em}\in\hspace{-0.15em} (L^2(\Omega))^d$ and $v\hspace{-0.15em}\in\hspace{-0.15em} H^1_D(\Omega)$, and $F\hspace{-0.15em}\in \hspace{-0.15em} (H^1_D(\Omega))^*$ is defined by $\langle F,v\rangle_{\Omega}\hspace{-0.15em}\coloneqq \hspace{-0.15em}(f,v)_{\Omega}$~for~all~${v\hspace{-0.15em}\in\hspace{-0.15em} H^1_D(\Omega)}$.
        For every $y\in W^2_N(\textrm{div};\Omega)$, there holds the explicit representation
        \begin{align}
			\smash{D(y)\coloneqq -\tfrac{1}{2}\|y\|_{\Omega}^2-(\textrm{div}\,y+f,\chi)_{\Omega}-I_-(\textup{div}\,y+f)}\,,\label{eq:obstacle_dual_representation}
		\end{align}
        where $I_-\colon L^2(\Omega)\to\mathbb{R}\cup\{+\infty\}$ is defined by $I_-(g)\coloneqq 0$ if
		$g\in L^2(\Omega)$ with $g\leq 0$~a.e.~in~$\Omega$~and $\smash{I_-(g)\coloneqq  +\infty}$~else.
		Moreover,~in~\cite[Sec.\ 2.4, p.~84 ff.]{ET99},  it is shown that there exists a
        unique maximizer  $z\in (L^2(\Omega))^d$ of 
        \eqref{eq:obstacle_dual}, called~\textit{dual~solution}, and a \textit{strong duality relation}, \textit{i.e.}, 
        \begin{align}
            I(u) = D(z)\,,\label{eq:obstacle_strong_duality}
        \end{align}
        applies. In addition, there hold the \textit{convex optimality relations}\enlargethispage{8.5mm}
		\begin{align}
			z&=\nabla u\quad\text{ a.e.\ in }\Omega\,,\label{eq:obstacle_optimality.2}\\\langle-\nabla^*\,z+F,u\rangle_{\Omega}&=I_K^*(-\nabla^*\,z+F)\label{eq:obstacle_optimality.1}\,.
		\end{align}

        \textit{Augmented problem.} Due to \cite[Thm.\ 5.2]{Bar15}, there exists a Lagrange multiplier~${\Lambda\in (H^1_D(\Omega))^*}$ with $\Lambda\leq 0$ in $(H^1_D(\Omega))^*$, \textit{i.e.}, $\langle\Lambda,v\rangle_{\Omega}\leq 0 $ for all $v\in H^1_D(\Omega)$ with $v\ge 0$ for a.e.\  $\Omega$,~such~that for every $v\in H^1_D(\Omega)$, there holds the \textit{augmented problem}
        \begin{align}
            \smash{(\nabla u,\nabla v)_{\Omega}+\langle \Lambda,v\rangle_{\Omega}=(f,v)_{\Omega}}\,,\label{eq:augmented_problem}
        \end{align}
        \textit{i.e.},  \hspace{-0.2mm}$\Lambda\!=\!f-\nabla^*\,z$ \hspace{-0.2mm}in \hspace{-0.2mm}$(H^1_D(\Omega))^*$. 
        Then, \hspace{-0.2mm}(\textit{cf}.\ \cite[Thm.\ 5.2]{Bar15}), \hspace{-0.2mm}there \hspace{-0.2mm}holds \hspace{-0.1mm}the \hspace{-0.2mm}\textit{complementary~\hspace{-0.2mm}\mbox{condition}} 
        \begin{align}\label{eq:complementary_condition}
            \smash{\langle \Lambda,u\rangle_{\Omega}=I_K^*(\Lambda)\,.}
        \end{align}
        If there exists $\lambda\in L^2(\Omega)$ such that $\langle \Lambda,v\rangle_{\Omega}=(\lambda,v)_{\Omega}$ for all $v\in H^1_D(\Omega)$~(\textit{cf}.~\cite{KS00}),~then~\eqref{eq:complementary_condition}~reads 
        \begin{align}\label{eq:complementary_condition_ptw}
            \smash{\lambda (u-\chi)=0\quad\text{ a.e.\  in }\Omega\,.}
        \end{align}

		\subsection{Discrete obstacle problem}\label{subsec:discrete_obstacle_problem}\vspace{-1mm}
		
		\qquad \textit{Discrete primal problem.} Given a force $f\in L^2(\Omega)$ and an obstacle $\chi\in H^1(\Omega)$~such~that~${\chi\leq 0 }$ a.e.\ on $\Gamma_D$, with $f_h\hspace{-0.1em}\coloneqq \hspace{-0.1em}\Pi_h f\hspace{-0.1em}\in\hspace{-0.1em} \mathcal{L}^0(\mathcal{T}_h)$ and $\chi_h\hspace{-0.1em}\in\hspace{-0.1em} \mathcal{L}^0(\mathcal{T}_h)$ approximating $\chi$, the~discrete~obstacle~problem is defined via the minimization of  ${I_h^{cr}\colon \mathcal{S}^{1,cr}_D(\mathcal{T}_h)\to \mathbb{R}\cup\{+\infty\}}$, for every $v_h\in \mathcal{S}^{1,cr}_D(\mathcal{T}_h)$~defined by
		\begin{align}
			{I_h^{cr}(v_h)\coloneqq \tfrac{1}{2}\| \nabla_hv_h\|_{\Omega}^2-(f_h,\Pi_hv_h)_{\Omega}+I_{K_h^{cr}}(v_h)\,,}\label{eq:obstacle_discrete_primal}
		\end{align}
        where
  \begin{align*}
      \smash{K_h^{cr}\coloneqq \big\{v_h\in \mathcal{S}^{1,cr}_D(\mathcal{T}_h)\mid \Pi_h v_h\ge \chi_h\text{ a.e.\  in }\Omega\big\}}\,,
  \end{align*}
        and $I_{K_h^{cr}}\colon \hspace{-0.15em}\mathcal{S}^{1,cr}_D(\mathcal{T}_h)\hspace{-0.15em}\to\hspace{-0.15em} \mathbb{R}\hspace{-0.1em}\cup\hspace{-0.1em}\{+\infty\}$ is given via $I_{K_h^{cr}}(v_h) \coloneqq 0$ if $v_h\hspace{-0.15em}\in\hspace{-0.15em} K_h^{cr}$ and ${I_{K_h^{cr}}(v_h)\coloneqq +\infty}$~else.
        In \hspace{-0.1mm}what \hspace{-0.1mm}follows, \hspace{-0.1mm}we \hspace{-0.1mm}refer \hspace{-0.1mm}to \hspace{-0.1mm}the \hspace{-0.1mm}minimization \hspace{-0.1mm}of \hspace{-0.1mm}the \hspace{-0.1mm}functional \hspace{-0.1mm}\eqref{eq:obstacle_discrete_primal} \hspace{-0.1mm}as \hspace{-0.1mm}the \textit{\hspace{-0.1mm}discrete~\hspace{-0.1mm}primal~\hspace{-0.1mm}\mbox{problem}}. 
         Since the functional \eqref{eq:obstacle_discrete_primal} is proper, strictly convex, weakly coercive,~and~lower semi-continuous, 
        the direct method in the calculus~of~variations  yields the existence of a unique minimizer $u_h^{cr}\in K_h^{cr}$, called~\textit{discrete primal~solution}.~In~what~follows, we reserve the notation $u_h^{cr}\in K_h^{cr}$ for the discrete primal solution.  
        In addition, ${u_h^{cr}\in K_h^{cr}}$ is the unique minimizer of \eqref{eq:obstacle_discrete_primal}~if~and~only if for every $v_h\in K_h^{cr}$, it holds that
        \begin{align}
            {(\nabla_h u_h^{cr},\nabla_h u_h^{cr}-\nabla_h v_h)_{\Omega}\leq (f_h,\Pi_h u_h^{cr}-\Pi_h v_h)_{\Omega}\,.}\label{eq:discrete_variational_ineq}
        \end{align}
		\qquad \textit{Discrete dual problem.} According to \cite[Subsec.\ 4.1]{Bar21}, the \textit{discrete dual problem} to the  discrete obstacle problem is defined via the maximization of  $D_h^{rt}\colon \mathcal{R}T^0_N(\mathcal{T}_h)\to \mathbb{R}\cup\{-\infty\}$,~for~every $y_h\in \mathcal{R}T^0_N(\mathcal{T}_h)$ defined by 
		\begin{align}
			{D_h^{rt}(y_h)\coloneqq -\tfrac{1}{2}\| \Pi_hy_h\|_{\Omega}^2- (\textup{div}\,y_h+f_h,\chi_h)_{\Omega}-I_-(\textup{div}\,y_h+f_h)\,.}\label{eq:obstacle_discrete_dual}
		\end{align}
         \hphantom{}\qquad\textit{Discrete augmented problem.} The \textit{discrete augmented problem},  similar to the augmented~problem \eqref{eq:augmented_problem}, seeks for a
        \textit{discrete Lagrange multiplier} $\smash{\overline{\lambda}}_h^{cr}\in \Pi_h( \mathcal{S}^{1,cr}_D(\mathcal{T}_h))$~such~that~$\smash{\overline{\lambda}}_h^{cr}\leq 0$~a.e.~in~$\Omega$  and 
		for every $v_h\in  \smash{\mathcal{S}^{1,cr}_D(\mathcal{T}_h)}$, it holds that
		\begin{align}
					 {(\smash{\overline{\lambda}}_h^{cr},\Pi_hv_h)_{\Omega}=(f_h,\Pi_h v_h)_{\Omega}-(\nabla_h u_h^{cr},\nabla_h v_h)_{\Omega}\,.}\label{eq:obstacle_lagrange_multiplier_cr}
		\end{align}
        The following proposition establishes the well-posedness of the discrete augmented problem~\eqref{eq:obstacle_lagrange_multiplier_cr}.
        
        \begin{proposition}\label{prop:augmented} The following statements apply:
            \begin{itemize}[noitemsep,topsep=2pt,leftmargin=!,labelwidth=\widthof{(iii)},font=\itshape]
            \item[(i)] The discrete augmented problem is well-posed, \textit{i.e.},
        there exists a unique discrete Lagrange multiplier $\smash{\overline{\lambda}}_h^{cr}\in \Pi_h(\mathcal{S}^{1,cr}_D(\mathcal{T}_h))$ that satisfies \eqref{eq:obstacle_lagrange_multiplier_cr}.
             \item[(ii)] The discrete Lagrange multiplier $\smash{\overline{\lambda}}_h^{cr}\in \Pi_h( \mathcal{S}^{1,cr}_D(\mathcal{T}_h))$   satisfies $\smash{\overline{\lambda}}_h^{cr}\leq 0$ a.e.\  in $\Omega$ and the  \textit{discrete~complementarity~condition}
            \begin{align}
		   \smash{\overline{\lambda}}_h^{cr}(\Pi_h u_h^{cr}-\chi_h)=0\quad\text{ a.e.\ in }\Omega\,.\label{eq:discrete_complementary}
		\end{align}
            \end{itemize}
        \end{proposition}

        \begin{remark}
            The discrete complementarity condition \eqref{eq:discrete_complementary} is a discrete analogue of the (continuous) variational complementarity condition \eqref{eq:complementary_condition} and the (continuous) point-wise complementarity condition \eqref{eq:complementary_condition_ptw}, respectively.\enlargethispage{5mm}
        \end{remark}

        \begin{proof}[Proof (of Proposition \ref{prop:augmented}).]
            \textit{ad (i).}
            We relax the obstacle constraint via a penalization scheme, \textit{i.e.}, for every $\varepsilon>0$, we consider the minimization~of~${I_{h,\varepsilon}^{cr}\colon \mathcal{S}^{1,cr}_D(\mathcal{T}_h)\to \mathbb{R}}$, for every  $v_h\in \smash{\mathcal{S}^{1,cr}_D(\mathcal{T}_h)}$ defined by 
            \begin{align*}
                \smash{I_{h,\varepsilon}^{cr}(v_h)\coloneqq \tfrac{1}{2}\| \nabla_hv_h\|_{\Omega}^2-(f_h,\Pi_hv_h)_{\Omega}+\tfrac{\varepsilon^{-2}}{2}\|(\Pi_hv_h-\chi_h)_-\|_\Omega^2\,.}
            \end{align*}
            Since for every $\varepsilon>0$, $I_{h,\varepsilon}^{cr}\colon \smash{\mathcal{S}^{1,cr}_D(\mathcal{T}_h)}\to \mathbb{R}$ is continuous, strictly convex, and~weakly~\mbox{coercive},~the direct method in the calculus of variation yields the existence of a unique~minimizer~$\smash{
            u_{h,\varepsilon}^{cr}\hspace{-0.15em}\in\hspace{-0.15em}\smash{\mathcal{S}^{1,cr}_D(\mathcal{T}_h)}}$, which, for every $v_h\in \smash{\mathcal{S}^{1,cr}_D(\mathcal{T}_h)}$, abbreviating $\lambda_{h,\varepsilon}^{cr}\coloneqq \varepsilon^{-2}(\Pi_h u_{h,\varepsilon}^{cr}-\chi_h)_-\in \smash{\mathcal{L}^0(\mathcal{T}_h)}$,
            satisfies 
            \begin{align}
                {(\nabla_hu_{h,\varepsilon}^{cr},\nabla_h v_h)_{\Omega}+(\lambda_{h,\varepsilon}^{cr},\Pi_h v_h)_{\Omega}=(f_h,\Pi_h v_h)_{\Omega}\,.}\label{eq:augmented.1}
            \end{align}
            Due to the minimality~of~${
            u_{h,\varepsilon}^{cr}\hspace{-0.1em}\in\hspace{-0.1em}\smash{\mathcal{S}^{1,cr}_D(\mathcal{T}_h)}}$, we find that
             $I_{h,\varepsilon}^{cr}(u_{h,\varepsilon}^{cr})\leq I_{h,\varepsilon}^{cr}(u_h^{cr})$ and,~as~a~\mbox{consequence},\newpage \hspace{-5mm}using that $(\Pi_h u_h^{cr}-\chi_h)_-=0$ a.e.\  in $\Omega$, that
             \begin{align}
                 \smash{\tfrac{1}{2}\|\nabla_hu_{h,\varepsilon}^{cr}\|_\Omega^2+\tfrac{\varepsilon^2}{2}\|\lambda_{h,\varepsilon}^{cr}\|_\Omega^2\leq \tfrac{1}{2}\|\nabla_hu_h^{cr}\|_\Omega^2+(f_h,\Pi_h (u_{h,\varepsilon}^{cr}-u_h^{cr}))_{\Omega}\,.}\label{eq:augmented.2}
             \end{align}
             Using the $\kappa$-Young inequality $ab\leq \frac{1}{4\kappa}a^2+\kappa b^2$, valid for all $a,b\ge 0$ and $\kappa>0$,~(\hyperlink{L0.1}{L0.1}),~and~the~discrete Poincar\'e inequality \eqref{discrete_poincare}, for every $\varepsilon>0$, we find that
             \begin{align}\label{eq:augmented.2.2}
                \begin{aligned}\smash{\vert (f_h,\Pi_h u_{h,\varepsilon}^{cr})_{\Omega}\vert 
                 \leq \tfrac{1}{4\kappa}\,\|f_h\|_{\Omega}^2+\kappa\,(c^{cr}_{P})^2\,\|\nabla_h u_{h,\varepsilon}^{cr}\|_{\Omega}^2\,.}
                 \end{aligned}
             \end{align}
             Using \eqref{eq:augmented.2.2} for $\kappa=\smash{\tfrac{1}{4(c^{cr}_{P})^2}}>0$ in \eqref{eq:augmented.2}, for every $\varepsilon>0$, we arrive at
             \begin{align}
                 \smash{\tfrac{1}{4}\|\nabla_hu_{h,\varepsilon}^{cr}\|_\Omega^2+\varepsilon^2\|\lambda_{h,\varepsilon}^{cr}\|_\Omega^2\leq \tfrac{1}{2}\|\nabla_hu_h^{cr}\|_\Omega^2-(f_h,\Pi_h u_h^{cr})_{\Omega}+(c^{cr}_{P})^2\,\|f_h\|_{\Omega}^2
                 \,. }\label{eq:augmented.3}
             \end{align}
             Using \hspace{-0.1mm}the \hspace{-0.1mm}discrete \hspace{-0.1mm}Poincar\'e \hspace{-0.1mm}inequality \hspace{-0.1mm}\eqref{discrete_poincare} \hspace{-0.1mm}in \hspace{-0.1mm}\eqref{eq:augmented.3}, \hspace{-0.1mm}we \hspace{-0.1mm}find \hspace{-0.1mm}that \hspace{-0.1mm}$(u_{h,\varepsilon}^{cr})_{\varepsilon>0}\subseteq \smash{\mathcal{S}^{1,cr}_D(\mathcal{T}_h)}$~\hspace{-0.1mm}is~\hspace{-0.1mm}bounded.
             Hence, 
             owing to the finite dimensionality of 
              $\smash{\mathcal{S}^{1,cr}_D(\mathcal{T}_h)}$, we~deduce~the existence of 
             $\tilde{u}_h^{cr}\in\smash{\mathcal{S}^{1,cr}_D(\mathcal{T}_h)} $ such that, for a not re-labeled subsequence, it holds that
             \begin{align}
                    \smash{ u_{h,\varepsilon}^{cr}\to \tilde{u}_h^{cr}\quad\text{ in }\mathcal{S}^{1,cr}_D(\mathcal{T}_h)\quad(\varepsilon\to 0^+)\,.}\label{eq:augmented.4}
             \end{align}
             Let $E_h^{cr}\colon \mathcal{L}^0(\mathcal{T}_h)\to \smash{(\mathcal{S}^{1,cr}_D(\mathcal{T}_h))^*}$, for every $\mu_h\in \mathcal{L}^0(\mathcal{T}_h)$ and $v_h\in \mathcal{S}^{1,cr}_D(\mathcal{T}_h)$, be defined by
             \begin{align}
                \smash{ \langle E_h^{cr}\mu_h,v_h\rangle_{\smash{\mathcal{S}^{1,cr}_D(\mathcal{T}_h)}}\coloneqq (\mu_h,\Pi_hv_h)_\Omega\,.}\label{eq:augmented.4.1}
             \end{align}
            Then, from \eqref{eq:augmented.1}, also using (\hyperlink{L0.1}{L0.1}), for every $\varepsilon>0$, it follows that \begin{align}\label{eq:augmented.5}\begin{aligned}\|E_h^{cr}\lambda_{h,\varepsilon}^{cr}\|_{\smash{(\mathcal{S}^{1,cr}_D(\mathcal{T}_h))^*}}
            &=\sup_{\substack{v_h\in \mathcal{S}^{1,cr}_D(\mathcal{T}_h);\|v_h\|_{\Omega}+\|\nabla_hv_h\|_{\Omega}\leq 1}}{\big\{(f_h,\Pi_h v_h)_{\Omega}-(\nabla_hu_{h,\varepsilon}^{cr},\nabla_h v_h)_{\Omega}\big\}}
            \\&\leq \|f_h\|_{\Omega}+\|\nabla_hu_{h,\varepsilon}^{cr}\|_{\Omega}\,.
            \end{aligned}
            \end{align}
            Using \eqref{eq:augmented.3} in \eqref{eq:augmented.5}, we find that $(E_h^{cr}\lambda_{h,\varepsilon}^{cr})_{\varepsilon>0}\hspace{-0.1em}\subseteq \hspace{-0.1em}\smash{(\mathcal{S}^{1,cr}_D(\mathcal{T}_h))^*}$ is bounded.
             Thus,~due~to~the~\mbox{finite} dimensionality of 
              $\smash{\smash{(\mathcal{S}^{1,cr}_D(\mathcal{T}_h))^*}}$ and the closedness of the range $R(E_h^{cr})$,~there~exists~${\smash{\tilde{\lambda}_h}\hspace{-0.15em}\in\hspace{-0.15em} \mathcal{L}^0(\mathcal{T}_h)}$~with
             \begin{align}
                \smash{E_h^{cr}\lambda_{h,\varepsilon}^{cr}\to  E_h^{cr}\smash{\tilde{\lambda}_h}\quad\text{ in }\smash{(\mathcal{S}^{1,cr}_D(\mathcal{T}_h))^*}\quad(\varepsilon\to 0^+)\,.}\label{eq:augmented.6}
             \end{align}
             Next, using \eqref{eq:augmented.3} once more and that, by definition, $\smash{\lambda_{h,\varepsilon}^{cr}=\varepsilon^{-2}(\Pi_h u_{h,\varepsilon}^{cr}-\chi_h)_-}$,~we~deduce~that
             \begin{align*}
                 \smash{\|(\Pi_h u_{h,\varepsilon}^{cr}-\chi_h)_-\|_{\Omega}^2=\varepsilon^4\,\|\lambda_{h,\varepsilon}^{cr}\|_{\Omega}^2\to 0\quad(\varepsilon\to 0^+)\,.}
             \end{align*}
             Because, on the other hand, due to \eqref{eq:augmented.4},    $\smash{(\Pi_h u_{h,\varepsilon}^{cr}-\chi_h)_-\to (\Pi_h \tilde{u}_h^{cr}-\chi_h)_-}$ in $L^2(\Omega)$ $(\varepsilon\to 0^+)$,\enlargethispage{6mm}
             we conclude that 
             $(\Pi_h \tilde{u}_h^{cr}-\chi_h )_-=0$ a.e.\  in $\Omega$.
             In other words, we have that
             \begin{align}\label{eq:augmented.7.2}
                \smash{\tilde{u}_h^{cr}\in K_h^{cr}}\,. 
             \end{align}
             As a consequence of \eqref{eq:augmented.7.2}, for every $\varepsilon>0$ and $v_h\in K_h^{cr}$, resorting to  \eqref{eq:augmented.4} and the minimality~of $u_{h,\varepsilon}^{cr}\in \mathcal{S}^{1,cr}_D(\mathcal{T}_h)$ for $I_{h,\varepsilon}^{cr}\colon \mathcal{S}^{1,cr}_D(\mathcal{T}_h)\to \mathbb{R}$, we find that
             \begin{align*}
             \begin{aligned}
                 \smash{I_h^{cr}(\tilde{u}_h^{cr})=\lim_{\varepsilon\to 0^+}{I_h^{cr}(u_{h,\varepsilon}^{cr})}
                 \leq \lim_{\varepsilon\to 0^+}{I_{h,\varepsilon}^{cr}(u_{h,\varepsilon}^{cr})}
                 \leq \lim_{\varepsilon\to 0^+}{I_{h,\varepsilon}^{cr}(v_h)}
                 = I_h^{cr}(v_h)\,.}
                 \end{aligned}
             \end{align*}
             Hence, due to the uniqueness of $u_h^{cr}\in K_h^{cr}$ as a minimizer of $I_h^{cr}\colon \mathcal{S}^{1,cr}_D(\mathcal{T}_h)\to \mathbb{R}$, we~infer~that $\tilde{u}_h^{cr}\hspace{-0.1em}=\hspace{-0.1em}u_h^{cr}$ in $\mathcal{S}^{1,cr}_D(\mathcal{T}_h)$. By passing for $\varepsilon\hspace{-0.1em}\to\hspace{-0.1em} 0^+$ in \eqref{eq:augmented.1}, for every $v_h\hspace{-0.1em}\in \hspace{-0.1em}\mathcal{S}^{1,cr}_D(\mathcal{T}_h)$, using~\eqref{eq:augmented.4}, \eqref{eq:augmented.6}, and the definition of $E_h^{cr}\colon \mathcal{L}^0(\mathcal{T}_h)\to \smash{(\mathcal{S}^{1,cr}_D(\mathcal{T}_h))^*}$ (\textit{cf}.\ \eqref{eq:augmented.4.1}), we conclude that
             \begin{align}\label{eq:augmented.7.3}
    \smash{(\nabla_hu_h^{cr},\nabla_h v_h)_{\Omega}+(\smash{\tilde{\lambda}_h},\Pi_h v_h)_{\Omega}=(f_h,\Pi_h v_h)_{\Omega}\,.}
             \end{align}
             Due to $\mathcal{L}^0(\mathcal{T}_h)=\Pi_h(\mathcal{S}^{1,cr}_D(\mathcal{T}_h))\bigoplus \Pi_h(\mathcal{S}^{1,cr}_D(\mathcal{T}_h))^{\perp_{L^2}}$, there exist unique $\smash{\overline{\lambda}}_h^{cr}\in \Pi_h(\mathcal{S}^{1,cr}_D(\mathcal{T}_h))$ and $\lambda_h\in \Pi_h(\mathcal{S}^{1,cr}_D(\mathcal{T}_h))^{\perp}$ such that $\smash{\tilde{\lambda}_h}=\smash{\overline{\lambda}}_h^{cr}+\lambda_h$ in $\mathcal{L}^0(\mathcal{T}_h)$. By the aid of the latter decomposition, for every $v_h\in\mathcal{S}^{1,cr}_D(\mathcal{T}_h)$, we conclude from  \eqref{eq:augmented.7.3} that 
                          \begin{align*}
  \smash{  (\nabla_hu_h^{cr},\nabla_h v_h)_{\Omega}+(\smash{\overline{\lambda}}_h^{cr},\Pi_h v_h)_{\Omega}=(f_h,\Pi_h v_h)_{\Omega}\,.}
             \end{align*}
             Next, let $\smash{\smash{\overline{\mu}}_h^{cr}\in \Pi_h(\mathcal{S}^{1,cr}_D(\mathcal{T}_h))}$ be such that for every $v_h\in\smash{\mathcal{S}^{1,cr}_D(\mathcal{T}_h)}$, it holds that
             \begin{align*}
  \smash{  (\nabla_hu_h^{cr},\nabla_h v_h)_{\Omega}+(\smash{\overline{\mu}}_h^{cr},\Pi_h v_h)_{\Omega}=(f_h,\Pi_h v_h)_{\Omega}\,.}
             \end{align*}
             Then,  $\smash{\smash{\overline{\lambda}}_h^{cr}-\smash{\overline{\mu}}_h^{cr}\in \Pi_h(\mathcal{S}^{1,cr}_D(\mathcal{T}_h))\cap \Pi_h(\mathcal{S}^{1,cr}_D(\mathcal{T}_h))^{\perp_{L^2}}=\{0\}}$ and, thus, $\smash{\smash{\overline{\lambda}}_h^{cr}=\smash{\overline{\mu}}_h^{cr}}$~in~$\smash{\Pi_h(\mathcal{S}^{1,cr}_D(\mathcal{T}_h))}$.

             \textit{ad (ii).} Let $\mathcal{T}_h^{cr}\hspace{-0.1em}\subseteq\hspace{-0.1em} \mathcal{T}_h$ be such that $\textup{span}(\{\chi_T\mid T\hspace{-0.1em}\in\hspace{-0.1em} \mathcal{T}_h^{cr}\})\hspace{-0.1em}=\hspace{-0.1em}\Pi_h(\mathcal{S}^{1,cr}_D(\mathcal{T}_h))$.~For~each~${T\hspace{-0.1em}\in\hspace{-0.1em} \mathcal{T}_h^{cr}}$, there exists $v_h^T\in \mathcal{S}^{1,cr}_D(\mathcal{T}_h)$ such that $\Pi_hv_h^T=\chi_T$. Next, let
             $\alpha_T\in \mathbb{R}$ be such that ${\Pi_h u_h^{cr} +\alpha_T\Pi_h v_h^T}$ $=\Pi_h u_h^{cr} +\alpha_T\chi_T\ge  \chi_h$~a.e.~in~$\Omega$.
             Then, for 
             $v_h=\alpha_T v_h^T\in \mathcal{S}^{1,cr}_D(\mathcal{T}_h)$~in~\eqref{eq:obstacle_lagrange_multiplier_cr},~in~particular, using \eqref{eq:discrete_variational_ineq} for $v_h=u_h^{cr}+\alpha_T v_h^T\in \mathcal{S}^{1,cr}_D(\mathcal{T}_h)$, we deduce that
             \begin{align}
                 \alpha_T\,\vert T\vert\, \smash{\overline{\lambda}}_h^{cr}=\alpha_T\,\big\{(f_h,\Pi_hv_h^T)_{\Omega}-(\nabla_h u_h^{cr},\nabla_h v_h^T)_{\Omega}\big\}\leq 0\quad\text{ a.e.\  in } T\,.\label{eq:augmented.10}
             \end{align}
             As $T\in \mathcal{T}_h^{cr}$ was arbitrary and $\smash{\overline{\lambda}}_h^{cr}\in\Pi_h(\mathcal{S}^{1,cr}_D(\mathcal{T}_h))$, we conclude from \eqref{eq:augmented.10} that $\smash{\overline{\lambda}}_h^{cr}\leq 0$~a.e.~in~$\Omega$. Eventually, for $T\in \mathcal{T}_h^{cr}$ such that $\Pi_h u_h^{cr} >  \chi_h$ 
             a.e.\  in $T$, there exists some $\alpha_T<0$ such that $\Pi_h u_h^{cr} +\alpha_T\Pi_h v_h^T=\Pi_h u_h^{cr} +\alpha_T\chi_T\ge  \chi_h$~a.e.~in~$\Omega$. For this $\alpha_T<0$ in \eqref{eq:augmented.10}, also using that $\smash{\overline{\lambda}}_h^{cr}\leq 0$ a.e.\  in $\Omega$, we arrive at
             \begin{align*}
                 0\leq\alpha_T\,\vert T\vert \,\smash{\overline{\lambda}}_h^{cr}=\alpha_T\,\big\{(f_h,\Pi_hv_h^T)_{\Omega}-(\nabla_h u_h^{cr},\nabla_h v_h^T)_{\Omega}\big\}\leq 0\quad\text{ a.e.\ in } T\,,
             \end{align*}
             so that $\smash{\overline{\lambda}}_h^{cr}=0$ a.e.\  in $T$.~In~other~words, the discrete complementarity condition \eqref{eq:discrete_complementary} applies.
        \end{proof}

		Given a discrete Lagrange multiplier $\smash{\overline{\lambda}}_h^{cr}\in \mathcal{L}^0(\mathcal{T}_h)$ satisfying \eqref{eq:obstacle_lagrange_multiplier_cr}, we define~the~\textit{discrete~flux} 
        \begin{align}\label{eq:generalized_marini}
			z_h^{rt}\coloneqq \nabla_hu_h^{cr}+\frac{\smash{\overline{\lambda}}_h^{cr}-f_h}{d}(\textup{id}_{\mathbb{R}^d}-\Pi_h\textup{id}_{\mathbb{R}^d})\in (\mathcal{L}^1(\mathcal{T}_h))^d\,,
		\end{align}
        which, by definition, satisfies 
        \begin{align}
            \Pi_h z_h^{rt}=\nabla_hu_h^{cr}\quad\text{ a.e.\ in }\Omega\,.\label{eq:discrete_optimality.1}
        \end{align}
        The following proposition proves that the discrete flux is admissible in the discrete dual problem and even a discrete dual solution.

        \begin{proposition}\label{prop:discrete_strong}  The following statements apply:
        \begin{itemize}[noitemsep,topsep=2pt,leftmargin=!,labelwidth=\widthof{(iii)},font=\itshape]
            \item[(i)] The discrete flux $z_h^{rt}\in (\mathcal{L}^1(\mathcal{T}_h))^d$ satisfies $z_h^{rt}\in \mathcal{R}T^0_N(\mathcal{T}_h)$ and
             \begin{align}\textup{div}\,z_h^{rt}=\smash{\overline{\lambda}}_h^{cr}-f_h\quad\text{ a.e.\ in }\Omega\,.\label{eq:discrete_optimality.2}
		      \end{align}
                In particular, it holds that $\textup{div}\,z_h^{rt}+f_h\leq 0$ a.e.\  in $\Omega$, \textit{i.e.}, $I_-(\textup{div}\,z_h^{rt}+f_h)=0$.
              \item[(ii)] The discrete flux $z_h^{rt}\in \mathcal{R}T^0_N(\mathcal{T}_h)$ is a maximizer of \eqref{eq:obstacle_discrete_dual} 
              and~discrete strong duality, \textit{i.e.}, $I_h^{cr}(u_h^{cr})=D_h^{rt}(z_h^{rt})$,~applies. 
        In addition,~there~holds~the~discrete~complementary condition
		 \begin{align}
        \qquad(\textup{div}\,z_h^{rt}+f_h)\,(\Pi_h u_h^{cr}-\chi_h)=0\quad\text{ a.e.\ in }\Omega\,.\label{eq:discrete_optimality.3}
		 \end{align}
            \end{itemize}
        \end{proposition} 

        \begin{proof}\let\qed\relax
            \textit{ad (i).} Since, due to $\vert \Gamma_D\vert>0$, $\textup{div}\colon  \mathcal{R}T^0_N(\mathcal{T}_h)\to \mathcal{L}^0(\mathcal{T}_h)$ 
            is surjective, there exists some $y_h\in \mathcal{R}T^0_N(\mathcal{T}_h)$ such that $\textup{div}\,y_h=\smash{\overline{\lambda}}_h^{cr}-f_h$ in $\mathcal{L}^0(\mathcal{T}_h)$. 
            Then, using the discrete integration-by-parts formula \eqref{eq:pi0} and \eqref{eq:obstacle_lagrange_multiplier_cr}, for every $v_h\in \mathcal{S}^{1,cr}_D(\mathcal{T}_h)$, we find that\enlargethispage{4mm}
            \begin{align}\label{eq:discrete_strong.1}
                \begin{aligned}
                (\Pi_hy_h,\nabla_h v_h)_{\Omega}&=-(\textup{div}\,y_h,\Pi_h v_h)_{\Omega}=(f_h-\smash{\overline{\lambda}}_h^{cr},\Pi_h v_h)_{\Omega}=(\nabla_hu_h^{cr},\nabla_h v_h)_{\Omega}\,.
                \end{aligned}
            \end{align}
            Using \eqref{eq:discrete_optimality.1} in \eqref{eq:discrete_strong.1}, for every $v_h\in \mathcal{S}^{1,cr}_D(\mathcal{T}_h)$, we arrive at
            \begin{align}\label{eq:discrete_strong.2}
                (y_h-z_h^{rt},\nabla_h v_h)_{\Omega}=(\Pi_h(y_h-z_h^{rt}),\nabla_h v_h)_{\Omega}=0\,.
            \end{align}
            On the other hand, we have that $\textup{div}(y_h-z_h^{rt})\hspace{-0.1em}=\hspace{-0.1em}0$ a.e.\ in $T$ for all $T\hspace{-0.1em}\in\hspace{-0.1em} \mathcal{T}_h$,~so~that~${y_h\hspace{-0.1em}-\hspace{-0.1em}z_h^{rt}\hspace{-0.1em}\in\hspace{-0.1em} (\mathcal{L}^0(\mathcal{T}_h))^d}$.
            By \eqref{eq:discrete_strong.2} and the orthogonal decomposition \eqref{eq:decomposition},
            we conclude that ${y_h-z_h^{rt}\in \nabla_h(\mathcal{S}^{1,cr}_D(\mathcal{T}_h))^{\perp_{L^2}}} $ $\subseteq\mathcal{R}T^0_N(\mathcal{T}_h)$ and, thus, $z_h^{rt}\in \mathcal{R}T^0_N(\mathcal{T}_h)$, since already $y_h\in \mathcal{R}T^0_N(\mathcal{T}_h)$.

            \textit{ad (ii).} The discrete optimality relation \eqref{eq:discrete_optimality.3} follows from \eqref{eq:discrete_optimality.2} and \eqref{eq:discrete_complementary}. In consequence, it remains to establish the strong duality relation. Using \eqref{eq:discrete_optimality.3}, the discrete integration-by-parts formula \eqref{eq:pi0},  \eqref{eq:discrete_optimality.1}, and $I_-(\textup{div}\,z_h^{rt}+f_h)=0$, we observe that
            \begin{align*}
                I_h^{cr}(u_h^{cr})&=\tfrac{1}{2}\|\nabla_hu_h^{cr}\|_\Omega^2-(f_h,\Pi_h u_h^{cr})_{\Omega}
                \\&=\tfrac{1}{2}\|\Pi_h z_h^{rt}\|_\Omega^2+(\textup{div}\,z_h^{rt},\Pi_h u_h^{cr})_{\Omega}-(\textup{div}\,z_h^{rt}+f_h,\chi_h)_{\Omega}
                 \\&=\tfrac{1}{2}\|\Pi_h z_h^{rt}\|_\Omega^2-
                 (\Pi_hz_h^{rt},\nabla_hu_h^{cr})_\Omega
                 -(\textup{div}\,z_h^{rt}+f_h,\chi_h)_{\Omega}
                 \\&=-\tfrac{1}{2}\|\Pi_h z_h^{rt}\|_\Omega^2-
                 (\textup{div}\,z_h^{rt}+f_h,\chi_h)_\Omega= D_h^{rt}(z_h^{rt})\,.\tag*{$\qedsymbol$}
            \end{align*}
        \end{proof}

    \section{\textit{A priori} error analysis}\label{sec:apriori}

    \qquad In this section, we establish  \textit{a priori} error estimates for the discrete primal problem.\vspace{-0.25mm}\enlargethispage{9mm}
    
    \begin{theorem}\label{thm:apriori}
        If \hspace{-0.1mm}$u,\chi\hspace{-0.15em}\in\hspace{-0.15em} H^2(\Omega)$,  \textit{i.e.},  $z\hspace{-0.15em}\in\hspace{-0.15em} (H^1(\Omega))^d$ \hspace{-0.1mm}and \hspace{-0.1mm}$\lambda \hspace{-0.15em}\coloneqq \hspace{-0.15em} f\hspace{-0.1em}+\hspace{-0.1em}\textup{div}\,z\hspace{-0.15em}\in\hspace{-0.15em} L^2(\Omega)$,~\hspace{-0.1mm}and~\hspace{-0.1mm}${\chi_h\hspace{-0.15em}\coloneqq\hspace{-0.15em} \Pi_h \Pi_h^{cr}\chi} $ $\in \mathcal{L}^0(\mathcal{T}_h)$, then there exists a constant 
        $ c>0$, depending only on the chunkiness~${\omega_0>0}$,~such~that
        \begin{align*}
            \smash{\|\nabla_h  \Pi_h^{cr}u-\nabla_h u_h^{cr}\|_{\Omega}^2}&+\smash{(-\smash{\overline{\lambda}}_h^{cr},\Pi_h(\Pi_h^{cr}u-u_h^{cr}))_{\Omega}+\|\Pi_h \Pi_h^{rt}z- \Pi_h z_h^{rt}\|_{\Omega}^2}\\&\leq \smash{c\,h_{\max}^2\,\big\{\|D^2 u\|_{\Omega}^2+\|D^2 \chi\|_{\Omega}^2+\|\lambda\|_{\Omega}^2\big\}\,.}
        \end{align*}
    \end{theorem}


    \begin{proof}
        Using  that, owing to the discrete augmented problem \eqref{eq:obstacle_lagrange_multiplier_cr}, $\frac{1}{2}a^2-\frac{1}{2}b^2=\frac{1}{2}(a-b)^2+b(a-b)$ for all $a,b\in \mathbb{R}$, and the strong concavity of \eqref{eq:obstacle_discrete_dual}, 
        for every $v_h\in K_h^{cr}$ and $y_h\in\mathcal{R}T^0_N(\mathcal{T}_h)$,~it~holds\vspace{-0.25mm}
        \begin{align*}
            \smash{\tfrac{1}{2}\|\nabla_h  v_h-\nabla_h u_h^{cr}\|_{\Omega}^2+(-\smash{\overline{\lambda}}_h^{cr},\Pi_h(v_h-u_h^{cr}))_{\Omega}}&=\smash{ I_h^{cr}(v_h)-I_h^{cr}(u_h^{cr})\,,}\\
            \smash{\tfrac{1}{2}\|\Pi_h y_h- \Pi_h z_h^{rt}\|_{\Omega}^2}&\leq \smash{D_h^{rt}(z_h^{rt}) -D_h^{rt}(y_h)\,,}
        \end{align*}
        that $\Pi_h^{cr}u\in K_h^{cr}$,  as $\fint_S{(u-\chi)\,\mathrm{d}s}\ge 0$ and $\varphi_S(x_T)=\frac{1}{d+1}$ for all ${T\in \mathcal{T}_h}$ and  $S\in \mathcal{S}_h$~with~${S\subseteq\partial T}$, 
        that $\Pi_h^{rt}z\in \mathcal{R}T^0_N(\mathcal{T}_h)$ with $\textup{div}\,\Pi_h^{rt}z+f_h\hspace{-0.1em}=\hspace{-0.1em}\Pi_h(\textup{div}\,z+f)\hspace{-0.1em}=\hspace{-0.1em}\Pi_h\lambda\leq 0$ a.e.\  in $\Omega$, the discrete strong duality relation $I_h^{cr}(u_h^{cr})=D_h^{rt}(z_h^{rt})$ (\textit{cf}.\ Proposition \ref{prop:discrete_strong}(ii)),  $\nabla_h\Pi_h^{cr}u=\Pi_h \nabla u$ a.e.\ in $\Omega$, (\hyperlink{L0.1}{L0.1}), and 
        the strong duality relation ${I(u)=D(z)}$ (\textit{cf}.~\eqref{eq:obstacle_strong_duality}), we~find~that\vspace{-0.25mm}
        \begin{align}\label{eq:apriori.1}
            \begin{aligned}
                \tfrac{1}{2}\|\nabla_h  \Pi_h^{cr}u-\nabla_h u_h^{cr}\|_{\Omega}^2&+(-\smash{\overline{\lambda}}_h^{cr},\Pi_h(\Pi_h^{cr}u-u_h^{cr}))_{\Omega}
            +\tfrac{1}{2}\|\Pi_h \Pi_h^{rt}z- \Pi_h z_h^{rt}\|_{\Omega}^2 
               \\&\leq I_h^{cr}(\Pi_h^{cr}u)
                -D_h^{rt}(\Pi_h^{rt}z)
                \\&\leq I(u)
                +(f,u-\Pi_h \Pi_h^{cr}u)_\Omega-D_h^{rt}(\Pi_h^{rt}z)
                \\&= -\tfrac{1}{2}\|z\|_\Omega^2+(f,u-\Pi_h \Pi_h^{cr}u)_\Omega+\tfrac{1}{2}\|\Pi_h \Pi_h^{rt} z\|_\Omega^2\\&\quad
                -(\textrm{div}\,z+f,\chi)_\Omega+(\textrm{div}\,\Pi_h^{rt} z+f_h,\Pi_h \Pi_h^{cr}\chi)_\Omega
                \,.
            \end{aligned}
        \end{align}
       Next, using in \eqref{eq:apriori.1} the exchange of quasi-interpolation operators \eqref{eq:exchange} and $z=\nabla u$~(\textit{cf}.~\eqref{eq:obstacle_optimality.2}),~\textit{i.e.},\vspace{-0.25mm}
        \begin{align*}
            \smash{(\textrm{div}\,z,u-\Pi_h \Pi_h^{cr}u)_\Omega=-(z,z-\Pi_h \Pi_h^{rt} z)_\Omega=-\|z\|_{\Omega}^2+(z,\Pi_h \Pi_h^{rt} z)_{\Omega}\,,}
        \end{align*}
        $\textup{div}\,\Pi_h^{rt}z\hspace{-0.1em}+\hspace{-0.1em}f_h\hspace{-0.1em}=\hspace{-0.1em}\Pi_h\lambda$ in $\mathcal{L}^0(\mathcal{T}_h)$, and  $\textrm{div}\,z\hspace{-0.1em}+\hspace{-0.1em}f\hspace{-0.1em}=\hspace{-0.1em}\lambda$ a.e.\ in $\Omega$, abbreviating $\tilde{u} \hspace{-0.1em}\coloneqq \hspace{-0.1em}u-\chi\hspace{-0.1em}\in \hspace{-0.1em}H^2(\Omega)$,~we~get\vspace{-0.25mm}
        \begin{align}\label{eq:apriori.2}
            \begin{aligned}
                \tfrac{1}{2}\|\nabla_h  \Pi_h^{cr}u-\nabla_h u_h^{cr}\|_{\Omega}^2&+(-\smash{\overline{\lambda}}_h^{cr},\Pi_h(\Pi_h^{cr}u-u_h^{cr}))_{\Omega} +\tfrac{1}{2}\|\Pi_h \Pi_h^{rt}z- \Pi_h z_h^{rt}\|_{\Omega}^2\\&\leq 
                (\lambda,\tilde{u}-\Pi_h \Pi_h^{cr}\tilde{u})_\Omega
                +\tfrac{1}{2}\|z\|_\Omega^2-(z,\Pi_h \Pi_h^{rt} z)_{\Omega}+\tfrac{1}{2}\|\Pi_h \Pi_h^{rt} z\|_\Omega^2
                \\&=(\lambda,\tilde{u}-\Pi_h^{cr}\tilde{u})_{\Omega}+(\lambda,\Pi_h^{cr}\tilde{u}-\Pi_h \Pi_h^{cr}\tilde{u})_{\Omega} +\tfrac{1}{2}\|z-\Pi_h\Pi_h^{rt} z\|_\Omega^2
                \\&\eqqcolon I_h^1+I_h^2+I_h^3\,.
            \end{aligned}
        \end{align}
        As a consequence, it remains to estimate the terms $I_h^1$, $I_h^2$ and $I_h^3$:

        \textit{ad $I_h^1$.} Using (\hyperlink{CR.3}{CR.3}), we obtain\vspace{-0.25mm}
        \begin{align}\label{eq:apriori.3}
            \begin{aligned}
            \smash{I_h^1 
            \leq   \| \lambda\|_{\Omega}\|\tilde{u}-\Pi_h^{cr}\tilde{u}\|_{\Omega} 
                \leq   c_{\textrm{cr}}\,h_{\max}^2\,\| \lambda\|_{\Omega}\,\|D^2 \tilde{u}\|_{\Omega}\,.}
            \end{aligned}
        \end{align}

        \textit{ad $I_h^2$.} Using that
        $\Pi_h^{cr}\tilde{u}-\Pi_h \Pi_h^{cr}\tilde{u}=\nabla_h \Pi_h^{cr}\tilde{u}\cdot (\textup{id}_{\mathbb{R}^d}-\Pi_h\textup{id}_{\mathbb{R}^d})$ a.e.\ in $\Omega$, $\lambda=0$ a.e.\ in $\{\tilde{u}>0\}$, $\nabla \tilde{u}=0$ a.e.\ in $\{\tilde{u}=0\}$,  and (\hyperlink{CR.3}{CR.3}), we obtain\vspace{-0.25mm}
        \begin{align}\label{eq:apriori.3.1}
            \begin{aligned}
            \smash{I_h^2 
            \leq   ( \lambda,(\nabla_h \Pi_h^{cr}\tilde{u}-\nabla \tilde{u}) \cdot (\textup{id}_{\mathbb{R}^d}-\Pi_h\textup{id}_{\mathbb{R}^d}))_{\Omega} 
                \leq   c_{\textrm{cr}}\,h_{\max}^2\,\| \lambda\|_{\Omega}\,\|D^2\tilde{u}\|_{\Omega}\,.}
            \end{aligned}
        \end{align}

        \textit{ad $I_h^3$.} Using (\hyperlink{L0.1}{L0.1}), (\hyperlink{L0.2}{L0.2}), and (\hyperlink{RT.2}{RT.2}), we obtain\vspace{-0.25mm}
        \begin{align}\label{eq:apriori.4}
            \begin{aligned}
           \smash{ I_h^3 \leq \|z-\Pi_h z\|_\Omega^2+\|\Pi_h(z-\Pi_h^{rt} z)\|_\Omega^2
                \leq   (c_{\Pi}^2+c_{\textit{rt}}^2)\,h_{\max}^2\,\|\nabla z\|_{\Omega}^2\,.}
            \end{aligned}
        \end{align}
        Combining \eqref{eq:apriori.2}--\eqref{eq:apriori.4}, we arrive at the claimed \textit{a priori} error estimate.
    \end{proof}

    \begin{corollary}
    \label{cor:apriori}
        If $u,\chi\in H^2(\Omega)$, then there exists a constant~${c>0}$, depending only on the chunkiness $\omega_0> 0$, such that\vspace{-1mm}
        \begin{align*}
            \smash{\|\nabla_h u_h^{cr}-\nabla u\|_{\Omega}^2\leq c\,h_{\max}^2\,\big\{\|D^2 u\|_{\Omega}^2+\|D^2 \chi\|_{\Omega}^2+\|\lambda\|_{\Omega}^2\big\}\,.}
        \end{align*}
    \end{corollary}

    \begin{proof}\let\qed\relax
        \!Resorting \hspace{-0.1mm}to \hspace{-0.1mm}(\hyperlink{CR.2}{CR.2}), \hspace{-0.1mm}the \hspace{-0.1mm}assertion \hspace{-0.1mm}follows \hspace{-0.1mm}from \hspace{-0.1mm}Theorem \hspace{-0.1mm}\ref{thm:apriori}, \hspace{-0.1mm}exploiting \hspace{-0.1mm}that~\hspace{-0.1mm}for~\hspace{-0.1mm}all~\hspace{-0.1mm}${v_h\!\in\! K_h^{cr}}$
        \begin{align*}
            \smash{(-\smash{\overline{\lambda}}_h^{cr},\Pi_h(v_h-u_h^{cr}))_{\Omega}=
            (\nabla_h u_h^{cr}, \nabla_h  v_h-\nabla_h u_h^{cr})_{\Omega}-(f_h,\Pi_h(v_h-\nabla_h u_h^{cr}))_{\Omega}\ge 0}\,.\tag*{$\qedsymbol$}
        \end{align*}
    \end{proof}

    \section{A posteriori error analysis}\label{sec:aposteriori}
 
	 \qquad In this section,  we examine the
        \textit{primal-dual \textit{a posteriori} error estimator} $\eta_h^2\colon H^1_D(\Omega)\to \mathbb{R}$, for every $v\in H^1_D(\Omega)$ defined by 
	 \begin{align}\label{eq:primal-dual.1}
    \begin{aligned}
	     \eta_h^2(v)&\coloneqq \eta_{A,h}^2(v)+\eta_{B,h}^2(v)+\eta_{C,h}^2\,,\\
        \eta_{A,h}^2(v)&\coloneqq \|\nabla v-\nabla_h u_h^{cr}\|_{\Omega}^2\,,\\
        \eta_{B,h}^2(v)&\coloneqq (-\smash{\overline{\lambda}}_h^{cr},\Pi_h(v-\chi))_{\Omega}\,,\\
         \eta_{C,h}^2&\coloneqq \tfrac{1}{d^2}\|h_{\mathcal{T}}(f_h-\smash{\overline{\lambda}}_h^{cr})\|_{\Omega}^2\,,
      \end{aligned}
	 \end{align}
    for  reliability and efficiency. 

    \begin{remark}
        \begin{itemize}[noitemsep,topsep=2pt,leftmargin=!,labelwidth=\widthof{(iii)},font=\itshape]
            \item[(i)] The estimator $\eta_h^2$ appeared in a similar form in \cite{CK17} in a Crouzeix--Raviart approximation of the obstacle problems, imposing the obstacle constraint~at~the~barycenters of element sides. However, imposing the obstacle constraint at the barycenters of elements leads to a simplified form compared to the estimator in \cite{CK17}.
             \item[(ii)] The estimator $\smash{\eta_{A,h}^2}$
             provides control over the flux relation \eqref{eq:obstacle_optimality.2}.
             \item[(iii)] The estimator $\smash{\eta_{B,h}^2}$  measures the discrepancy in the complementary condition~\eqref{eq:complementary_condition}~(\textit{cf}.~\cite{BHS08}).
             \item[(iv)]  The estimator $\eta_{C,h}^2$ measures the irregularity of the  dual solution, 
             \textit{i.e.},~$\textup{div}\,z\notin L^2(\Omega)$.\vspace{-1mm}\enlargethispage{6mm}
              \end{itemize}
    \end{remark}

        \subsection{Reliability}\label{subsec:reliability}\vspace{-1mm}

        \hspace{5mm}In this subsection, we identify error quantities that are controlled by the \textit{a posteriori} error estimator $\eta_h^2\colon H^1_D(\Omega)\hspace{-0.1em}\to\hspace{-0.1em} \mathbb{R}$, \textit{cf}.\ \eqref{eq:primal-dual.1}.~In~doing~so, we combine two different~but~related~approaches: first, we resort  to first-order relations based on (discrete) convex duality, leading to constant-free estimates; second, we resort to second-order relations based on the (discrete) augmented problems, leading to estimates for further error quantities that are not covered be the first approach.

        \subsubsection{Reliability based on (discrete) convex duality}\vspace{-1mm}

        \hspace{5mm}In this subsection, we follow the procedure for the derivation of, by definition, reliable~and~con-stant-free \hspace{-0.1mm}a \hspace{-0.1mm}posteriori \hspace{-0.1mm}error \hspace{-0.1mm}estimates \hspace{-0.1mm}based \hspace{-0.1mm}on~\hspace{-0.1mm}(discrete)~\hspace{-0.1mm}convex~\hspace{-0.1mm}duality \hspace{-0.1mm}outlined~\hspace{-0.1mm}in~\hspace{-0.1mm}the~\hspace{-0.1mm}\mbox{introduction}.\vspace{-1mm}

    \begin{lemma}\label{lem:primal_dual_W12}
		If $f=f_h\in \mathcal{L}^0(\mathcal{T}_h)$, then for every $v\in K$, we have that
		\begin{align*}
			\begin{aligned}
			\tfrac{1}{2}\|\nabla v-\nabla u\|_{\Omega}^2+\langle -\Lambda,v-u\rangle_{\Omega}+\tfrac{1}{2}\|z_h^{rt}-z\|_{\Omega}^2&\leq\tfrac{1}{2}\| \nabla v-z_h^{rt}\|_{\Omega}^2+\eta_{B,h}^2(v)
			\\&\leq	\eta_{A,h}^2(v)+\eta_{B,h}^2(v)+\eta_{C,h}^2\,.
		\end{aligned}
		\end{align*}
	\end{lemma}
	
		\begin{proof}
			Using that, owing to the augmented problem \eqref{eq:augmented_problem}, $\frac{1}{2}a^2-\frac{1}{2}b^2=\frac{1}{2}(a-b)^2+b(a-b)$ for all $a,b\in \mathbb{R}$, and the strong concavity of \eqref{eq:obstacle_dual}, 
   for every $v\in K$ and $y\in (L^2(\Omega))^d$, it holds that
            \begin{align}
                \tfrac{1}{2}\|\nabla v-\nabla u\|_{\Omega}^2+\langle -\Lambda,v-u\rangle_{\Omega}&= I(v)-I(u)\,,\label{eq:co-co}\\
                \tfrac{1}{2}\|y-z\|_{\Omega}^2&\leq  D(z)-D(y)\,,\label{eq:co-co.2}
            \end{align}
            the strong duality relation \eqref{eq:obstacle_strong_duality}, that $z_h^{rt}\in H^2_N(\textup{div};\Omega)$ with $\textup{div}\,z_h^{rt}+f=\textup{div}\,z_h^{rt}+f_h=\smash{\overline{\lambda}}_h^{cr}\leq 0$ a.e.\ in $\Omega$ (\textit{cf}.\ \eqref{eq:discrete_optimality.2}), integration-by-parts, and $\Pi_hz_h^{rt}=\nabla_hu_h^{cr}$ a.e.\ in $\Omega$  (\textit{cf}.\ \eqref{eq:discrete_optimality.1}),~we~find~that
            	\begin{align*}
			\tfrac{1}{2}\|\nabla v-\nabla u\|_{\Omega}^2&+\langle -\Lambda,v-u\rangle_{\Omega}+\tfrac{1}{2}\|z_h^{rt}-z\|_{\Omega}^2\leq 
   I(v)-D(z_h^{rt})
   \\&= \tfrac{1}{2}\| \nabla v\|_{\Omega}^2+(\textup{div}\,z_h^{rt}-\smash{\overline{\lambda}}_h^{cr},v)_{\Omega}\\&\quad+\tfrac{1}{2}\| z_h^{rt}\|_{\Omega}^2+(\textup{div}\,z_h^{rt}+f_h,\chi)_{\Omega}\\&
					=\tfrac{1}{2}\|\nabla v\|_{\Omega}^2-(z_h^{rt},\nabla v)_{\Omega}+\tfrac{1}{2}\| z_h^{rt}\|_{\Omega}^2
				+\eta_{B,h}^2(v)
                \\&
					=\tfrac{1}{2}\|\nabla v-z_h^{rt}\|_{\Omega}^2+\eta_{B,h}^2(v)
					\\&
					\leq \eta_{A,h}^2(v)+\eta_{B,h}^2(v)
					+\| z_h^{rt}-\Pi_hz_h^{rt}\|_{\Omega}^2\,.
				\end{align*}
			Due to $z_h^{rt}-\Pi_hz_h^{rt}=\frac{\smash{\overline{\lambda}}_h^{cr}-f_h}{d}(\textup{id}_{\mathbb{R}^d}-\Pi_h\textup{id}_{\mathbb{R}^d})$ a.e.\ in $\Omega$ (\textit{cf}.\ \eqref{eq:generalized_marini}), we~conclude~the~assertion.
		\end{proof}

  \begin{remark} 
    \begin{itemize}[noitemsep,topsep=2pt,leftmargin=!,labelwidth=\widthof{(iii)},font=\itshape]
        \item[(i)] For every $v\in K$, due to \eqref{eq:variational_ineq}, we have that
    \begin{align*}
        \langle -\Lambda,v-u\rangle_{\Omega}=(\nabla u,\nabla v-\nabla u)_{\Omega}-(f,v-u)_{\Omega}\ge 0\,.
    \end{align*}
    \item[(ii)] Since $\smash{\overline{\lambda}}_h^{cr}\leq 0$ a.e.\  in $\Omega$ (\textit{cf}.\ Proposition \ref{prop:augmented}(ii)) and $\Pi_h(v-\chi)\ge 0$ a.e.\  in $\Omega$ for all $v\in K$, for every $v\in K$, we have that  $\eta_{B,h}^2(v)\ge 0$ and, thus,  $\eta_h^2(v)\ge 0$, since, then,
    \begin{align*}
        (-\smash{\overline{\lambda}}_h^{cr})\Pi_h(v-\chi)\ge 0\quad\text{ a.e.\  in }\Omega\,.
    \end{align*}
    \end{itemize}
    
    \end{remark}

  \begin{remark}
    \label{rem:primal_dual_CR0}
        The reliability estimate in Lemma \ref{lem:primal_dual_W12} is entirely constant-free.\vspace{-1mm}
	\end{remark}

  \begin{remark}[Improved reliability]
    \label{rem:primal_dual_CR}
        If we have that $v=v_h\in \mathcal{S}^1_D(\mathcal{T}_h)\cap K$ in Lemma \ref{lem:primal_dual_W12}, then, given  $z_h^{rt}-\Pi_hz_h^{rt}\perp_{L^2}\nabla v_h-\nabla_hu_h^{cr}$, we arrive at the improved reliability estimate
		\begin{align*}
			\begin{aligned}
			\tfrac{1}{2}\|\nabla v_h-\nabla u\|_{\Omega}^2+\langle -\Lambda,v_h-u\rangle_{\Omega}+\tfrac{1}{2}\|z_h^{rt}-z\|_{\Omega}^2
			&\leq	 \tfrac{1}{2}\eta_{A,h}^2(v_h)+\eta_{B,h}^2(v_h)+\tfrac{1}{2}\eta_{C,h}^2\,.
		\end{aligned}
		\end{align*}
	\end{remark}

    The  \textit{a posteriori} error estimator $\eta_h^2\colon\smash{ H^1_D(\Omega)}\to \mathbb{R}$ (\textit{cf}.\ \eqref{eq:primal-dual.1}), furthermore, controls the error between the continuous Lagrange multiplier $\Lambda\in \smash{(H^1_D(\Omega))}^*$, defined by \eqref{eq:augmented_problem}, and the discrete Lagrange multiplier $\smash{\overline{\lambda}}_h^{cr}\in \Pi_h(\mathcal{S}^{1,cr}_D(\mathcal{T}_h))$, defined by \eqref{eq:obstacle_lagrange_multiplier_cr},
    measured~in~the~Sobolev~dual~norm. To this end, we introduce the $\smash{(H^1_D(\Omega))}^*$-representation $\smash{\overline{\Lambda}}_h^{cr}\in \smash{(H^1_D(\Omega))}^*$ of $\smash{\overline{\lambda}}_h^{cr}\in \smash{\Pi_h(\mathcal{S}^{1,cr}_D(\mathcal{T}_h))}$, for every $v\in \smash{H^1_D(\Omega)}$ defined by\vspace{-1mm}
    \begin{align*}
       \smash{ \langle \smash{\overline{\Lambda}}_h^{cr},v\rangle_{\Omega}\coloneqq (\smash{\overline{\lambda}}_h^{cr},\Pi_h v)_{\Omega}\,.}
    \end{align*}

    \begin{lemma}\label{lem:large_lambda_efficiency}
        The following statements apply:
         \begin{itemize}[noitemsep,topsep=2pt,leftmargin=!,labelwidth=\widthof{(iii)},font=\itshape]
            \item[(i)] If we set $\textup{osc}_h^2(f)\coloneqq \|h_{\mathcal{T}}(f_h-f)\|_{\Omega}^2$ (\textit{cf}.\ Theorem \ref{thm:best-approxCR}), then it holds that
            \begin{align}\label{lem:large_lambda_efficiency.0.1}
            \|\smash{\overline{\Lambda}}_h^{cr}-\Lambda\|_{*,\Omega}&\leq 
            \|\nabla_h u_h^{cr}-\nabla u\|_{\Omega}+\eta_{C,h}+c_{\Pi}\,\smash{\textup{osc}_h(f)}
            \,.
        \end{align}
            \item[(ii)] If  $f=f_h\in \mathcal{L}^0(\mathcal{T}_h)$, then for every $v\in K$, it holds that
        \begin{align}\label{lem:large_lambda_efficiency.0.2}
		      \|\smash{\overline{\Lambda}}_h^{cr}-\Lambda\|_{*,\Omega}^2&\leq 9\,\eta_{A,h}^2(v)+6\,\eta_{B,h}^2(v)+9\,\eta_{C,h}^2\leq 9\,\eta_h^2(v)\,.
		\end{align}
         \end{itemize}
    \end{lemma}

    \begin{proof}\let\qed\relax
    \textit{ad (i).} For every $v\in H^1_D(\Omega)$ satisfying $\|v\|_{\Omega}+\|\nabla v\|_{\Omega}\leq 1$, using \eqref{eq:augmented_problem}, \eqref{eq:discrete_optimality.2},~integration-by-parts, \eqref{eq:generalized_marini}, $f-f_h\perp_{L^2} \Pi_h v$, and (\hyperlink{L0.2}{L0.2}), it holds that
    \begin{align}\label{lem:apriori_Lambda.1}
        \begin{aligned}
        \langle \smash{\overline{\Lambda}}_h^{cr}-\Lambda, v\rangle_{\Omega}&=(f_h+\textup{div}\,z_h^{rt},v)_{\Omega}-(f,v)_{\Omega}+(\nabla u,\nabla v)_{\Omega}\\&=(\nabla u-z_h^{rt},\nabla v)_{\Omega}+(f_h-f,v)_{\Omega}
        \\&=(\nabla u- \nabla_h u_h^{cr},\nabla v)_{\Omega}+(\tfrac{1}{d}(f_h-\smash{\overline{\lambda}}_h^{cr})(\textup{id}_{\mathbb{R}^d}-\Pi_h\textup{id}_{\mathbb{R}^d}),\nabla v )_{\Omega}\\&\quad+(f_h-f,v-\Pi_h v)_{\Omega}
        \\&\leq \|\nabla_h u_h^{cr}-\nabla u\|_{\Omega}+\eta_{C,h}+c_{\Pi}\,\smash{\textup{osc}_h(f)}\,.
        \end{aligned}
    \end{align}
    Taking the supremum with respect to every $v\in H^1_D(\Omega)$ satisfying $\|v\|_{\Omega}+\|\nabla v\|_{\Omega}\leq 1$ in \eqref{lem:apriori_Lambda.1}, we conclude the assertion.

    \textit{ad (ii).} Due to (i) and $f=f_h\in \mathcal{L}^0(\mathcal{T}_h)$, for every $v\in H^1_D(\Omega)$,~it~holds that
        \begin{align}\label{lem:large_lambda_efficiency.1}
            \|\smash{\overline{\Lambda}}_h^{cr}-\Lambda\|_{*,\Omega}\leq \|\nabla v-\nabla u\|_{\Omega}+\eta_{A,h}^2(v)+\eta_{C,h}^2
            \,.
        \end{align}
        Then, resorting in \eqref{lem:large_lambda_efficiency.1} to Lemma \ref{lem:primal_dual_W12}, for every $v\in K$, we find that
        \begin{align*}
            \|\smash{\overline{\Lambda}}_h^{cr}-\Lambda\|_{*,\Omega}^2&\leq 3\,\big\{ \|\nabla v-\nabla u\|_{\Omega}^2+
            \eta_{A,h}^2(v)+\eta_{C,h}^2
            \big\}\\&\leq 3\,\big\{3\,\eta_{A,h}^2(v)+2\,\eta_{B,h}^2(v)+3\, \eta_{C,h}^2
            \big\}\,.\tag*{$\qedsymbol$}
        \end{align*}
    \end{proof}

    \begin{remark}
        If $v=v_h\in \mathcal{S}^1_D(\mathcal{T}_h)\cap K$  in Lemma \ref{lem:primal_dual_W12}, then, given  Remark \ref{rem:primal_dual_CR}(i), we arrive at the improved reliability estimate
        \begin{align*}
		      \|\smash{\overline{\Lambda}}_h^{cr}-\Lambda\|_{*,\Omega}^2&\leq 6\,\eta_h^2(v)\,.
		\end{align*}
    \end{remark}

    \subsubsection{Reliability based on variational equations}
    
    \hspace{5mm}Following an approach which resorts to the (discrete) augmented problems, \textit{i.e.},~\eqref{eq:augmented_problem}~and~\eqref{eq:obstacle_lagrange_multiplier_cr}, it is possible to establish the following reliability result, which identifies additional~quantities~that are controlled by the \textit{a posteriori} error estimator $\eta_h^2\colon\smash{ H^1_D(\Omega)}\to \mathbb{R}$ (\textit{cf}.\ \eqref{eq:primal-dual.1}). 

    \begin{lemma}\label{lem:lambda_efficiency} 
    For every $v\in H^1_D(\Omega)$  and $\varepsilon,\tilde{\varepsilon}>0$, we have that
    \begin{align*}
        (\tfrac{1}{2}-\varepsilon\,c_{cr}^2-\tilde{\varepsilon}\,c_{\Pi}^2\big)\|\nabla v-\nabla u\|_{\Omega}^2&+\langle -\Lambda,v-u\rangle_{\Omega}+\tfrac{1}{2}\|\nabla u-\nabla_h u_h^{cr}\|_{\Omega}^2+(-\smash{\overline{\lambda}}_h^{cr},\Pi_h(u-\chi))_{\Omega}  \\&\leq \tfrac{1}{2}\eta_{A,h}^2(v)+\eta_{B,h}^2(v)+\tfrac{d^2}{4\varepsilon}\,\eta_{C,h}^2  +\tfrac{1}{4\tilde{\varepsilon}}\,\textup{osc}_h^2(f)\,.    
    \end{align*}
    \end{lemma}

    From Lemma \ref{lem:lambda_efficiency} we can immediately deduce the following reliability results.

    \begin{corollary}\label{rem:large_lambda_efficiency} The following statements apply:
    \begin{itemize}[noitemsep,topsep=2pt,leftmargin=!,labelwidth=\widthof{(iii)},font=\itshape]
        \item[(i)] For every $v\in H^1_D(\Omega)$, it holds that
        \begin{align*}
            \tfrac{1}{4}\|\nabla v-\nabla u\|_{\Omega}^2&+\langle -\Lambda,v-u\rangle_{\Omega}+\tfrac{1}{2}\|\nabla u-\nabla_h u_h^{cr}\|_{\Omega}^2+(-\smash{\overline{\lambda}}_h^{cr},\Pi_h(u-\chi ))_{\Omega}  \\&\quad\leq \tfrac{1}{2}\eta_{A,h}^2(v)+\eta_{B,h}^2(v)+2\,c_{cr}^2\,d^2\,\eta_{C,h}^2  +2\,c_{\Pi}^2\,\textup{osc}_h^2(f)\,.
        \end{align*}
        \item[(ii)] For every $v\in H^1_D(\Omega)$, it holds that
        \begin{align*}
          \langle -\Lambda,v-u\rangle_{\Omega}&+ \tfrac{1}{2}\|\nabla u-\nabla_h u_h^{cr}\|_{\Omega}^2+(-\smash{\overline{\lambda}}_h^{cr},\Pi_h(u-\chi))_{\Omega}  \\&\quad\leq 
           \tfrac{1}{2}\eta_{A,h}^2(v)+\eta_{B,h}^2(v)+c_{cr}^2\,d^2\,\eta_{C,h}^2  +\,c_{\Pi}^2\,\textup{osc}_h^2(f)\,.
        \end{align*}
         \end{itemize}
    \end{corollary}

    \begin{proof}
        The claim (i) follows from Lemma \ref{lem:lambda_efficiency} for $\varepsilon=\smash{\frac{1}{8c_{cr}^2}}>0$ and $\tilde{\varepsilon}=\smash{\frac{1}{8c_{\Pi}^2}}>0$.
        The claim (ii) follows from Lemma \ref{lem:lambda_efficiency} for $\varepsilon=\smash{\frac{1}{4c_{cr}^2}}>0$ and $\tilde{\varepsilon}=\smash{\frac{1}{4c_{\Pi}^2}}>0$.
    \end{proof}

    Having Corollary \ref{rem:large_lambda_efficiency} (ii) at hand, by analogy with Lemma \ref{lem:large_lambda_efficiency}(ii), we arrive at the following reliability result for the error between the continuous and the discrete Lagrange multiplier measured in the dual norm.

     \begin{corollary}\label{cor:large_lambda_efficiency}
        For every $v\in K$, we have that
        \begin{align*}
		      \|\smash{\overline{\Lambda}}_h^{cr}-\Lambda\|_{*,\Omega}^2&\leq 3\,\eta_{A,h}^2(v)+6\,\eta_{B,h}^2(v)+6(1+c_{cr}^2\,d^2)\,\eta_{C,h}^2+12\,c_{\Pi}^2\,\textup{osc}_h^2(f)
        \\&\leq  6(1+c_{cr}^2\,d^2)\,\eta_h^2(v)+12\,c_{\Pi}^2\,\textup{osc}_h^2(f)\,.
		\end{align*}
    \end{corollary}

    \begin{proof}
        Appealing to Lemma \ref{lem:large_lambda_efficiency}\eqref{lem:large_lambda_efficiency.0.1}, 
        we have that
        \begin{align}
            \|\smash{\overline{\Lambda}}_h^{cr}-\Lambda\|_{*,\Omega}\leq\|\nabla u-\nabla_hu_h^{cr}\|_{\Omega}+\eta_{C,h}+c_{\Pi}\,(\textup{osc}_h^2(f))^{\smash{\frac{1}{2}}}\,.\label{cor:large_lambda_efficiency.1}
        \end{align}
        Then, resorting in \eqref{cor:large_lambda_efficiency.1} to Corollary \ref{rem:large_lambda_efficiency}(ii), for every $v\in K$,~we~find~that
        \begin{align*}
            \|\smash{\overline{\Lambda}}_h^{cr}-\Lambda\|_{*,\Omega}^2&\leq 3\,\big\{ \|\nabla u-\nabla_hu_h^{cr}\|_{\Omega}^2+\eta_{C,h}^2+c_{\Pi}\,\textup{osc}_h^2(f)\big\}
            \\&\leq 3\,\big\{\eta_{A,h}^2(v)
            +2\,\eta_{B,h}^2(v)
            +
            2\,(1+c_{cr}^2\,d^2)\,\eta_{C,h}^2
            +4\,c_{\Pi}^2\,\textup{osc}_h^2(f)\big\}\,,
        \end{align*}
        which is the claimed reliability estimate.
    \end{proof}

     \begin{proof}[Proof (of Lemma \ref{lem:lambda_efficiency}).]
        Resorting to \eqref{eq:augmented_problem}, for every $v\in H^1_D(\Omega)$, owing to~${f-f_h\perp_{L^2}\Pi_h(u-v)}$, we find that
        \begin{align}\label{lem:lambda_efficiency.1}
            \begin{aligned}
                \langle \Lambda ,u-v\rangle_{\Omega}+(\nabla u,\nabla u-\nabla v)_{\Omega}&=
                (f,u-v)_{\Omega}
                \\&=(f-f_h,u-v)_{\Omega} +(f_h,u-v)_{\Omega}
                \\&=(f-f_h,u-v-\Pi_h(u-v))_{\Omega}+(f_h,u-v)_{\Omega}\,.
            \end{aligned}
        \end{align}
        Resorting to \eqref{eq:obstacle_lagrange_multiplier_cr}, for every $v\in H^1_D(\Omega)$, we find that
        \begin{align}\label{lem:lambda_efficiency.2}
            \begin{aligned}
              (\smash{\overline{\lambda}}_h^{cr}, \Pi_h\Pi_h^{cr}(u-v))_{\Omega} +(\nabla_h u_h^{cr},\nabla \Pi_h^{cr}(u- v))_{\Omega}&=
                (f_h,\Pi_h \Pi_h^{cr}(u-v))_{\Omega}
                \\&=-(f_h-\smash{\overline{\lambda}}_h^{cr}, u-v-\Pi_h^{cr}(u-v))_{\Omega}\\&\quad+(\smash{\overline{\lambda}}_h^{cr},\Pi_h \Pi_h^{cr}(u-v)-(u-v))_{\Omega}
                \\&\quad+(f_h,u-v)_{\Omega}\,,
            \end{aligned}
        \end{align}
        \textit{i.e.}, owing to $\nabla_h\Pi_h^{cr}(u-v)=\Pi_h\nabla(u-v)$ a.e.\ in $\Omega$ and $\nabla_hu_h^{cr}\in (\mathcal{L}^0(\mathcal{T}_h))^d$,~for~every~$v\in H^1_D(\Omega)$
         \begin{align}\label{lem:lambda_efficiency.3}
            \begin{aligned}
              (\smash{\overline{\lambda}}_h^{cr}, u-v)_{\Omega} +(\nabla u_h^{cr},\nabla u- \nabla v)_{\Omega}&=
                -(f_h-\smash{\overline{\lambda}}_h^{cr}, u-v-\Pi_h^{cr}(u-v))_{\Omega}
                \\&\quad+(f_h,u-v)_{\Omega}\,.
            \end{aligned}
        \end{align}
        If we subtract \eqref{lem:lambda_efficiency.3} from \eqref{lem:lambda_efficiency.1}, we arrive at
        \begin{align}\label{lem:lambda_efficiency.4}
            \begin{aligned}
            \langle\Lambda-\smash{\overline{\Lambda}}_h^{cr}, u-v\rangle_{\Omega} + (\nabla u-\nabla_h u_h^{cr},\nabla u-\nabla v)_{\Omega}&=(f-f_h,u-v-\Pi_h(u-v))_{\Omega} \\&\quad +(f_h-\smash{\overline{\lambda}}_h^{cr}, u-v-\Pi_h^{cr}(u-v))_{\Omega}\,.
            \end{aligned}
        \end{align}
        The binomial theorem shows that
        \begin{align}\label{lem:lambda_efficiency.5}
           \smash{(\nabla u-\nabla_h u_h^{cr},\nabla u-\nabla v)_{\Omega}+\tfrac{1}{2}\|\nabla v-\nabla_h u_h^{cr}\|_{\Omega}^2=\tfrac{1}{2}\|\nabla u-\nabla_h u_h^{cr}\|_{\Omega}^2+\tfrac{1}{2}\|\nabla v-\nabla u\|_{\Omega}^2\,.}
        \end{align}
        Resorting \hspace{-0.1mm}to  \hspace{-0.1mm}(\hyperlink{CR.2}{CR.2}), \hspace{-0.1mm}(\hyperlink{L0.2}{L0.2}) \hspace{-0.1mm}and \hspace{-0.1mm}the \hspace{-0.1mm}$\varepsilon$-Young \hspace{-0.1mm}inequality \hspace{-0.1mm}$ab\!\leq\! \frac{1}{4\varepsilon}a^2+\varepsilon b^2$, \hspace{-0.1mm}valid~\hspace{-0.1mm}for~\hspace{-0.1mm}all~\hspace{-0.1mm}${a,b\!\ge\! 0}$~\hspace{-0.1mm}and~\hspace{-0.1mm}${\varepsilon\!>\!0}$, for every $\varepsilon,\tilde{\varepsilon}>0$, we find that
        \begin{align} \vert (f-f_h,u-v-\Pi_h(u-v))_{\Omega}\vert &\leq \tfrac{1}{4\tilde{\varepsilon}}\,\text{osc}_h(f)+\tilde{\varepsilon}\, c_{\Pi}^2\,\|\nabla v-\nabla u\|_{\Omega}^2\label{lem:lambda_efficiency.7}\,,\\\label{lem:lambda_efficiency.6}
            \vert (f_h-\smash{\overline{\lambda}}_h^{cr},u-v-\Pi_h^{cr}(u-v))_{\Omega}\vert &\leq \tfrac{1}{4\varepsilon}\,\|h_{\mathcal{T}}(f_h-\smash{\overline{\lambda}}_h^{cr})\|_{\Omega}^2+\varepsilon\, c_{cr}^2\,\|\nabla v-\nabla u\|_{\Omega}^2\,.
        \end{align}
        Therefore,  combining \eqref{lem:lambda_efficiency.4}--\eqref{lem:lambda_efficiency.7}, we conclude the claimed inequality.
    \end{proof}
    
    Given the findings of Lemma \ref{lem:primal_dual_W12}, Lemma \ref{lem:lambda_efficiency}, and Lemma \ref{lem:large_lambda_efficiency}, 
    we introduce the error measure $\rho_h^2\colon H^1_D(\Omega)\to \mathbb{R}$, for every $v\in H^1_D(\Omega)$ defined by 
    \begin{align}\label{eq:rho}
        \begin{aligned}
             \rho_h^2(v)&\coloneqq \tfrac{1}{2}\|\nabla v-\nabla u\|_{\Omega}^2+\langle -\Lambda,v-u\rangle_{\Omega}
        \\&\quad+\|\nabla u-\nabla_hu_h^{cr}\|_{\Omega}^2+(-\smash{\overline{\lambda}}_h^{cr},\Pi_h(u-\chi ))_{\Omega}+\|\smash{\overline{\Lambda}}_h^{cr}-\Lambda\|_{*,\Omega}^2\,.
       \end{aligned}
    \end{align}

    \begin{theorem}[Reliability] \label{thm:reliability}
        There exist constants $c_{\textit{\textrm{rel}}},c_{\textit{\textrm{osc}}}>0$, depending only on the chunkiness $\omega_0>0$, such that for every $v\in K$, we have that
    \begin{align*}
    \smash{\rho_h^2(v)\leq c_{\textit{\textrm{rel}}}\,\eta_h^2(v)+c_{\textit{\textrm{osc}}}\,\textup{osc}_h^2(f)\,.}
    \end{align*}
    \end{theorem}

    \begin{proof}
        Immediate consequence of Lemma \ref{lem:primal_dual_W12}, Lemma \ref{lem:lambda_efficiency}, and Lemma \ref{lem:large_lambda_efficiency}.
    \end{proof}

    \begin{remark}[Comments on the reliability constant $c_{\textit{\textrm{rel}}}>0$]\hphantom{                      }
        \begin{itemize}[noitemsep,topsep=1pt,leftmargin=!,labelwidth=\widthof{(iii)}]
            \item[(i)] Appealing to Corollary \ref{cor:large_lambda_efficiency}, for every $v\in K$, we have that\enlargethispage{6mm}
            \begin{align*}
                \smash{\|\smash{\overline{\Lambda}}_h^{cr}-\Lambda\|_{*,\Omega}^2\leq 6\,(1+c_{cr}^2\,d^2)\,\eta_h^2(v)+12\,c_{\Pi}^2\,\textup{osc}_h^2(f)\,.}
            \end{align*}
            If $f=f_h\in \mathcal{L}^0(\mathcal{T}_h)$, then Lemma \ref{lem:large_lambda_efficiency} yields that  for every $v\in K$, we have that
            \begin{align*}
                \smash{\|\smash{\overline{\Lambda}}_h^{cr}-\Lambda\|_{*,\Omega}^2\leq \min\{9, 6\,(1+c_{cr}^2\,d^2)\}\,\eta_h^2(v)\,.}
            \end{align*}

            \item[(ii)] Appealing to Corollary \ref{rem:large_lambda_efficiency}, for every $v\in K$, we have that
            \begin{align*}
               \smash{ \tfrac{1}{2}\|\nabla v -\nabla u\|_{\Omega}^2}&+\smash{2\langle -\Lambda,v-u\rangle_{\Omega}+\|\nabla u -\nabla_h u_h^{cr}\|_{\Omega}^2+2( -\smash{\overline{\lambda}}_h^{cr},\Pi_h(u-\chi))_{\Omega}}\\&\leq 
              \smash{  \max\{2,4\,c_{cr}^2\,d^2\} \,\eta_h^2(v)+2\,c_{\Pi}^2\,\textup{osc}_h^2(f)\,.}
            \end{align*}
            If $f=f_h\in \mathcal{L}^0(\mathcal{T}_h)$, then Lemma \ref{lem:primal_dual_W12} yields that  for every $v\in K$, we have that
            \begin{align*}
                \smash{\tfrac{1}{2}\|\nabla v -\nabla u\|_{\Omega}^2+\langle -\Lambda,v-u\rangle_{\Omega}+\tfrac{1}{2}\|z_h^{rt} -z\|_{\Omega}^2\leq \,\eta_h^2(v)\,.}
            \end{align*}

            \item[(iii)] Combining (i) and (ii), we find that
            \begin{align*}
                \smash{c_{\textit{\textrm{rel}}}\leq \max\{2,4\,c_{cr}^2\,d^2\}+6(1+c_{cr}^2\,d^2)\,,\quad 
                c_{\textit{\textrm{osc}}} \leq 14\, c_{\Pi}^2\,.}
            \end{align*}
        \end{itemize}
    \end{remark}

    \newpage

    \subsection{Efficiency}

    \hspace{5mm}In this subsection, we show the efficiency of the \textit{a posteriori} error estimator $\eta_h^2\colon  H^1_D(\Omega)\to \mathbb{R}$ (\textit{cf}.\ \eqref{eq:primal-dual.1}) with respect to the error measure $\rho_h^2\colon H^1_D(\Omega)\to \mathbb{R}$ (\textit{cf}.\ \eqref{eq:rho}).

     \begin{theorem}[efficiency]\label{thm:efficiency}
    There exist constants $c_{\textit{\textrm{eff}}},c_{\textit{\textrm{osc}}}>0$, depending on the chunkiness $\omega_0>0$, such that for every $v\in H^1_D(\Omega)$, we have that
    \begin{align*}
        \eta_h^2(v)\leq c_{\textit{\textrm{eff}}}\,\rho_h^2(v)
        +c_{\textit{\textrm{osc}}}\,\textup{osc}_h^2(f)\,.
    \end{align*}
    \end{theorem}

    \begin{proof}
        Apparently, for every  $v\in H^1_D(\Omega)$, we have that
        \begin{align*}
            \eta_{A,h}^2(v)\leq 2\,\big\{\|\nabla  v-\nabla u\|_{\Omega}^2+\|\nabla u-\nabla_h u_h^{cr}\|_{\Omega}^2\big\}\,,
        \end{align*}
        In addition, appealing to Lemma \ref{lem:efficiency_global}\eqref{lem:efficiency_global.1},  there exists a constant $c>0$, depending only on the chunkiness $\omega_0>0$, such that  
        \begin{align*}
            \eta_{C,h}^2\leq c\,\big\{\|\nabla_h u_h^{cr}-\nabla u\|_{\Omega}^2+\|\smash{\overline{\Lambda}}_h^{cr}-\Lambda\|_{*,\Omega}^2+\textup{osc}_h^2(f)\big\}\,.
        \end{align*}
        Eventually, for every $v\in H^1_D(\Omega)$, using Young's and Poincar\'e's inequality, we find that
        \begin{align*}
           \eta_{B,h}^2(v) 
            &= (-\smash{\overline{\lambda}}_h^{cr},\Pi_h(u-\chi))_{\Omega}+\langle-\Lambda,v-u\rangle_{\Omega}+\langle\Lambda-\smash{\overline{\Lambda}}_h^{cr},v-u\rangle_{\Omega}
            \\&\leq (-\smash{\overline{\lambda}}_h^{cr},\Pi_h(u-\chi))_{\Omega}+\langle-\Lambda,v-u\rangle_{\Omega}\\&\quad+\tfrac{1}{2}\|\smash{\overline{\Lambda}}_h^{cr}-\Lambda\|_{*,\Omega}^2+\smash{\tfrac{1+c_P^2}{2}}\|\nabla v-\nabla u\|_{\Omega}^2\,,
        \end{align*}
        where $c_P>0$ denotes the Poincar\'e constant.
    \end{proof}

    \begin{remark}\label{rem:efficiency}
        Since the discrete primal solution $u_h^{cr}\in K^{cr}_h$ is neither an admissible~approx-imation of the primal solution $u\in K$ in Theorem \ref{thm:reliability} nor in Theorem \ref{thm:efficiency}, since, in general,
        \begin{align*}
            u_h^{cr}\notin H^1_D(\Omega)\quad\text{ and }\quad u_h^{cr}\not\ge \chi\text{ a.e.\  in }\Omega\,,
        \end{align*}
        it is necessary to post-process $u_h^{cr}\in K^{cr}_h$. 
        In the numerical experiments (\textit{cf}.\ Section \ref{sec:experiments}), we employ the post-processed function $v_h=\max\{\Pi_h^{av}u_h^{cr},\chi\}$, where
        $\Pi_h^{av}\colon \mathcal{S}^{1,cr}_D(\mathcal{T}_h)\to \mathcal{S}^1_D(\mathcal{T}_h)$ is~a~node-aver-aging quasi-interpolation operator (\textit{cf}.\ Subsection \ref{subsec:node-averagin}),
        which, by the Sobolev~chain~rule,~\mbox{satisfies}
        \begin{align*}
            v_h\in H^1_D(\Omega)\quad \text{ and }\quad v_h\ge \chi\text{ a.e.\  in }\Omega\,,
        \end{align*}
        \textit{i.e.}, $v_h\in K$. Note that $\textup{tr}\,v_h=0$ in $\Gamma_D$ due to $\Pi_h^{av}u_h^{cr}=0$ on $\Gamma_D$ and $\chi\leq 0$ in $\Gamma_D$.~In~addition,  using the best-approximation property of $\Pi_h^{av}\colon \mathcal{S}^{1,cr}_D(\mathcal{T}_h)\to \mathcal{S}^1_D(\mathcal{T}_h)$ (\textit{cf}.
        Proposition~\ref{lem:best-approx-inv}),~we~have~that
        \begin{align*}
            \|\nabla_h v_h-\nabla_h u_h^{cr}\|_\Omega&\leq\|\nabla \Pi_h^{av}u_h^{cr}-\nabla u_h^{cr}\|_{\Omega}+\|\nabla \Pi_h^{av}u_h^{cr}-\nabla \chi\|_{\{\Pi_h^{av}u_h^{cr}<\chi\}}
            \\&
            \leq 2\,\|\nabla \Pi_h^{av}u_h^{cr}-\nabla_h u_h^{cr}\|_{\Omega}+\|\nabla_h u_h^{cr}-\nabla \chi\|_{\{\Pi_h^{av}u_h^{cr}<\chi\}}
            \\&\leq 
            c\,\|\nabla_h u_h^{cr}-\nabla u\|_{\Omega}+\|\nabla u-\nabla \chi\|_{\{\Pi_h^{av}u_h^{cr}<\chi\}}\,,
        \end{align*}
        and, thus, 
        \begin{align*}
            \|\nabla_h v_h-\nabla u\|_\Omega
             \leq 
            c\,\|\nabla u_h^{cr}-\nabla u\|_{\Omega}+\|\nabla u-\nabla \chi\|_{\{\Pi_h^{av}u_h^{cr}<\chi\}}\,.
        \end{align*}
        In other words, the error between  $v_h\in K$ and $u_h^{cr}\in K^{cr}_h$ (and $u\in K $, respectively)~is~controlled~by the error
        between $u_h^{cr}\in K^{cr}_h$ and $u\in K$ plus a contribution capturing the violation~of~the~continuous obstacle constraint by $u_h^{cr}\in K^{cr}_h$.
    \end{remark}

    \begin{remark}
        Appealing to Theorem \ref{thm:best-approxCR} and Lemma \ref{lem:efficiency_global}\eqref{lem:efficiency_global.1}, there exists a constant $c>0$, depending only on the chunkiness $\omega_0>0$, such that for every $v_h\in \mathcal{S}^{1,cr}_D(\mathcal{T}_h)$, we have that
        \begin{align*}
            \|\nabla_h v_h-\nabla_h u_h^{cr}\|_\Omega^2+\|h_{\mathcal{T}}(f_h-\smash{\overline{\lambda}}_h^{cr})\|_{\Omega}^2\leq c\,\big\{\|\nabla_h v_h-\nabla u\|_{\Omega}^2+\|\smash{\overline{\Lambda}}_h^{cr}-\Lambda\|_{*,\Omega}^2+\textup{osc}_h^2(f)\big\}\,.
        \end{align*}
        Thus, it is possible to establish efficiency estimates for parts of the primal-dual \textit{a posteriori} error estimator, which also apply to non-conforming functions  (\textit{cf}.\ Appendix \ref{sec:medius}). 
    \end{remark}

    \newpage
    \section{Numerical experiments}\label{sec:experiments}
	
	\qquad In this section, we review the theoretical findings of  Section \ref{sec:apriori} and Section \ref{sec:aposteriori} via numerical experiments.\enlargethispage{3mm}
     All experiments were carried out using the finite element software package \mbox{\texttt{FEniCS}} \hspace{-0.1mm}(version \hspace{-0.1mm}2019.1.0, \hspace{-0.1mm}\textit{cf}.~\hspace{-0.1mm}\cite{LW10}). 
    \hspace{-0.1mm}All \hspace{-0.1mm}graphics \hspace{-0.1mm}are  \hspace{-0.1mm}created \hspace{-0.1mm}using   \hspace{-0.1mm}\texttt{Matplotlib} \hspace{-0.1mm}(version~\hspace{-0.1mm}3.5.1)~(\textit{cf}.~\hspace{-0.1mm}\cite{Hun07}).

    \subsection{Implementation details} 

    \hspace{5mm}We approximate the discrete primal solution ${u_h^{cr}\in \smash{\mathcal{S}^{1,cr}_D(\mathcal{T}_h)}}$ and the associated discrete Lagrange multiplier ${\smash{\overline{\lambda}}_h^{cr}\!\in\! \smash{\Pi_h(\mathcal{S}^{1,cr}_D(\mathcal{T}_h))}}$ jointly satisfying the discrete augmented problem~\eqref{eq:obstacle_lagrange_multiplier_cr}~via 
    the primal-dual active set strategy interpreted as a super-linear converging~semi-smooth~\mbox{Newton} method (\textit{cf}.\ \cite[Subsec.\ 5.3.1]{Bar15} or \cite{HIK02}). For sake of completeness, we will~briefly~outline~important implementation details related with this strategy. 

    We \hspace{-0.15mm}fix \hspace{-0.15mm}an \hspace{-0.15mm}ordering \hspace{-0.15mm}of \hspace{-0.15mm}the \hspace{-0.15mm}element \hspace{-0.15mm}sides \hspace{-0.15mm}$\smash{(S_i)_{i=1,\dots, N_h^{cr}}}$ \hspace{-0.15mm}and \hspace{-0.15mm}an \hspace{-0.15mm}ordering \hspace{-0.15mm}of \hspace{-0.15mm}the~\hspace{-0.15mm}elements~\hspace{-0.15mm}$\smash{(T_i)_{i=1,\dots, N_h^0}}$, where $N_h^{cr}\coloneqq \textup{card}(\mathcal{S}_h)$ and $N_h^0\coloneqq \textup{card}(\mathcal{T}_h)$, such that\footnote{In practice, the element $\widehat{T}\in \mathcal{T}_h$ for which $\mathbb{R}\chi_{\widehat{T}}\perp\smash{\Pi_h(\mathcal{S}^{1,cr}_D(\mathcal{T}_h))}$ is found via searching and erasing~a~zero~column (if existent) in the matrix $((\Pi_ h\varphi_{S_i},\chi_T)_{\Omega})_{i=1,\dots,N_h^{cr},T\in \mathcal{T}_h}\in \smash{\mathbb{R}^{N_h^{cr}\times N_h^0}}$ leading to $\mathrm{P}_h^{cr,0}\in \smash{\mathbb{R}^{N_h^{cr}\times N_h^{\smash{cr,0}}}}$.}  
    \begin{align*}
        \textup{span}(\{\chi_{T_i}\mid i= 1,\dots, N_h^{\smash{cr,0}}\})=\Pi_h(\mathcal{S}^{1,cr}_D(\mathcal{T}_h))\,,
    \end{align*}
    where $N_h^{\smash{cr,0}}=\textup{dim}(\Pi_h(\mathcal{S}^{1,cr}_D(\mathcal{T}_h)))\in \{N_h^0,N_h^0-1\}$ because of  ${\textup{codim}_{\mathcal{L}^0(\mathcal{T}_h)}(\Pi_h(\mathcal{S}^{1,cr}_D(\mathcal{T}_h)))\in \{0,1\}}$ (\textit{cf}.\ \cite[Cor.\ 3.2]{BW21}).
    Then, if we define the matrices
    \begin{align*}
        \mathrm{S}_h^{cr}&\coloneqq((\nabla_h\varphi_{S_i},\nabla_h\varphi_{S_j})_{\Omega})_{i,j=1,\dots,N_h^{cr}}\in \mathbb{R}^{N_h^{cr}\times N_h^{cr}}\,,\\
        \mathrm{P}_h^{cr,0}&\coloneqq((\Pi_ h\varphi_{S_i},\chi_{T_j})_{\Omega})_{i=1,\dots,N_h^{cr},j=1,\dots,N_h^{\smash{cr,0}}}\in \mathbb{R}^{N_h^{cr}\times N_h^{\smash{cr,0}}}\,,
        \intertext{and, assuming for the entire section that $\chi_h\coloneqq \Pi_h\Pi_h^{cr} \chi\in \mathcal{L}^0(\mathcal{T}_h)$, the vectors}
        \mathrm{X}_h^{cr}&\coloneqq ((\Pi_h^{cr}\chi,\varphi_{S_i})_{\Omega})_{i=1,\dots,N_h^{cr}}\in \mathbb{R}^{ N_h^{cr}}\,,\\
         \mathrm{F}_h^0&\coloneqq((f_h,\chi_{T_i})_{\Omega})_{i=1,\dots,N_h^{\smash{cr,0}}}\in \mathbb{R}^{ N_h^{\smash{cr,0}}}\,,
    \end{align*}
    the same argumentation as in \cite[Lem.\ 5.3]{Bar15} shows that
    the discrete augmented~problem \eqref{eq:obstacle_lagrange_multiplier_cr} is equivalent to finding vectors $(\mathrm{U}_h^{cr},\mathrm{L}_h^{cr})^\top \in \mathbb{R}^{N_h^{cr}}\times \mathbb{R}^{N_h^{\smash{cr,0}}}$ such that
    \begin{align}
\begin{aligned}\mathrm{S}_h^{cr}\mathrm{U}_h^{cr}+\mathrm{P}_h^{cr,0}\mathrm{L}_h^{cr}&=\mathrm{P}_h^{cr,0}\mathrm{F}_h^0&&\quad \text{ in }\mathbb{R}^{N_h^{cr}}\,,\\
    \mathscr{C}_h(\mathrm{U}_h^{cr},\mathrm{L}_h^{cr})&=0_{\smash{\tiny\mathbb{R}^{N_h^{\smash{cr,0}}}}}&&\quad \text{ in }\mathbb{R}^{N_h^{\smash{cr,0}}}\,,
    \end{aligned}\label{eq:details.1}
    \end{align}
    where for given $\alpha\hspace{-0.1em}>\hspace{-0.1em}0$, the mapping $\mathscr{C}_h\colon \mathbb{R}^{N_h^{cr}}\times \mathbb{R}^{N_h^{\smash{cr,0}}}\!\to \mathbb{R}^{N_h^{\smash{cr,0}}}$ for every ${(\mathrm{U}_h,\mathrm{L}_h)^\top\!\in  \mathbb{R}^{N_h^{cr}}\hspace{-0.15em}\times \hspace{-0.15em}\mathbb{R}^{N_h^{\smash{cr,0}}}}$ is defined by\enlargethispage{5mm}
    \begin{align*}
        \mathscr{C}_h(\mathrm{U}_h,\mathrm{L}_h)\coloneqq \mathrm{L}_h-\min\big\{0,\mathrm{L}_h+\alpha(\mathrm{P}_h^{cr,0})^{\top}(\mathrm{U}_h-\mathrm{X}_h^{cr})\big\}\quad\text{ in }\mathbb{R}^{N_h^{\smash{cr,0}}}\,.
    \end{align*}
    More precisely, the discrete primal solution ${u_h^{cr}\in \smash{\mathcal{S}^{1,cr}_D(\mathcal{T}_h)}}$ and  the associated discrete Lagrange multiplier ${\smash{\overline{\lambda}}_h^{cr}\in \smash{\Pi_h(\mathcal{S}^{1,cr}_D(\mathcal{T}_h))}}$ jointly satisfying the discrete augmented problem \eqref{eq:obstacle_lagrange_multiplier_cr}~as~well~as $(\mathrm{U}_h^{cr},\mathrm{L}_h^{cr})^\top \in \mathbb{R}^{N_h^{cr}}\times \mathbb{R}^{N_h^{\smash{cr,0}}}$, respectively, are related by\footnote{Here, for every $i=1,\dots,N$, $N\in \{N_h^{\smash{cr}},N_h^{\smash{cr,0}}\}$, we denote by $\mathrm{e}_i=(\delta_{ij})_{j=1,\dots,N}\in \mathbb{R}^N$, the $i$-th~unit~vector.}
    \begin{align*}
        \begin{aligned}
                u_h^{cr}&=\sum_{i=1}^{N_h^{cr}}{(\mathrm{U}_h^{cr}\cdot \mathrm{e}_i)\varphi_{S_i}}\in \mathcal{S}^{1,cr}_D(\mathcal{T}_h)\,,\\
                \smash{\overline{\lambda}}_h^{cr}&=\sum_{i=1}^{N_h^{cr,0}}{(\mathrm{L}_h^{cr}\cdot \mathrm{e}_i)\chi_{T_i}}\in \Pi_h(\mathcal{S}^{1,cr}_D(\mathcal{T}_h))\,.
        \end{aligned}
    \end{align*}
    Next,  \hspace{-0.15mm}define \hspace{-0.15mm}the \hspace{-0.15mm}mapping \hspace{-0.15mm}$\mathscr{F}_h\colon \mathbb{R}^{N_h^{cr}}\times \mathbb{R}^{N_h^{\smash{cr,0}}}\!\to \mathbb{R}^{N_h^{cr}}\times \mathbb{R}^{N_h^{\smash{cr,0}}}\!$ \hspace{-0.15mm}for \hspace{-0.15mm}every \hspace{-0.15mm}${(\mathrm{U}_h,\mathrm{L}_h)^\top\!\in  \mathbb{R}^{N_h^{cr}}\times \mathbb{R}^{N_h^{\smash{cr,0}}}}\!$~\hspace{-0.15mm}by
    \begin{align*}
        \mathscr{F}_h(\mathrm{U}_h,\mathrm{L}_h)\coloneqq \bigg[\begin{array}{c}
            \mathrm{S}_h^{cr}\mathrm{U}_h+\mathrm{P}_h^{cr,0}(\mathrm{L}_h- \mathrm{F}_h^0)\\
            \mathscr{C}_h(\mathrm{U}_h,\mathrm{L}_h)
        \end{array}\bigg]\quad \text{ in }\mathbb{R}^{N_h^{cr}}\times \mathbb{R}^{N_h^{\smash{cr,0}}}\,.
    \end{align*}
    Then, the non-linear system \eqref{eq:details.1} is equivalent to finding $(\mathrm{U}_h^{cr},\mathrm{L}_h^{cr})^\top \in \mathbb{R}^{N_h^{cr}}\times \mathbb{R}^{N_h^{\smash{cr,0}}}$ such that 
    \begin{align*}
        \mathscr{F}_h(\mathrm{U}_h^{cr},\mathrm{L}_h^{cr})=0_{\smash{\tiny\mathbb{R}^{N_h^{cr}}\times \mathbb{R}^{N_h^{\smash{cr,0}}}}}\quad\text{ in }\mathbb{R}^{N_h^{cr}}\times \mathbb{R}^{N_h^{\smash{cr,0}}}\,.
    \end{align*} 
    By analogy with \cite[Thm.\ 5.11]{Bar16}, one finds that the mapping $\mathscr{F}_h\colon \mathbb{R}^{N_h^{cr}}\times \mathbb{R}^{N_h^{\smash{cr,0}}}\to \mathbb{R}^{N_h^{cr}}\times \mathbb{R}^{N_h^{\smash{cr,0}}}$ is Newton-differentiable at every $(\mathrm{U}_h,\mathrm{L}_h)^\top\in  \mathbb{R}^{N_h^{cr}}\times \mathbb{R}^{N_h^{\smash{cr,0}}}$ and, with the \textit{(active) set}
    \begin{align*}
        \mathscr{A}_h\coloneqq \mathscr{A}_h(\mathrm{U}_h,\mathrm{L}_h)\coloneqq\big\{i\in \{1,\dots,N_h^{\smash{cr,0}}\}\mid (\mathrm{L}_h+\alpha(\mathrm{P}_h^{cr,0})^{\top}(\mathrm{U}_h-\mathrm{X}_h^{cr}))\cdot \mathrm{e}_i<0\big\}\,,
    \end{align*}
    we have that
    \begin{align*}
        D \mathscr{F}_h(\mathrm{U}_h,\mathrm{L}_h)\coloneqq \bigg[\begin{array}{cc}
            \mathrm{S}_h^{cr} & \mathrm{P}_h^{cr,0} \\
             \mathrm{I}_{\mathscr{A}_h}(\mathrm{P}_h^{cr,0})^\top & \mathrm{I}_{\mathscr{A}^c_h}
        \end{array}\bigg]\quad\text{ in }\mathbb{R}^{N_h^{cr}+N_h^{\smash{cr,0}}}\times \mathbb{R}^{N_h^{cr}+N_h^{\smash{cr,0}}}\,,
    \end{align*}
    where $\mathrm{I}_{\mathscr{A}_h},\mathrm{I}_{\mathscr{A}^c_h}\coloneqq I_{N_h^{\smash{cr,0}}\times N_h^{\smash{cr,0}}}-\mathrm{I}_{\mathscr{A}_h}\in \mathbb{R}^{N_h^{\smash{cr,0}}}\times \mathbb{R}^{N_h^{\smash{cr,0}}}$ for every $i,j\in \{1,\dots,N_h^{\smash{cr,0}}\}$ are~defined~by
    $(\mathrm{I}_{\mathscr{A}_h})_{ij}=1$ if $i=j\in \mathscr{A}_h$ and $(\mathrm{I}_{\mathscr{A}_h})_{ij}=0$ else.

    For a given iterate $(\mathrm{U}_h^{k-1},\mathrm{L}_h^{k-1})^\top\in \mathbb{R}^{N_h^{cr}}\times \mathbb{R}^{N_h^{\smash{cr,0}}}$, 
    one step of the~semi-smooth~\mbox{Newton}~method (\textit{cf}.\ \cite[Subsec.\ 5.3.1]{Bar15} or \cite{HIK02}) determines a direction $(\delta\mathrm{U}_h^{k-1},\delta\mathrm{L}_h^{k-1})^\top\in \mathbb{R}^{N_h^{cr}}\times \mathbb{R}^{N_h^{\smash{cr,0}}}$~such~that
    \begin{align}
             D \mathscr{F}_h(\mathrm{U}_h^{k-1},\mathrm{L}_h^{k-1})(\delta\mathrm{U}_h^{k-1},\delta\mathrm{L}_h^{k-1})^\top=-\mathscr{F}_h(\mathrm{U}_h^{k-1},\mathrm{L}_h^{k-1})\quad\text{ in }\mathbb{R}^{N_h^{cr}}\times \mathbb{R}^{N_h^{\smash{cr,0}}}\,.\label{eq:details.3}
    \end{align}
    Setting 
    $(\mathrm{U}_h^k,\mathrm{L}_h^k)^\top\hspace{-0.1em}\coloneqq\hspace{-0.1em} (\mathrm{U}_h^{k-1}+\delta\mathrm{U}_h^{k-1},\mathrm{L}_h^{k-1}+\delta\mathrm{L}_h^{k-1})^\top\in \mathbb{R}^{N_h^{cr}}\times \mathbb{R}^{N_h^{\smash{cr,0}}} $ and $\mathscr{A}_h^{k-1}\hspace{-0.1em}\coloneqq \hspace{-0.1em}\mathscr{A}_h(\mathrm{U}_h^{k-1},\mathrm{L}_h^{k-1})$, the 
     linear system~\eqref{eq:details.3} can equivalently be re-written as
    \begin{align}\label{eq:equiv_equ}
    \begin{aligned}
        \mathrm{S}_h^{cr} \mathrm{U}_h^k+\mathrm{P}_h^{cr,0}\mathrm{L}_h^k&=\mathrm{P}_h^{cr,0}\mathrm{F}_h^0&&\quad\text{ in }\mathbb{R}^{N_h^{cr}}\,,\\
         \mathrm{I}_{\smash{(\mathscr{A}_h^{k-1})^c}}\mathrm{L}_h^k&=0_{\smash{\tiny \mathbb{R}^{N_h^{\smash{cr,0}}}}}&&\quad\text{ in }\mathbb{R}^{N_h^{\smash{cr,0}}}\,,\\
        \mathrm{I}_{\smash{\mathscr{A}_h^{k-1}}}(\mathrm{P}_h^{cr,0})^\top\mathrm{U}_h^k&= \mathrm{I}_{\smash{\mathscr{A}_h^{k-1}}}(\mathrm{P}_h^{cr,0})^\top\mathrm{X}_h^{cr}&& \quad\text{ in }\mathbb{R}^{N_h^{\smash{cr,0}}}\,.
         \end{aligned}
    \end{align}
    The semi-smooth Newton method can, thus, equivalently be formulated in the following form, which is a version of a primal-dual active set strategy.\enlargethispage{12.5mm}

    \begin{algorithm}[Primal-dual active set strategy]\label{alg:semi-smooth} Choose parameters $\alpha>0$ and $\varepsilon_{\textup{STOP}}>0$. Moreover, let
     $(\mathrm{U}_h^0,\mathrm{L}_h^0)^\top\in \mathbb{R}^{N_h^{cr}}\times \mathbb{R}^{N_h^{\smash{cr,0}}}$ and set $k=1$. Then, for every $k\in \mathbb{N}$:
    \begin{itemize}[noitemsep,topsep=2pt,leftmargin=!,labelwidth=\widthof{(iii)},font=\itshape]
         \item[(i)] Define the most recent active set
         $$\mathscr{A}^{k-1}_h\coloneqq \mathscr{A}_h(\mathrm{U}_h^{k-1},\mathrm{L}_h^{k-1})\coloneqq \big\{i\in \{1,\dots,N_h^{\smash{cr,0}}\}\mid (\mathrm{L}_h^{k-1}+\alpha(\mathrm{P}_h^{cr,0})^{\top}(\mathrm{U}_h^{k-1}-\mathrm{X}_h^{cr}))\cdot \mathrm{e}_i<0\big\}\,.$$
         
         \item[(ii)]  Compute the iterate $(\mathrm{U}_h^k,\mathrm{L}_h^k)^\top\in \mathbb{R}^{N_h^{cr}}\times \mathbb{R}^{N_h^{\smash{cr,0}}}$ such that
         \begin{align*}
             \bigg[\begin{array}{cc}
            \mathrm{S}_h^{cr} & \mathrm{P}_h^{cr,0} \\
             \mathrm{I}_{\smash{\mathscr{A}_h^{k-1}}}(\mathrm{P}_h^{cr,0})^\top & \mathrm{I}_{\smash{(\mathscr{A}_h^{k-1})^c}}
        \end{array}\bigg]
         \bigg[\begin{array}{c}
            \mathrm{U}_h^k \\
            \mathrm{L}_h^k
        \end{array}\bigg]=\bigg[\begin{array}{c}
            \mathrm{P}_h^{cr,0}\mathrm{F}_h^0\\
            \mathrm{I}_{\smash{\mathscr{A}^{k-1}_h}}(\mathrm{P}_h^{cr,0})^\top\mathrm{X}_h^{cr}
        \end{array}\bigg]\,.
         \end{align*}

         \item[(iii)] Stop if $\vert \mathrm{U}_h^k-\mathrm{U}_h^{k-1}\vert \leq \varepsilon_{\textup{STOP}}$; otherwise, increase $k\to k+1$ and continue with step (i).
     \end{itemize}
        
    \end{algorithm}

    \begin{remark}[Important implementation details]
        \begin{itemize}[noitemsep,topsep=2pt,leftmargin=!,labelwidth=\widthof{(iii)},font=\itshape]
            \item[(i)]  Algorithm \ref{alg:semi-smooth} converges~\mbox{super-linearly} if $(U_h^0,L_h^0)^\top\in \mathbb{R}^{N_h^{cr}}\times \mathbb{R}^{N_h^{\smash{cr,0}}}$ is sufficiently close to the solution $(U_h^{cr},L_h^{cr})^\top\in \mathbb{R}^{N_h^{cr}}\times \mathbb{R}^{N_h^{\smash{cr,0}}}$ (\textit{cf}.\ \cite[Thm.\ 3.1]{HIK02}). Since the Newton-differentiability only holds in finite-dimensional~situations and deteriorates as the dimension increases, the condition on the initial guess becomes more critical for increasing dimensions.
            \item[(ii)] The degrees of freedom related to the entries $L_h^k|_{\smash{(\mathscr{A}^{k-1}_h)^c}}$ can be eliminated from the linear system of equations in Algorithm \ref{alg:semi-smooth}, step (ii) (see also \eqref{eq:equiv_equ}).
            \item[(iii)] Since only a finite number of active sets are possible, the algorithm terminates within a finite number of iterations at the exact solution $(\mathrm{U}_h^{cr},\mathrm{L}_h^{cr})^\top \in \mathbb{R}^{N_h^{cr}}\times \mathbb{R}^{N_h^{\smash{cr,0}}}$.~For~this~reason, in practice, the stopping criterion in step (iii) is reached with $\vert \mathrm{U}_h^{k^*}-\mathrm{U}_h^{k^*-1}\vert =0$~for~some~${k^*\in \mathbb{N}}$, in which case, one has that $\mathrm{U}_h^{k^*}=\mathrm{U}_h^{cr}$, provided $\varepsilon_{\textup{STOP}}>0$ is sufficiently small.
            \item[(iv)] The linear system emerging in each semi-smooth Newton step (\textit{cf}.\ Algorithm \ref{alg:semi-smooth}, step (ii)) is solved using a sparse direct solver from \textup{\texttt{SciPy}}~(version 1.8.1,~\textit{cf}.~\cite{SciPy}).
            \item[(v)] Global convergence of the algorithm and monotonicity, \textit{i.e.}, $\mathrm{U}_h^k\ge \mathrm{U}_h^{k-1}\ge \mathrm{X}_h^{cr}$ for $k\ge 3$ can be proved if $\mathrm{S}_h^{cr}\in \mathbb{R}^{N_h^{cr}\times N_h^{cr}}$ is an $M$-Matrix  (\textit{cf}.\ \cite[Thm.\ 3.2]{HIK02}).
            \item[(vi)] Classical active set strategies define $\mathscr{A}_h^{k-1}\coloneqq \{i\in \{1,\dots,N_h^{\smash{cr,0}}\}\mid \mathrm{L}_h^{k-1}\cdot \mathrm{e}_i<0\}$, which corresponds to the formal limit $\alpha\to \infty$.
        \end{itemize}
    \end{remark}

    \newpage
	\subsection{Numerical experiments concerning the \textit{a priori} error analysis}\label{subsec:num_a_priori}
	
	\hspace{5.5mm}In \hspace{-0.1mm}this \hspace{-0.1mm}subsection, \hspace{-0.1mm}we \hspace{-0.1mm}review \hspace{-0.1mm}the \hspace{-0.1mm}theoretical \hspace{-0.1mm}findings \hspace{-0.1mm}of \hspace{-0.1mm}Section \hspace{-0.1mm}\ref{sec:apriori}.
    
    For  our numerical experiments, we choose $\Omega\coloneqq (-\frac{3}{2},\frac{3}{2})^2$, $\Gamma_D\coloneqq \partial \Omega$, $f\coloneqq -2\in L^2(\Omega)$,~${\chi \coloneqq 0}\in H^2(\Omega)$, so that the exact solution ${u\in K}$, for~every~${x\in \Omega}$~defined~by
    \begin{align*}
    	u(x)\coloneqq \begin{cases}
    	    \tfrac{\vert x\vert^2}{2}-\ln(\vert x\vert)-\tfrac{1}{2}&\text{ if } x\in\Omega\setminus B_1^2(0)\\
            0&\text{ else }
    	\end{cases}\,, 
    \end{align*}
    satisfies $u\in H^2(\Omega)$.
    Therefore, Theorem \ref{thm:apriori} lets us expect a
    convergence rate of about $1$.
     
    An initial triangulation $\mathcal
    T_{h_0}$, $h_0=\smash{\frac{3}{2\sqrt{2}}}$, is constructed by subdividing a rectangular~Cartesian grid into regular triangles with different orientations.  Refined     triangulations $\mathcal T_{h_k}$,~$k=1,\dots,7$, where $h_{k+1}=\frac{h_k}{2}$ for all $k=1,\dots,7$, are 
    obtained by 
    applying the red-refinement rule (\textit{cf}.\ \cite{Ver13}).
    
    For the resulting series of triangulations $\mathcal T_k\coloneqq \mathcal T_{h_k}$, $k=1,\dots,7$, we~apply~the~primal-dual active set strategy (\textit{cf}.\ Algorithm \ref{alg:semi-smooth}) to compute the discrete~primal~solution ${u_k^{cr}\coloneqq u_{h_k}^{cr}\in \mathcal{S}^{1,cr}_D(\mathcal{T}_k)}$, ${k=1,\dots,7}$, the discrete Lagrange multiplier $\smash{\overline{\lambda}}_k^{cr}\coloneqq \smash{\overline{\lambda}}_{h_k}^{cr}\in \Pi_{h_k}(\mathcal{S}^{1,cr}_D(\mathcal{T}_k))$, and, subsequently, resorting to \eqref{eq:generalized_marini}, the discrete dual solution $z_k^{\textit{rt}}\coloneqq z_{h_k}^{\textit{rt}}\in\mathcal{R}T^0_N(\mathcal{T}_k)$,~${k=1,\dots,7}$.~Afterwards, we compute the error quantities
    \begin{align}\label{errors}
    	\left.\begin{aligned}
    		e_{u}^{k}&\coloneqq \|\nabla_{\!h_k}u_k^{cr}-\nabla u\|_{\Omega}\,,\\
            e_{\Pi_h^{cr}u}^{k}&\coloneqq \|\nabla_{\!h_k}u_k^{cr}-\nabla_{\! h_k}\Pi_h^{cr}u\|_{\Omega}\,,\\
    		e_{z}^{k}&\coloneqq \|z_k^{\textrm{\textit{rt}}}-z\|_{\Omega}\,,\\
            e_{\Pi_h^{rt}z}^{k}&\coloneqq \|\Pi_{h_k} z_k^{\textrm{\textit{rt}}}-\Pi_{h_k} \Pi_h^{rt}z\|_{\Omega}\,,\\
            e_{\lambda,u}^{k}&\coloneqq (-\smash{\overline{\lambda}}_k^{cr},\Pi_{h_k}(u-u_k^{cr}))_{\Omega}\,,\\
            e_{\lambda,\Pi_h^{cr}u}^{k}&\coloneqq (-\smash{\overline{\lambda}}_k^{cr},\Pi_{h_k}(\Pi_h^{cr}u-u_k^{cr}))_{\Omega}\,,
    \end{aligned}\quad\right\}\quad k=1,\dots,7\,.
    \end{align}
    For the determination of the convergence rates,  the experimental order of convergence~(EOC)
    \begin{align*}
    	\texttt{EOC}_k(e_k)\coloneqq \frac{\log(e_k/e_{k-1})}{\log(h_k/h_{k-1})}\,, \quad k=2,\dots,7\,,
    \end{align*}
    where for every $k= 1,\dots,7$, we denote by $e_k$,
    either $\smash{e_u^k}$, $\smash{e_{\Pi_h^{cr}u}^k}$, $\smash{e_z^k}$, 
    $\smash{e_{\Pi_h^{rt}z}^k}$, $\smash{e_{\lambda,u}^k}$, $\smash{e_{\lambda,\Pi_h^{cr}u}^k}$, $\smash{e_{\lambda,u}^{\textrm{\textit{tot}},k}}\coloneqq \smash{e_{\lambda,u}^k}+\smash{e_u^k}$, or $\smash{e_{\lambda,\Pi_h^{cr}u}^{\textrm{\textit{tot}},k}}\coloneqq \smash{e_{\lambda,\Pi_h^{cr}u}^k}+\smash{e_{\Pi_h^{cr}u}^k}$,
    respectively, is recorded.
    
    For a series of triangulations $\mathcal{T}_k$, $k = 1,\dots,7$,
    obtained by uniform mesh~refinement~as~described above, the EOC is
    computed and  presented in Table~\ref{tab1} and Table~\ref{tab2}.~In~each~case, except for $e_k\in \{e_{\lambda,u,k} ,e_{\lambda,\Pi_h^{cr}u,k}\}$,  we record a convergence ratio of about $\texttt{EOC}_k(e_k)\approx 1$, $k=2,\dots,7$, confirming the optimality of the  \textit{a priori} error estimates established in Theorem~\ref{thm:apriori} and Corollary \ref{cor:apriori}. For $e_k\in \{e_{\lambda,u,k} ,e_{\lambda,\Pi_h^{cr}u,k}\}$, we record a 
    convergence ratio~of~about~${\texttt{EOC}_k(e_k)\approx 1.5}$,~${k=2,\dots,7}$.
\begin{table}[H]
     \setlength\tabcolsep{8pt}
 	\centering
 	\begin{tabular}{c |c|c|c|c|c|c|c|c|} 
 	\hline 
		   
		    \multicolumn{1}{|c||}{\cellcolor{lightgray}$k$}
		    & \cellcolor{lightgray}$e_u^k$
		    & \cellcolor{lightgray} $\texttt{EOC}_k$ & \cellcolor{lightgray}$e_{\Pi_h^{cr}u}^k$  & \multicolumn{1}{c|}{\cellcolor{lightgray}$\texttt{EOC}_k$}  &  \multicolumn{1}{c|}{\cellcolor{lightgray}$e_z^k$} & \cellcolor{lightgray}$\texttt{EOC}_k$  & \cellcolor{lightgray}$e_{\Pi_h^{rt}z}^k$ &  \cellcolor{lightgray}$\texttt{EOC}_k$  \\ \hline\hline
            \multicolumn{1}{|c||}{\cellcolor{lightgray}$1$}  & 1.359 & ---   & 0.732 & ---   & 1.094 & ---   & 0.656 & ---   \\ \hline
			\multicolumn{1}{|c||}{\cellcolor{lightgray}$2$}  & 0.787 & 0.788 & 0.664 & 0.141 & 0.533 & 1.038 & 0.453 & 0.535 \\ \hline
			\multicolumn{1}{|c||}{\cellcolor{lightgray}$3$}  & 0.380 & 1.048 & 0.324 & 1.034 & 0.260 & 1.033 & 0.212 & 1.097 \\ \hline
			\multicolumn{1}{|c||}{\cellcolor{lightgray}$4$}  & 0.197 & 0.948 & 0.166 & 0.968 & 0.131 & 0.993 & 0.116 & 0.872 \\ \hline
			\multicolumn{1}{|c||}{\cellcolor{lightgray}$5$}  & 0.099 & 0.996 & 0.082 & 1.008 & 0.067 & 0.974 & 0.059 & 0.967 \\ \hline
			\multicolumn{1}{|c||}{\cellcolor{lightgray}$6$}  & 0.050 & 0.989 & 0.042 & 0.986 & 0.033 & 1.010 & 0.030 & 0.968 \\ \hline
			\multicolumn{1}{|c||}{\cellcolor{lightgray}$7$}  & 0.025 & 0.998 & 0.021 & 1.001 & 0.017 & 0.980 & 0.015 & 0.993 \\ \hline
\end{tabular}
 	\caption{For $e_k\in \{e_u^k,e_{\Pi_h^{cr}u}^k,e_z^k,e_{\Pi_h^{rt}z}^k\}$,~${k=2,\dots,7}$: error $e_k$ and $\texttt{EOC}_k(e_k)$.} 
 	\label{tab1}
 \end{table}

    \begin{table}[H]
     \setlength\tabcolsep{8pt}
 	\centering
 	\begin{tabular}{c |c|c|c|c|c|c|c|c|} 
 	\hline 
		   
		    \multicolumn{1}{|c||}{\cellcolor{lightgray}$k$}
		    & \cellcolor{lightgray}$e_{\lambda,u}^k$
		    & \cellcolor{lightgray} $\texttt{EOC}_k$ & \cellcolor{lightgray}$e_{\lambda,\Pi_h^{cr}u}^k$  & \multicolumn{1}{c|}{\cellcolor{lightgray}$\texttt{EOC}_k$}  &  \multicolumn{1}{c|}{\cellcolor{lightgray}$e_{\lambda,u}^{\textrm{\textit{tot}},k}$} & \cellcolor{lightgray}$\texttt{EOC}_k$  & \cellcolor{lightgray}$e_{\lambda,\Pi_h^{cr}u}^{\textrm{\textit{tot}},k}$ &  \cellcolor{lightgray}$\texttt{EOC}_k$  \\ \hline\hline
            \multicolumn{1}{|c||}{\cellcolor{lightgray}$1$}  & 0.262 & ---   & 0.490 & ---   & 1.849 & ---   & 1.223 & ---   \\ \hline
			\multicolumn{1}{|c||}{\cellcolor{lightgray}$2$}  & 0.144 & 0.866 & 0.199 & 1.300 & 0.986 & 0.907 & 0.863 & 0.502 \\ \hline
			\multicolumn{1}{|c||}{\cellcolor{lightgray}$3$}  & 0.044 & 1.706 & 0.072 & 1.461 & 0.453 & 1.123 & 0.397 & 1.122 \\ \hline
			\multicolumn{1}{|c||}{\cellcolor{lightgray}$4$}  & 0.020 & 1.133 & 0.029 & 1.308 & 0.226 & 0.999 & 0.195 & 1.024 \\ \hline
			\multicolumn{1}{|c||}{\cellcolor{lightgray}$5$}  & 0.006 & 1.732 & 0.009 & 1.636 & 0.108 & 1.064 & 0.092 & 1.086 \\ \hline
			\multicolumn{1}{|c||}{\cellcolor{lightgray}$6$}  & 0.002 & 1.363 & 0.003 & 1.447 & 0.053 & 1.024 & 0.045 & 1.027 \\ \hline
			\multicolumn{1}{|c||}{\cellcolor{lightgray}$7$}  & 0.001 & 1.618 & 0.001 & 1.535 & 0.026 & 1.027 & 0.022 & 1.036 \\ \hline
\end{tabular}\vspace{-1mm}
 	\caption{For $e_k\in\smash{\{e_{\lambda,u}^k,e_{\lambda,\Pi_h^{cr}u}^k,e_{\lambda,u}^{\textrm{\textit{tot}},k},e_{\lambda,\Pi_h^{cr}u}^{\textrm{\textit{tot}},k}\}}$,~${k=2,\dots,7}$: error $e_k$ and  $\texttt{EOC}_k(e_k)$.} \vspace{-2mm}
 	\label{tab2}
 \end{table}
    
    \subsection{Numerical experiments concerning \textit{a posteriori} error analysis}\label{subsec:num_a_posteriori}\vspace{-1mm}

    \qquad In \hspace{-0.1mm}this \hspace{-0.1mm}subsection, \hspace{-0.1mm}we \hspace{-0.1mm}review \hspace{-0.1mm}the \hspace{-0.1mm}theoretical \hspace{-0.1mm}findings \hspace{-0.1mm}of \hspace{-0.1mm}Section \hspace{-0.1mm}\ref{sec:aposteriori}.
	\hspace{-0.1mm}More~\hspace{-0.1mm}precisely,~\hspace{-0.1mm}we~\hspace{-0.1mm}apply~\hspace{-0.1mm}the $\smash{\mathcal{S}^{1,cr}_D(\mathcal{T}_h)}$-approximation \eqref{eq:obstacle_discrete_primal} of 
	the obstacle problem \eqref{eq:obstacle_primal} in an adaptive~mesh~refinement algorithm based on local refinement indicators $(\eta^2_{h,T})_{T\in \mathcal{T}_h}$ associated with  the \textit{a posteriori} error estimator $\eta^2_{h}$ (\textit{cf}.\ \eqref{eq:primal-dual.1}). 
    Specifically, for every $v\in \smash{H^1_D(\Omega)}$~and~${T\in \mathcal{T}_h}$,~we~define
    \begin{align*}
        \eta_{A,h,T}^2(v)&\coloneqq \| \nabla v-\nabla_h u_h^{cr}\|_T^2\,,\\
        \eta_{B,h,T}^2(v)&\coloneqq (-\smash{\overline{\lambda}}_h^{cr},\Pi_h(v-\chi))_T\,,\\
        \eta_{C,h,T}^2&\coloneqq \tfrac{1}{d^2}\|h_{\mathcal{T}}(f_h-\smash{\overline{\lambda}}_h^{cr})\|_T^2\,,\\
        \eta_{h,T}^2(v)&\coloneqq \eta_{A,h,T}^2(v)+\eta_{B,h,T}^2(v)+\eta_{C,h,T}^2\,.
    \end{align*}
    
    Before we present numerical experiments, we briefly outline the  details of the implementations. 
    In general, we follow the adaptive algorithm:\enlargethispage{7.5mm}
    
	\begin{algorithm}[AFEM]\label{alg:afem}
		Let $\varepsilon_{\textup{STOP}}\!>\!0$, $\theta\!\in\! (0,1)$ and  $\mathcal{T}_0$ a conforming initial  triangulation~of~$\Omega$. Then, for every $k\in \mathbb{N}\cup\{0\}$:
	\begin{description}[noitemsep,topsep=1pt,labelwidth=\widthof{\textit{('Estimate')}},leftmargin=!,font=\normalfont\itshape]
		\item[('Solve')]\hypertarget{Solve}{}
		\hspace{-0.5mm}Compute \hspace{-0.15mm}the \hspace{-0.15mm}discrete \hspace{-0.15mm}primal \hspace{-0.15mm}solution \hspace{-0.15mm}$u_k^{cr}\hspace{-0.15em}\coloneqq\hspace{-0.15em} u_{h_k}^{cr} \hspace{-0.15em}\in\hspace{-0.15em} \smash{K_k^{cr}\hspace{-0.15em}\coloneqq \hspace{-0.15em}K_{h_k}^{cr}}$
        \hspace{-0.15mm}and~\hspace{-0.15mm}the~\hspace{-0.15mm}discrete~\hspace{-0.15mm}\mbox{Lagrange} multiplier $\smash{\overline{\lambda}}_k^{cr}\hspace{-0.1em}\coloneqq \hspace{-0.1em}\smash{\overline{\lambda}}_{h_k}^{cr}\hspace{-0.1em}\in \hspace{-0.1em}\Pi_{h_k}(\smash{\mathcal{S}^{1,cr}_D(\mathcal{T}_k)})$ jointly solving the discrete~augmented~problem \eqref{eq:obstacle_lagrange_multiplier_cr}.  Post-process $u_k^{cr}\in \smash{\mathcal{S}^{1,cr}_D(\mathcal{T}_k)}$ to obtain a conforming approximation $v_k\in K$ of the primal solution $u\in K$ and a discrete dual solution~${z_k^{\textrm{\textit{rt}}}\coloneqq z_{h_k}^{\textrm{\textit{rt}}}\in \mathcal{R}T^0_N(\mathcal{T}_k)}$.
		\item[('Estimate')]\hypertarget{Estimate}{} Compute the local refinement indicators $\smash{(\eta^2_{k,T}(v_k))_{T\in \mathcal{T}_k}\coloneqq (\eta^2_{h_k,T}(v_k))_{T\in \mathcal{T}_k}}$. If $\eta^2_k(v_k)\coloneqq \eta^2_{h_k}(v_k)\leq \varepsilon_{\textup{STOP}}$, then \textup{STOP}; otherwise, continue with step (\hyperlink{Mark}{'Mark'}).
		\item[('Mark')]\hypertarget{Mark}{}  Choose a minimal (in terms of cardinality) subset $\mathcal{M}_k\subseteq\mathcal{T}_k$ such that
		\begin{align*}
			\sum_{T\in \mathcal{M}_k}{\eta_{k,T}^2(v_k)}\ge \theta^2\sum_{T\in \mathcal{T}_k}{\eta_{k,T}^2(v_k)}\,.
		\end{align*}
		\item[('Refine')]\hypertarget{Refine}{} Perform a conforming refinement of $\mathcal{T}_k$ to obtain $\mathcal{T}_{k+1}$~such~that~each $T\in \mathcal{M}_k$  is refined in $\mathcal{T}_{k+1}$. 
		Increase~$k\mapsto k+1$~and~continue~with~step~(\hyperlink{Solve}{'Solve'}).
	\end{description}
	\end{algorithm}

	\begin{remark}
			\begin{description}[noitemsep,topsep=1pt,labelwidth=\widthof{\textit{(iii)}},leftmargin=!,font=\normalfont\itshape]
				\item[(i)] \hspace{-0.5em}The \hspace{-0.1mm}discrete \hspace{-0.1mm}primal \hspace{-0.1mm}solution \hspace{-0.1mm}$u_k^{cr}\in\smash{K_k^{cr}}$ \hspace{-0.1mm}and \hspace{-0.1mm}the~\hspace{-0.1mm}discrete~\hspace{-0.1mm}Lagrange~\hspace{-0.1mm}\mbox{multiplier} $\smash{\overline{\lambda}}_k^{cr}\in \Pi_{h_k}(\smash{\mathcal{S}^{1,cr}_D(\mathcal{T}_k)})$ in step (\hyperlink{Solve}{'Solve'}) are computed using the primal-dual~active~set~strategy (\textit{cf}.~Algorithm~\ref{alg:semi-smooth}) for the parameter $\alpha=1$.
				\item[(ii)] The reconstruction of the discrete dual solution $z_k^{\textit{rt}}\in \smash{\mathcal{R}T^0_N(\mathcal{T}_k)}$ in step (\hyperlink{Solve}{'Solve'}) is based on the generalized Marini formula \eqref{eq:generalized_marini}.
                \item[(iii)] In accordance with Remark \ref{rem:efficiency}, as a conforming approximation $v_k\in K$ of the primal~solution $u\in K$ in step (\hyperlink{Solve}{'Solve'}), we employ $v_k=\max\{\Pi^{av}_{h_k} u_k^{cr},\chi\}\in K$.
				\item[(iv)] We always employ the parameter $\theta=\smash{\frac{1}{2}}$ in (\hyperlink{Estimate}{'Mark'}).
				\item[(v)] To find the set $\mathcal{M}_k\subseteq \mathcal{T}_k$ in step (\hyperlink{Mark}{'Mark'}), we~deploy~the~D\"orfler marking strategy~(\textit{cf}.~\cite{Doe96}).
				\item[(vi)] The \hspace*{-0.1mm}(minimal) \hspace*{-0.1mm}conforming \hspace*{-0.1mm}refinement \hspace*{-0.1mm}of \hspace*{-0.1mm}$\mathcal{T}_k$ \hspace*{-0.1mm}with \hspace*{-0.1mm}respect \hspace*{-0.1mm}to \hspace*{-0.1mm}$\mathcal{M}_k$~\hspace*{-0.1mm}in~\hspace*{-0.1mm}step \hspace*{-0.1mm}(\hyperlink{Refine}{'Refine'})~\hspace*{-0.1mm}is~\hspace*{-0.1mm}\mbox{obtained}~by deploying the \textit{red}-\textit{green}-\textit{blue}-refinement algorithm~(\textit{cf}.~\cite{Ver13}).\vspace{-1mm}
			\end{description}
	\end{remark}\newpage

    \subsubsection{Example with corner singularity}\vspace{-1mm}

    \qquad We examine an example from \cite{BC08}. In this example, we let $\Omega\coloneqq (-2,2)^2\setminus([0,2]\times [-2,0]) $, $ \Gamma_D \coloneqq \partial\Omega$, $ \Gamma_N \coloneqq \emptyset$, $f\in L^2(\Omega)$, in polar coordinates, for every $(r,\varphi)^\top\in \mathbb{R}_{>0}\times (0,2\pi)$ defined by
    \begin{align*}
        \smash{f(r,\varphi)\coloneqq -r^{\frac{2}{3}}\sin(\tfrac{2\varphi}{3})(\tfrac{\gamma_1'(r)}{r}+\gamma_1''(r))-\tfrac{4}{3}r^{-\tfrac{1}{3}}\gamma_1'(r)\sin(\tfrac{2\varphi}{3})-\gamma_2(r)}\,,
    \end{align*}
    where $\gamma_1,\gamma_2\colon \mathbb{R}_{>0}\to \mathbb{R}$ for every $r\in \mathbb{R}_{>0}$, abbreviating $\overline{r}\coloneqq 2(r-\tfrac{1}{4})$, are defined by
    \begin{align*}
        \gamma_1(r)\coloneqq 
        \begin{cases}
            1&\text{ if }\overline{r}<0\\[-0.5mm]
            -6\overline{r}^5+15\overline{r}^4-10\overline{r}^3+1&\text{ if }0\leq \overline{r}<1\\[-0.5mm]
            0&\text{ if }\overline{r}\ge 1
        \end{cases}\,,\qquad
        \gamma_2(r)\coloneqq 
        \begin{cases}
            0&\text{ if }\overline{r}\leq \frac{5}{4}\\
            1&\text{ if }\overline{r}> \frac{5}{4}
        \end{cases}\,,
    \end{align*}
    and $\chi\coloneqq 0\in H^1_D(\Omega)$. Then, the primal solution $u\in K$, in polar coordinates, for every $(r,\varphi)^\top\in \mathbb{R}_{>0}\times (0,2\pi)$ defined by $u(r,\varphi)\coloneqq r^{\frac{2}{3}}\gamma_1(r)\sin(\tfrac{2\varphi}{3})$, 
    has a singularity~at~the~origin~and,~therefore, satisfies $u\not\in H^2(\Omega)$, so that we cannot  expect uniform mesh refinement to yield the quasi-optimal linear convergence rate.\enlargethispage{8mm}

    The coarsest triangulation $\mathcal{T}_0$ of Figure \ref{fig:triang_singularity} (and starting triagulation of Algorithm \ref{alg:afem}) consists of $48$ halved squares. More precisely, Figure \ref{fig:triang_singularity} displays
    the triangulations~$\mathcal{T}_k$,~$k\in \{0,4,8,12,16,20\}$, generated by the adaptive Algorithm \ref{alg:afem}.
    The approximate contact zones $\mathcal{C}_k^{cr}\coloneqq \{\Pi_{h_k}u_k^{cr}=0\}=\{\smash{\overline{\lambda}}_k^{cr}<0\}$, $k\in \{0,4,8,12,16,20\}$, 
    are plotted~in~white~Figure~\ref{fig:triang_singularity}~while~its~complement~is~shaded\footnote{we chose this color as in most of the examples the complement of the contact zone~is~refined~and~appears~darker.}.
    Algorithm \ref{alg:afem} refines in the complement of the contact zone $\mathcal{C}\coloneqq\Omega\cap \{\vert\cdot\vert > \frac{3}{4}\}$.
    A refinement towards~the origin, where the solution has a singularity in the gradient, and~in~$\{\frac{1}{4}\leq \vert\cdot\vert \leq \frac{3}{4}\}$, where the solution has large gradients, is reported.
    This behavior can also be seen in Figure \ref{fig:solution_singularity}, where the discrete primal solution $u_{10}^{cr}\in \mathcal{S}^{1,cr}_D(\mathcal{T}_{10})$, the node-averaged discrete primal solution $\Pi_{h_{10}}^{av}u_{10}^{cr}\in \mathcal{S}^1_D(\mathcal{T}_{10})$, the discrete Lagrange multiplier
    $\smash{\overline{\lambda}}_{10}^{cr}\in \Pi_{h_{10}}(\mathcal{S}^{1,cr}_D(\mathcal{T}_{10}))$, and~the~discrete~dual solution $z_{10}^{rt}\in \mathcal{R}T^0_N(\mathcal{T}_{10})$ are plotted
    on the  triangulation $\mathcal{T}_{10}$, which~has~$1858$~degrees~of~freedom. Figure \ref{fig:rate_singularity} demonstrates that the adaptive Algorithm \ref{alg:afem}  improves the experimental convergence rate of about $\frac{3}{4}$ for uniform mesh-refinement to the optimal value $1$. For uniform mesh-refinement, we expect an asymptotic convergence rate $\frac{3}{4}$ due to the corner singularity.
    Since not all quantities in the error measure $\rho_{h_k}^2(v_k)$ are computable, in Figure \ref{fig:rate_singularity},~we~employ~the~reduced~error~measure 
    \begin{align*}
        \tilde{\rho}_k^2(v_k)\coloneqq \tfrac{1}{2}\|\nabla v_k-\nabla u\|_\Omega^2+\langle-\Lambda,v_k-u\rangle_{\Omega}+\|\nabla u- \nabla_{\! h_k} u_k^{cr}\|_{\Omega}^2+(-\smash{\overline{\lambda}}_k^{cr}, \Pi_{h_k}(u-\chi))_{\Omega}\,,
    \end{align*}
    where we exploit for the computation of the first two terms the identity \eqref{eq:co-co}. 

    \begin{figure}[H]
        \centering
        \includegraphics[width=14.5cm]{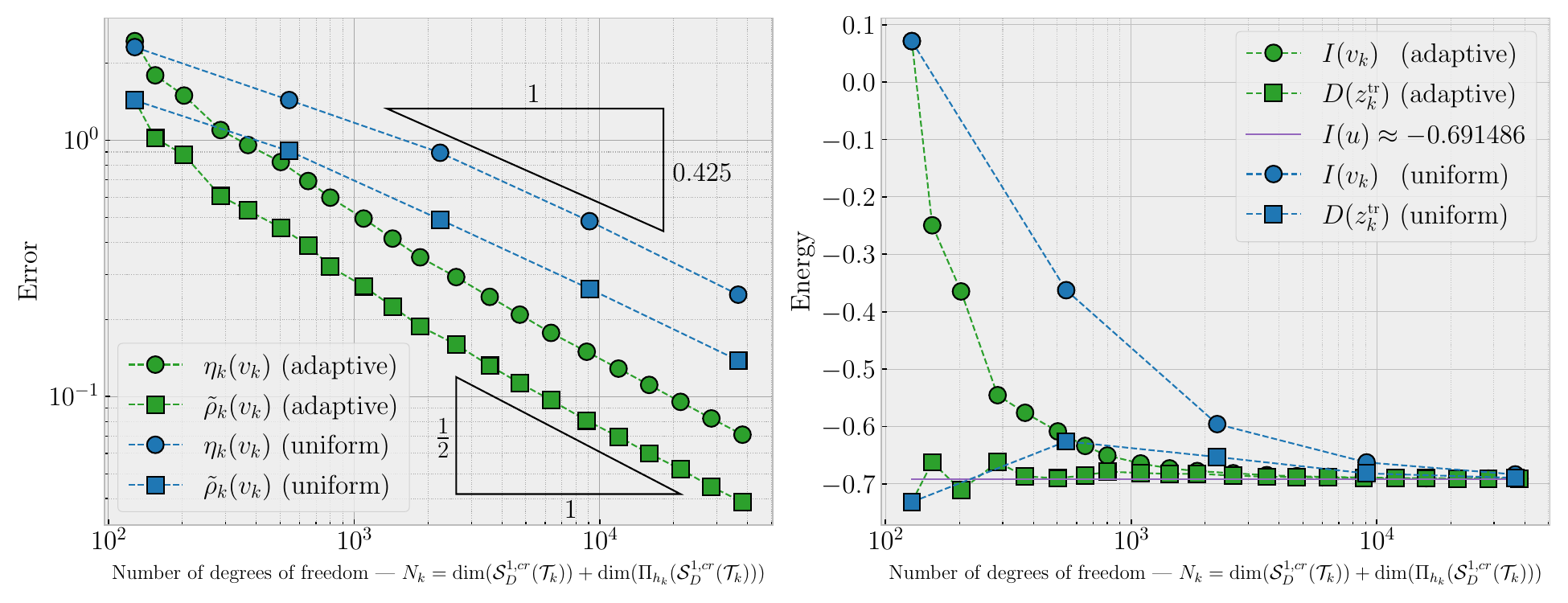}
       \caption{\textit{left:} Plots of $\eta_k^2(v_k)$ and $\tilde{\rho}_k^2(v_k)$ for $v_k\coloneqq \max\{\Pi_{h_k}^{av} u_k^{cr},\chi\}\in K$ using adaptive mesh refinement for $k=0,\dots,20$ and using uniform mesh refinement for $k=0,\dots, 4$. \textit{right:} Plots of  $I(v_k)$ (\textit{cf}.\ \eqref{eq:obstacle_primal}), for $v_k\coloneqq \max\{\Pi_{h_k}^{av} u_k^{cr},\chi\}\in K$ and $D(z_k^{rt})$ (\textit{cf}.\ \eqref{eq:obstacle_dual}), using adaptive mesh refinement for $k=0,\dots,20$ and using uniform mesh refinement for $k=0,\dots, 4$.}
        \label{fig:rate_singularity}
    \end{figure}

    \begin{figure}[H]
        \centering
        \includegraphics[width=12cm]{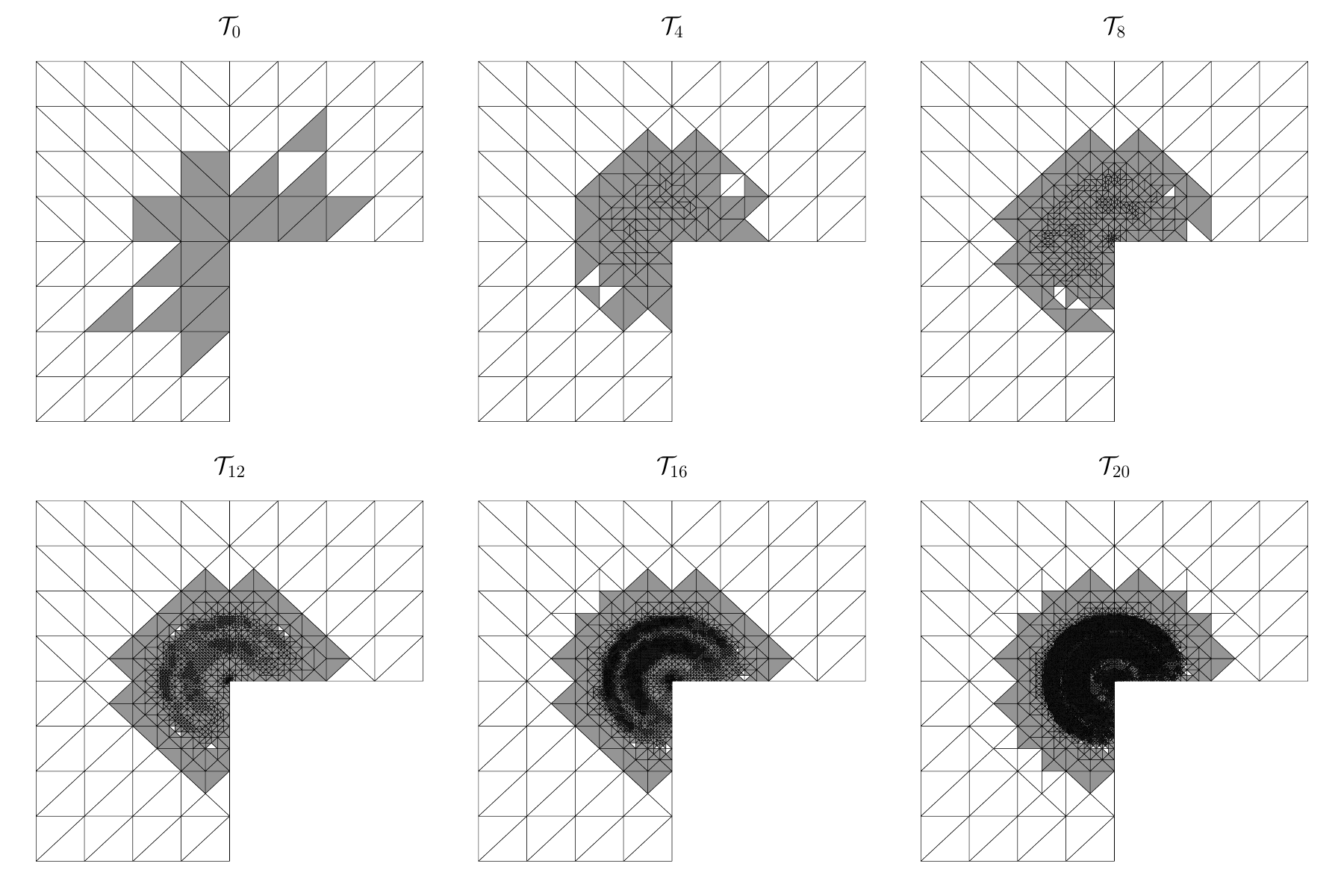}
        \caption{Adaptively refined meshes $\mathcal{T}_k$, $k\in \{0,4,8,12,16,20\}$, with approximate contact zones $\mathcal{C}_k^{cr}\coloneqq\{\Pi_{h_k}u_k^{cr}=0\}=\{\smash{\overline{\lambda}}_k^{cr}<0\}$, $k\in \{0,4,8,12,16,20\}$, shown in white.}
        \label{fig:triang_singularity}
    \end{figure}\vspace{-0.75cm}\enlargethispage{1cm}

     \begin{figure}[H]
        \centering
        \includegraphics[width=12cm]{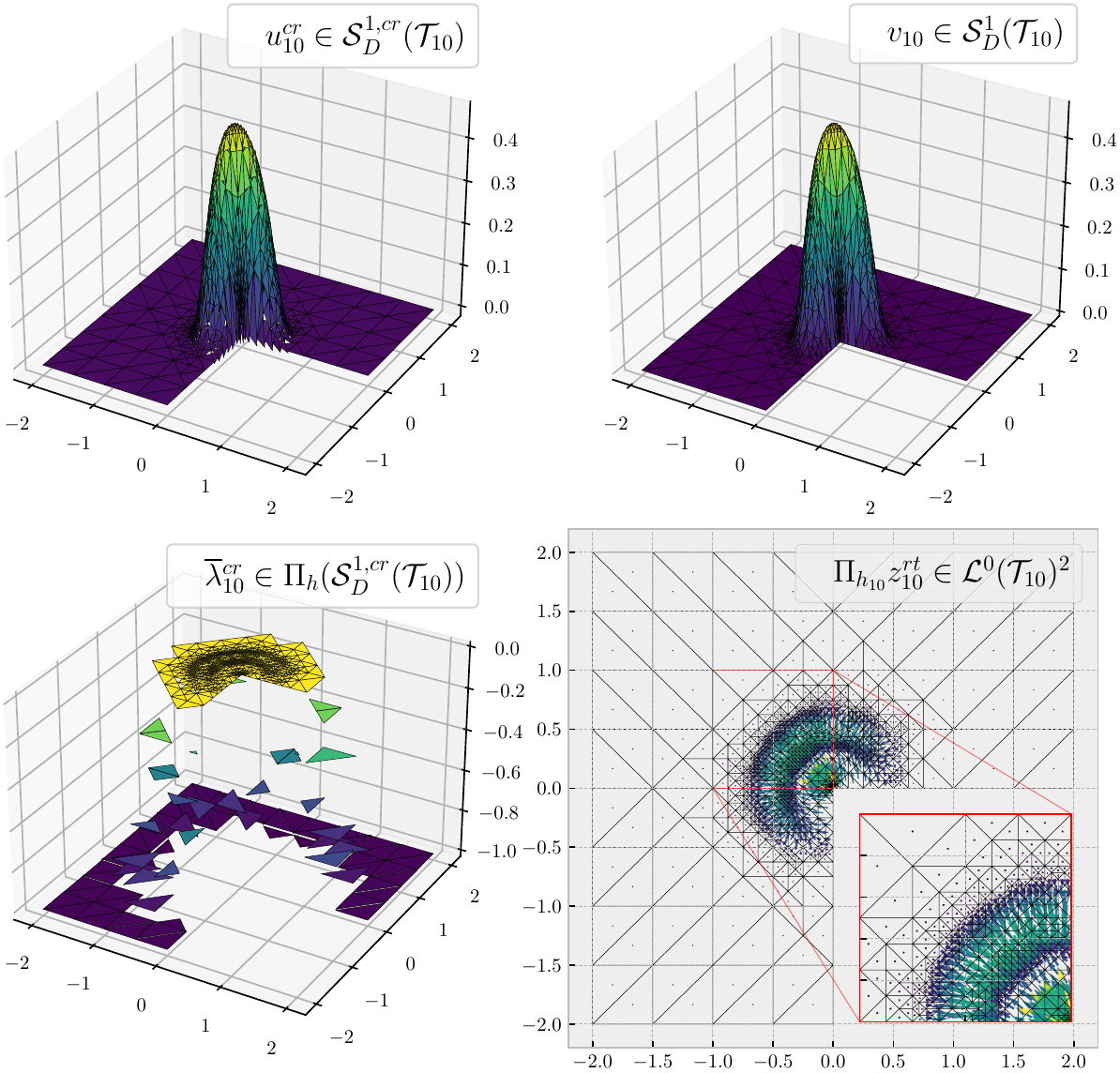}\vspace{-0.25cm}
        \caption{\textit{upper left:} discrete primal solution $u_{10}^{cr}\in \mathcal{S}^{1,cr}_D(\mathcal{T}_{10})$; \textit{upper right:} node-averaged discrete primal solution $\Pi_{h_{10}}^{av}u_{10}^{cr}\in \mathcal{S}^1_D(\mathcal{T}_{10})$; \textit{lower left:} discrete Lagrange multiplier
    $\smash{\overline{\lambda}}_{10}^{cr}\in \Pi_{h_{10}}(\mathcal{S}^{1,cr}_D(\mathcal{T}_{10}))$; \textit{lower right:} and discrete dual solution $z_{10}^{rt}\in \mathcal{R}T^0_N(\mathcal{T}_{10})$.}
        \label{fig:solution_singularity}
    \end{figure}

    \newpage
    \subsubsection{Example with unknown exact solution}

    \qquad We examine an example from \cite{BC08}. In this example, we let $\Omega\coloneqq (-1,1)^2 $, $ \Gamma_D \coloneqq \partial\Omega$, $ \Gamma_N \coloneqq \emptyset$, $f=1\in L^2(\Omega)$, and $\chi=\textup{dist}(\cdot,\partial\Omega)\in H^1_D(\Omega)$. The primal solution $u\in K$ is~not~known~and cannot be expected to satisfy $u\in H^2(\Omega)$ inasmuch~as~$\chi\notin H^2(\Omega)$, so that uniform~mesh~refinement is \hspace{-0.1mm}expected \hspace{-0.1mm}to \hspace{-0.1mm}yield \hspace{-0.1mm}a \hspace{-0.1mm}reduced \hspace{-0.1mm}error \hspace{-0.1mm}decay \hspace{-0.1mm}rate \hspace{-0.1mm}compared~\hspace{-0.1mm}to~\hspace{-0.1mm}the~\hspace{-0.1mm}optimal~\hspace{-0.1mm}linear~\hspace{-0.1mm}error~\hspace{-0.1mm}decay~\hspace{-0.1mm}rate.

    The coarsest triangulation $\mathcal{T}_0$ in Figure \ref{fig:triang_pyramid} (and starting triangulation in Algorithm \ref{alg:afem}) consists of 64 elements. 
    More precisely, Figure \ref{fig:triang_pyramid} displays
    the triangulations $\mathcal{T}_k$, $k\in \{0,5,10,15,20,25\}$, \hspace{-0.15mm}generated \hspace{-0.15mm}by \hspace{-0.15mm}the \hspace{-0.15mm}adaptive \hspace{-0.15mm}Algorithm \hspace{-0.15mm}\ref{alg:afem}.
    \hspace{-0.15mm}The \hspace{-0.15mm}approximate \hspace{-0.15mm}contact \hspace{-0.15mm}zones \hspace{-0.15mm}${\mathcal{C}_k^{cr}\coloneqq\{\Pi_{h_k}u_k^{cr}=\chi_k\}}$ $=\{\smash{\overline{\lambda}}_k^{cr}<0\}$, $k\in \{0,5,10,15,20,25\}$, where $\chi_k\coloneqq \Pi_{h_k}\chi\in \mathcal{L}^0(\mathcal{T}_k)$ for every  $k\in \{0,5,10,15,20,25\}$,
    are plotted in white in Figure \ref{fig:triang_pyramid} while their complements are  shaded. Note that for every
    $k\in \mathbb{N}$, we have that $\chi=\Pi_h^{cr}\chi\in \mathcal{S}^1_D(\mathcal{T}_k)$ and $f=f_{h_k}\in \mathcal{L}^0(\mathcal{T}_k)$.

    This example is different from the previous examples; in the sense that the solution and the obstacle are non-smooth along the lines $\mathcal{C}\coloneqq \{(x,y)^\top \in \Omega\mid x=y\text{ or }x=1-y\}$. Algorithm \ref{alg:afem} refines the mesh towards these lines as can be seen in Figure \ref{fig:triang_pyramid}. In addition, the approximate contact zones $\mathcal{C}_k^{cr}$, $k\in\{0,\dots,25\}$, reduce to $\mathcal{C}$. 
    This behavior can also be observed in Figure \ref{fig:solution_singularity}, where the discrete primal solution $u_{15}^{cr}\in \mathcal{S}^{1,cr}_D(\mathcal{T}_{15})$, the node-averaged discrete primal solution $\Pi_h^{av}u_{15}^{cr}\in \mathcal{S}^1_D(\mathcal{T}_{15})$, the discrete Lagrange multiplier
    $\smash{\overline{\lambda}}_{15}^{cr}\in \Pi_{h_{15}}(\mathcal{S}^{1,cr}_D(\mathcal{T}_{15}))$, and~the~discrete~dual solution $z_{15}^{rt}\in \mathcal{R}T^0_N(\mathcal{T}_{15})$ are plotted
    on the  triangulation $\mathcal{T}_{15}$, which~has~$3769$~degrees~of~freedom. Algorithm \ref{alg:afem}  improves the experimental convergence rate of about $\frac{1}{2}$ for uniform mesh-refinement to the quasi-optimal value $1$. 
    Since not all quantities in the error measure $\rho_{h_k}^2(v_k)$ are computable, in Figure \ref{fig:rate_pyramid},~we~employ~the~reduced~error~measure 
    \begin{align*}
        \tilde{\rho}_k^2(v_k)\coloneqq \tfrac{1}{2}\|\nabla v_k-\nabla u\|_\Omega^2+\langle-\Lambda,v_k-u\rangle_{\Omega}\,,
    \end{align*}
     where we exploit for the computation of $\tilde{\rho}_k^2(v_k)$ the identity \eqref{eq:co-co} and approximate the value $I(u)$ via Aitken's $\delta^2$-process (\textit{cf}.\ \cite{Ait26}). More precisely, we always employ the approximation $I(u)\approx \epsilon_{27}$, where the sequence $(\epsilon_k)_{k\in \mathbb{N};k\ge 2}$, for every $k\in \mathbb{N}$ with $k\ge 2$, is defined by
    \begin{align*}
       \epsilon_k\coloneqq \frac{I(v_k)I(v_{k-2})-I(v_{k-1})^2}{I(v_k)-2\,I(v_{k-1})+I(v_{k-2})}\in \mathbb{R}\,.
    \end{align*}
    However, it remains unclear whether this is a sufficiently accurate approximation of the exact errors $\rho_{h_k}^2(v_k)$, $k=0,\dots,25$.

     \begin{figure}[H]
        \centering
        \includegraphics[width=14.5cm]{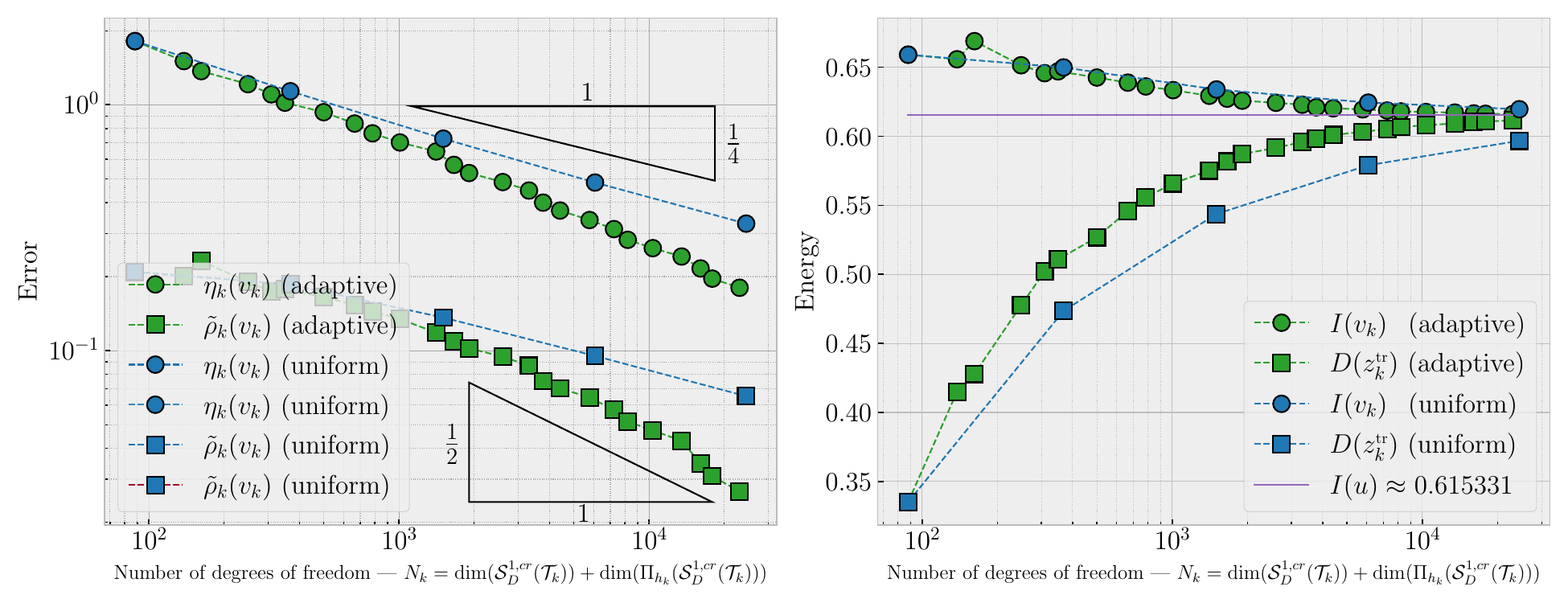}
        \caption{\textit{left:} Plots of $\eta_k^2(v_k)$ and $\tilde{\rho}_k^2(v_k)$ for $v_k\coloneqq \max\{\Pi_{h_k}^{av} u_k^{cr},\chi\}\in K$ using adaptive mesh refinement for $k=0,\dots,25$ and using uniform mesh refinement for $k=0,\dots, 4$. \textit{right:} Plots of  $I(v_k)$, \textit{cf}.\ \eqref{eq:obstacle_primal}, for $v_k\coloneqq \max\{\Pi_{h_k}^{av} u_k^{cr},\chi\}\in K$ and $D(z_k^{rt})$, \textit{cf}.\ \eqref{eq:obstacle_dual}, using adaptive mesh refinement for $k=0,\dots,25$ and using uniform mesh refinement for $k=0,\dots, 4$.}
        \label{fig:rate_pyramid}
    \end{figure}

     \begin{figure}[H]
        \centering
        \includegraphics[width=12cm]{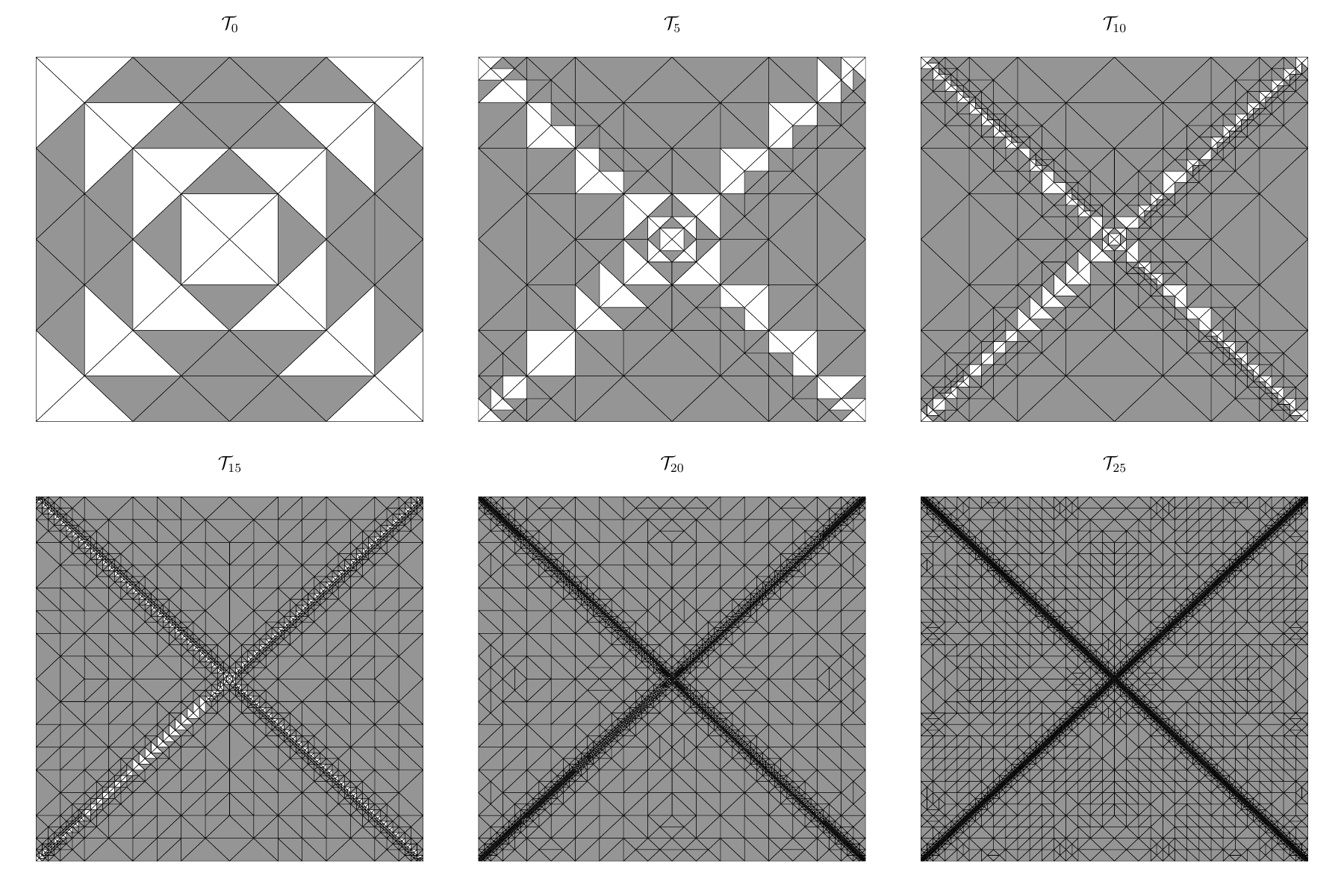}
        \caption{Adaptively refined meshes $\mathcal{T}_k$, $k\in\{0,5,10,15,20,25\}$, with approximate~contact~zones $\mathcal{C}_k^{cr}\coloneqq\{\Pi_{h_k}u_k^{cr}=\Pi_{h_k}\chi_{h_k}\}=\{\smash{\overline{\lambda}}_k^{cr}<0\}$, $k\in\{0,5,10,15,20,25\}$, shown in white.}
        \label{fig:triang_pyramid}
    \end{figure}\vspace{-0.5cm}\enlargethispage{1cm}

     \begin{figure}[H]
        \centering
        \includegraphics[width=12cm]{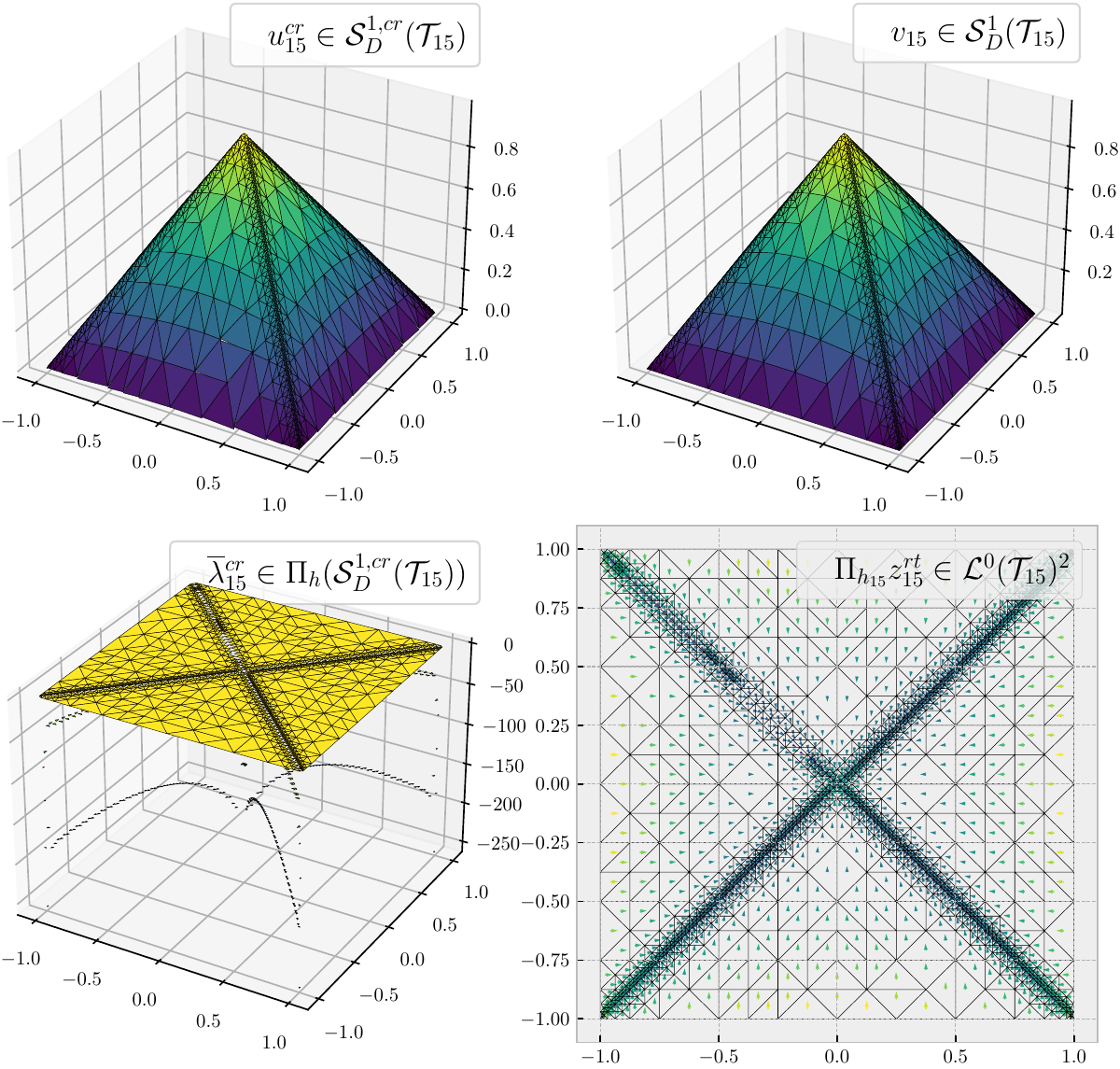}\vspace{-0.25cm}
        \caption{\textit{upper left:} discrete primal solution $u_{15}^{cr}\in \mathcal{S}^{1,cr}_D(\mathcal{T}_{15})$; \textit{upper right:} node-averaged discrete primal solution $\Pi_{h_{15}}^{av}u_{15}^{cr}\in \mathcal{S}^1_D(\mathcal{T}_{15})$; \textit{lower left:} discrete Lagrange multiplier
    $\smash{\overline{\lambda}}_{15}^{cr}\in \Pi_{h_{15}}(\mathcal{S}^{1,cr}_D(\mathcal{T}_{15}))$; \textit{lower right:} discrete dual solution $z_{15}^{rt}\in \mathcal{R}T^0_N(\mathcal{T}_{15})$.}
        \label{fig:solution_pyramid}
    \end{figure}
    \newpage

    \appendix
    \section{Appendix}\label{sec:medius}

        \qquad In this appendix, 
        we derive local efficiency estimates for the Crouzeix--Raviart approximation \eqref{eq:obstacle_discrete_primal} of \eqref{eq:obstacle_primal}, which, in turn, imply the following non-conforming efficiency result.

        \begin{theorem}\label{thm:best-approxCR}
   There exists a constant $c>0$, depending on the chunkiness $\omega_0>0$, such that
			\begin{align*}
					\|\nabla_h v_h- \nabla_h u_h^{cr}\|_{\Omega}^2\leq c\,\big\{\| \nabla_h v_h-\nabla u\|_{\Omega}^2
					    	+\|\smash{\overline{\Lambda}}_h^{cr}-\Lambda\|_{*,\Omega}^2+\mathrm{osc}_h^2(f)\big\}\,,
			\end{align*}
           where for every $\mathcal{M}_h\subseteq \mathcal{T}_h $, we define $\mathrm{osc}_h^2(f,\mathcal{M}_h)\coloneqq \sum_{T\in \mathcal{M}_h}{\mathrm{osc}_h^2(f,T)}$,
		    where $ \mathrm{osc}_h^2(f,T)\coloneqq \|h_T(f-f_h)\|_T^2$ for every $T\in \mathcal{T}_h$, 
      and $\mathrm{osc}_h^2(f)\coloneqq\mathrm{osc}_h^2(f,\mathcal{T}_h)$.
	\end{theorem}

            The proof of Theorem \ref{thm:best-approxCR} involves two tools.\vspace{-1mm}\enlargethispage{6mm}

        \subsection{Node-averaging quasi-interpolation operator}\label{subsec:node-averagin}
		  
		  \qquad The \hspace{-0.1mm}first \hspace{-0.1mm}tool \hspace{-0.1mm}is \hspace{-0.1mm}the \hspace{-0.1mm}\textit{node-averaging \hspace{-0.1mm}quasi-interpolation \hspace{-0.1mm}operator}   $\Pi_h^{av}\colon\mathcal{L}^1(\mathcal{T}_h)\hspace{-0.075em}\to\hspace{-0.075em} \mathcal{S}^1_D(\mathcal{T}_h)$,~that, 
		  denoting~for~every~${z\in \mathcal{N}_h}$, by ${\mathcal{T}_h(z)\coloneqq \{T\in \mathcal{T}_h\mid z\in T\}}$ the set of elements sharing~$z$, for every ${v_h\in \mathcal{L}^1(\mathcal{T}_h)}$, is defined by
	\begin{align*}
		\Pi_h^{av}v_h\coloneqq \sum_{z\in \smash{\mathcal{N}_h}}{\langle v_h\rangle_z\varphi_z}\,,\qquad \langle v_h\rangle_z\coloneqq \begin{cases}
			\frac{1}{\textup{card}(\mathcal{T}_h(z))}\sum_{T\in \mathcal{T}_h(z)}{(v_h|_T)(z)}&\;\text{ if }z\in \Omega\cup \Gamma_N\\
			0&\;\text{ if }z\in \Gamma_D
		\end{cases}\,,
	\end{align*}
	where we denote by $(\varphi_z)_{\smash{z\in \mathcal{N}_h}}$ the nodal basis of $\mathcal{S}^1(\mathcal{T}_h)$.
	If $p\in [1,\infty)$, then, there~exists~a~constant $c>0$, depending on $p\in [1,\infty)$ and the chunkiness ${\omega_0>0}$, such that for every $v_h\in \smash{\smash{\mathcal{S}^{1,cr}_D(\mathcal{T}_h)}}$, ${T\in \mathcal{T}_h}$, and $m\in\{0,1\}$,~we~have~that  (\textit{cf}.\ \cite[Appx.~A.2]{BKAFEM22})
    \begin{itemize}[noitemsep,topsep=2pt,leftmargin=!,labelwidth=\widthof{\quad(AV.2)},font=\itshape]
     \item[(AV.1)] \hypertarget{AV.1}{} $\|v_h-\Pi_h^{av}v_h\|_T\leq c_{\textit{av}}\,h_T\,\|\nabla_h v_h\|_{\omega_T}$\,;
     \item[(AV.2)] \hypertarget{AV.2}{} $\| \nabla \Pi_h^{av}v_h\|_T\leq c_{\textit{av}}\,\|\nabla_h v_h\|_{\omega_T}$\,.\vspace{-1mm}
    \end{itemize}

    \subsection{Local efficiency estimates}

     \qquad The second tool consists in local efficiency estimates that are  based on standard bubble function techniques (\textit{cf}.\ \cite{Ver94}).
    
		\begin{lemma}\label{lem:efficiency}
		There exists a constant $c>0$, depending only the chunkiness $\omega_0>0$, such that for every $v_h\in \mathcal{L}^1(\mathcal{T}_h)$, $T\in \mathcal{T}_h$, and $S\in \mathcal{S}_h^{i}$, respectively, it holds that
		    \begin{align}
		        \|h_T(f_h-\smash{\overline{\lambda}}_h^{cr})\|_T^2&\leq c\,\big\{\|\nabla  v_h-\nabla u\|_T^2+\|\smash{\overline{\Lambda}}_h^{cr}-\Lambda\|_{*,T}^2+\mathrm{osc}_h(f,T)\big\}\,,\label{lem:efficiency.1}\\
		        h_S\,\|\jump{\nabla_h v_h\cdot n}_S\|_S^2&\leq c\,\big\{\|\nabla_h v_h-\nabla u\|_{\omega_S}^2+\|\smash{\overline{\Lambda}}_h^{cr}-\Lambda\|_{*,\omega_S}^2+\mathrm{osc}_h(f,\omega_S)\big\}\,,\label{lem:efficiency.2}
		    \end{align}
            where $\smash{\|\smash{\overline{\Lambda}}_h^{cr}-\Lambda\|_{*,\omega}\coloneqq \|\smash{\overline{\Lambda}}_h^{cr}-\Lambda\|_{\smash{(H^1_D(\omega))^*}}}$ for any open set $\omega\subseteq\Omega $.
		\end{lemma}
		
		\begin{proof}
		    
		    \hspace{-0.17em}\textit{ad \hspace{-0.15mm}\eqref{lem:efficiency.1}.} \hspace{-0.17em}Let \hspace{-0.15mm}$T\hspace{-0.17em}\in\hspace{-0.17em} \mathcal{T}_h$\hspace{-0.15mm} be \hspace{-0.15mm}fixed, \hspace{-0.15mm}but \hspace{-0.15mm}arbitrary. \hspace{-0.15mm}Then, \hspace{-0.15mm}there \hspace{-0.15mm}exists \hspace{-0.15mm}a \hspace{-0.15mm}bubble~\hspace{-0.15mm}\mbox{function}~\hspace{-0.15mm}${b_T\hspace{-0.17em}\in\hspace{-0.17em} H^1_0(T)}$ such that $0\hspace{-0.16em}\leq \hspace{-0.16em}b_T\hspace{-0.16em}\leq\hspace{-0.1em} c_b $ in $T$, $\vert \nabla b_T\vert\hspace{-0.16em}\leq\hspace{-0.16em} c_b\,\smash{h^{-1}_T}$ in $T$, and  $\smash{\fint_T{b_T\,\mathrm{d}x}}\hspace{-0.16em}=\hspace{-0.16em}1$, where the constant~${c\hspace{-0.16em}>\hspace{-0.16em}0}$~depends only on the chunkiness $\omega_0>0$.
		    Using \eqref{eq:augmented_problem} and integration-by-parts, taking into account that $\nabla_h  v_h\in  (\mathcal{L}^0(\mathcal{T}_h))^d$ and ${b_T\in H^1_0(T)}$ 
		    in doing so, for every $\mu\in \mathbb{R}$, we find that
		    \begin{align}
		        (\nabla u-\nabla v_h, \nabla(\mu b_T))_T+\langle \smash{\overline{\Lambda}}_h^{cr}-\Lambda,\mu b_T\rangle_T=(f-\smash{\overline{\lambda}}_h^{cr},\mu b_T)_T\,.\label{lem:efficiency.5}
		    \end{align}
		    For the particular choice $\mu=\mu_T\coloneqq h_T(f_h-\smash{\overline{\lambda}}_h^{cr})\in \mathbb{R}$ in \eqref{lem:efficiency.5} and applying the $\varepsilon$-Young inequality $ab\hspace{-0.1em}\leq\hspace{-0.1em} \frac{1}{4\varepsilon}a^2+\varepsilon b^2$, valid for all $a,b\hspace{-0.1em}\ge\hspace{-0.1em} 0$ and $\varepsilon\hspace{-0.1em}>\hspace{-0.1em}0$, also using that $\vert b_T\vert\hspace{-0.1em}\leq\hspace{-0.1em} c_b$ in~$T$~and~${h_T\vert \nabla b_T\vert\hspace{-0.1em}\leq\hspace{-0.1em} c_b}$~in~$T$,  we observe that
		    \begin{align}
		            \|h_T(f_h-\smash{\overline{\lambda}}_h^{cr})\|_T^2&=(f-\smash{\overline{\lambda}}_h^{cr},h_T\mu_T b_T)_T+(f_h-f,h_T\mu_T b_T)_T\label{lem:efficiency.7}
		            \\&= (\nabla u-\nabla v_h, \nabla(h_T\mu_T b_T))_T+\langle \smash{\overline{\Lambda}}_h^{cr}-\Lambda,h_T\mu_T b_T\rangle_T\notag
              +(f_h-f,h_T\mu_T b_T)_T
                    \\&\leq \tfrac{1}{4\varepsilon}\,\big\{\|\nabla v_h-\nabla     u\|_T^2+\|\smash{\overline{\Lambda}}_h^{cr}-\Lambda\|_{*,T}^2+\mathrm{osc}_h(f,T)\big\}\notag
                    + 3\,\varepsilon\,c_b^2\,\|h_T(f_h-\smash{\overline{\lambda}}_h^{cr})\|_T^2\,.
		    \end{align}
            For the particular choice $\varepsilon=\smash{\frac{1}{6c_b^2}}>0$ in \eqref{lem:efficiency.7}, we conclude that \eqref{lem:efficiency.1} applies.
		    
		    \textit{ad \eqref{lem:efficiency.2}.} Let $S\hspace{-0.1em}\in\hspace{-0.1em} \mathcal{S}_h^{i}$ be fixed, but arbitrary. Then, there exists a bubble~function~${b_S\hspace{-0.2em}\in\hspace{-0.2em} W^{1,p}_0(\omega_S)}$ such that $0\hspace{-0.16em}\leq\hspace{-0.2em} b_S\hspace{-0.2em}\leq\hspace{-0.2em} c_b $ in $\omega_S$, $\vert \nabla b_S\vert\hspace{-0.2em}\leq\hspace{-0.2em}  c_b\,\smash{h^{-1}_S}$ in $\omega_S$, and  $\smash{\fint_S{b_S\,\mathrm{d}s}}\hspace{-0.2em}=\hspace{-0.2em}1$, where~the~constant~${c\hspace{-0.2em}>\hspace{-0.2em}0}$~depends only on the chunkiness $\omega_0>0$.
		    Using \eqref{eq:augmented_problem} and integration-by-parts, taking into account~that $\nabla_h v_h\in  (\mathcal{L}^0(\mathcal{T}_h))^d$ and ${b_S\in W^{1,p}_0(\omega_S)}$ in doing so, for every $\mu\in \mathbb{R}$, we find that
		    \begin{align}\label{lem:efficiency.10}
		        (\nabla u-\nabla_h v_h, \nabla(\mu b_S))_{\omega_S}+\langle \smash{\overline{\Lambda}}_h^{cr}-\Lambda,\mu b_S\rangle_{\omega_S}=(f-\smash{\overline{\lambda}}_h^{cr},\mu b_S)_{\omega_S}- \vert S\vert \jump{\nabla_h v_h\cdot n}_S\mu\,.
		    \end{align}
		    Let $T\hspace{-0.05em}\in\hspace{-0.05em} \mathcal{T}_h$ be with $T\hspace{-0.05em}\subseteq\hspace{-0.05em} \omega_S$. Then, for the particular choice ${\mu\hspace{-0.05em}=\hspace{-0.05em}\mu_{S}\hspace{-0.05em}\coloneqq \hspace{-0.05em}\smash{\frac{\vert \omega_S\vert}{\vert S\vert}}\jump{\nabla_h v_h\cdot n}_S\hspace{-0.05em}\in\hspace{-0.05em} \mathbb{R}}$~in~\eqref{lem:efficiency.10}, $\vert \omega_S\vert \leq c_{\omega_0}\vert S\vert h_S$ for all $S\in \mathcal{S}^{i}_h$, where $c_{\omega_0}>0$ depends only on the chunkiness $\omega_0>0$, and applying the $\varepsilon$-Young inequality $ab\leq \smash{\frac{1}{4\varepsilon}}a^2+\varepsilon b^2$ valid for all $a,b\ge 0$ and $\varepsilon>0$, also using that $\vert b_S\vert\leq c_b$ in~$\omega_S$~and~${h_S\vert \nabla b_S\vert\leq c_b}$~in~$\omega_S$,  we observe that
		    \begin{align}
		            h_S\|\jump{\nabla_h v_h\cdot n}_S\|_S^2\notag
              &= -\tfrac{\vert \omega_S\vert}{\vert S\vert}\big\{(\nabla u-\nabla_h v_h, \nabla(\mu_S     b_S))_{\omega_S}
                    \hspace{-0.1em}+\hspace{-0.1em}\langle \smash{\overline{\Lambda}}_h^{cr}\hspace{-0.1em}-\hspace{-0.1em}\Lambda,\mu_S b_S\rangle_{\omega_S}\hspace{-0.1em}-\hspace{-0.1em}(f-\smash{\overline{\lambda}}_h^{cr},\mu_{S} b_S)_{\omega_S}\big\}\\&
                    \leq \tfrac{c_{\omega_0}}{4\varepsilon}\,\big\{\|\nabla_h v_h-\nabla u\|_{\omega_S}^2+\|\smash{\overline{\Lambda}}_h^{cr}-\Lambda\|_{*,\omega_S}^2+\|h_{\mathcal{T}}(f_h-\lambda_h^{cr})\|_{\omega_S}^2+\textup{osc}_h^2(f,\omega_S)\big\}\notag\\&\quad+\varepsilon\,c_{\omega_0}\,c_b^2\,h_S\|\jump{\nabla_h v_h\cdot n}_S\|_S^2\,.\label{lem:efficiency.12.1}
		    \end{align}
            Using \eqref{lem:efficiency.12.1} and \eqref{lem:efficiency.1} in \eqref{lem:efficiency.10}, for  $\varepsilon>0$ sufficiently small, conclude that \eqref{lem:efficiency.2} applies.
		\end{proof}

        From Lemma \ref{lem:efficiency} we can derive the following global efficiency result.
        
        \begin{lemma}\label{lem:efficiency_global}
		There exists a constant $c>0$, depending only the chunkiness $\omega_0>0$, such that for every $v_h\in \mathcal{L}^1(\mathcal{T}_h)$, it holds that
		    \begin{align}
		        \|h_{\mathcal{T}}(f_h-\smash{\overline{\lambda}}_h^{cr})\|_{\Omega}^2&\leq c\,\big\{\|\nabla_h  v_h-\nabla u\|_{\Omega}^2+\|\smash{\overline{\Lambda}}_h^{cr}-\Lambda\|_{*,\Omega}^2+\mathrm{osc}_h^2(f)\big\}\,,\label{lem:efficiency_global.1}\\
		        \|h_{\mathcal{S}}^{1/2}\jump{\nabla_h v_h\cdot n}\|_{\mathcal{S}_h^{i}}^2&\leq c\,\big\{\|\nabla_h v_h-\nabla u\|_{\Omega}^2+\|\smash{\overline{\Lambda}}_h^{cr}-\Lambda\|_{*,\Omega}^2+\mathrm{osc}_h^2(f)\big\}\,.\label{lem:efficiency_global.2}
		    \end{align}
		\end{lemma}

        \begin{proof}
            \textit{ad \eqref{lem:efficiency_global.1}.} For $\mu_{\mathcal{T}} b_{\mathcal{T}} \coloneqq \sum_{T\in \mathcal{T}_h}{\mu_Tb_T}\in H^1_D(\Omega)$ in the proof of \eqref{lem:efficiency.1}, from \eqref{lem:efficiency.7},~we~obtain
            \begin{align*}
		            \|h_{\mathcal{T}}(f_h-\smash{\overline{\lambda}}_h^{cr})\|_{\Omega}^2
		            = (\nabla u-\nabla_h v_h, \nabla(h_{\mathcal{T}}\mu_{\mathcal{T}} b_{\mathcal{T}}))_{\Omega}+\langle \smash{\overline{\Lambda}}_h^{cr}-\Lambda,h_{\mathcal{T}}\mu_{\mathcal{T}} b_{\mathcal{T}}\rangle_{\Omega}
              +(f_h-f,h_{\mathcal{T}}\mu_{\mathcal{T}} b_{\mathcal{T}})_{\Omega}\,,
		    \end{align*}
            which together with the $\varepsilon$-Young inequality and  $\vert b_{\mathcal{T}}\vert+ h_{\mathcal{T}}\vert \nabla b_{\mathcal{T}}\vert\leq c_b$~in~$\Omega$ implies \eqref{lem:efficiency_global.1}.

             \textit{ad \eqref{lem:efficiency_global.2}.} For $\mu_{\mathcal{S}} b_{\mathcal{S}} \coloneqq \sum_{S\in \mathcal{S}_h^{i}}{\mu_Sb_S}\in H^1_D(\Omega)$ in the proof of \eqref{lem:efficiency.2}, from \eqref{lem:efficiency.12.1},~we~obtain
            \begin{align*}
		            \|h_{\mathcal{S}}^{1/2}\jump{\nabla_h v_h\cdot n}\|_{\mathcal{S}_h^{i}}^2&\leq \vert (\nabla u-\nabla_h v_h, \nabla(\mu_{\mathcal{S}} b_{\mathcal{S}}))_{\Omega}
                    +\langle \smash{\overline{\Lambda}}_h^{cr}-\Lambda,\mu_{\mathcal{S}} b_{\mathcal{S}}\rangle_{\Omega}-(f-\smash{\overline{\lambda}}_h^{cr},\mu_{\mathcal{S}} b_{\mathcal{S}})_{\Omega}\vert \,,
		    \end{align*}
            which together with the $\varepsilon$-Young inequality and  $\vert b_{\mathcal{S}}\vert+ h_{\mathcal{S}}\vert \nabla b_{\mathcal{S}}\vert\leq c_b$~in~$\Omega$ implies \eqref{lem:efficiency_global.2}.
        \end{proof}

        \subsection{Proof of Theorem \ref{thm:best-approxCR}}\enlargethispage{9mm}
  
		\qquad Eventually, we have everything at out disposal  to prove Theorem \ref{thm:best-approxCR}.

		\begin{proof}[Proof (of Theorem \ref{thm:best-approxCR}).]
			Let $v_h\in \mathcal{S}^{1,\textrm{cr}}_D(\mathcal{T}_h)$ be arbitrary 
				and introduce $e_h\coloneqq v_h-u_h^{cr}\in \smash{\mathcal{S}^{1,cr}_D(\mathcal{T}_h)}$. Then, resorting to \eqref{eq:augmented_problem}, \eqref{eq:obstacle_lagrange_multiplier_cr} and $f-f_h\perp_{L^2} \Pi_he_h$, we arrive at
				\begin{align}
				\begin{aligned}
					\|\nabla_h v_h-\nabla_h u_h^{cr}\|_{\Omega}^2&=( \nabla_h v_h-\nabla_h u_h^{cr},\nabla_h e_h )_\Omega\\&=
					(\nabla_h v_h,\nabla_h( e_h- \Pi_h^{av} e_h) )_\Omega	\\&\quad-( f_h-\smash{\overline{\lambda}}_h^{cr},e_h)_\Omega
					\\&\quad+( \nabla_h v_h-\nabla u ,\nabla \Pi_h^{av} e_h)_\Omega
					\\&\quad+(f,\Pi_h^{av} e_h)_\Omega-\langle \Lambda,\Pi_h^{av} e_h\rangle_{\Omega}
                    \\&=(\nabla_h v_h,\nabla_h( e_h- \Pi_h^{av} e_h) )_\Omega	
                    \\&\quad-( f-\smash{\overline{\lambda}}_h^{cr},e_h-\Pi_h^{av} e_h)_\Omega
					\\&\quad+( \nabla_h v_h-\nabla u ,\nabla \Pi_h^{av} e_h)_\Omega
					\\&\quad+(f-f_h, e_h-\Pi_h e_h)_\Omega
                    \\&\quad-\langle \smash{\overline{\Lambda}}_h^{cr}-\Lambda,\Pi_h^{av} e_h\rangle_{\Omega}
			        \\&\eqqcolon I_h^1+I_h^2+I_h^3+I_h^4+I_h^5\,.
			        \end{aligned}\label{thm:best-approx.1}
				\end{align}\newpage
				
				\textit{ad $I_h^1$.} Using that $\jump{\nabla_h v_h\cdot n( e_h\hspace{-0.15em}-\hspace{-0.15em} \Pi_h^{av} e_h)}_S\hspace{-0.15em}=\hspace{-0.15em}\jump{\nabla_h v_h\hspace{-0.15em}\cdot\hspace{-0.15em} n}_S\{e_h- \Pi_h^{av} e_h\}_S +\{\nabla_h v_h\cdot n\}_S \jump{e_h\hspace{-0.15em}-\hspace{-0.15em}\Pi_h^{av} e_h}_S$ on $S$, $\int_S{\jump{e_h-\Pi_h^{av} e_h}_S\,\textup{d}s}=0$, and $\{\nabla_h v_h\cdot n\}_S=\textup{const}$ on $S$ for all ${S\in \mathcal{S}_h^{i}}$, an element-wise integration-by-parts,  \hspace{-0.1mm}the \hspace{-0.1mm}discrete \hspace{-0.1mm}trace \hspace{-0.1mm}inequality \hspace{-0.1mm}\mbox{\cite[\hspace{-0.1mm}Lem.\  \hspace{-0.1mm}12.8]{EG21}}, \hspace{-0.1mm}(\hyperlink{AV.1}{AV.1}), \hspace{-0.1mm}the \hspace{-0.1mm}$\varepsilon$-Young~\hspace{-0.1mm}\mbox{inequality}, and  \eqref{lem:efficiency.2}, for every $\varepsilon>0$, we find that\enlargethispage{5mm}
				\begin{align}\label{thm:best-approx.3}
					\begin{aligned}
					I_h^1&=(\smash{h_{\mathcal{S}}^{1/2}}\jump{\nabla_h v_h\cdot n},\smash{h_{\mathcal{S}}^{-1/2}}\{e_h- \Pi_h^{av} e_h\})_{\mathcal{S}_h^{i}}\\
                        &\leq c_{\textit{tr}}\,\|\smash{h_{\mathcal{S}}^{1/2}}\jump{\nabla_h v_h\cdot n}\|_{\mathcal{S}_h^{i}}\|h_{\mathcal{T}}^{-1}(e_h- \Pi_h^{av} e_h)\|_{\Omega}
                        \\&\leq c_{\textit{eff}}\,\smash{\tfrac{c_{\textit{tr}}^2}{4\varepsilon}}\,\big\{\|\nabla_h v_h-\nabla u\|_{\Omega}^2+\mathrm{osc}_h(f)
					    	+\|\smash{\overline{\Lambda}}_h^{cr}-\Lambda\|_{*,\Omega}^2\big\}+ \varepsilon\,c_{\textit{av}}^2\,
					    	\|\nabla_h e_h\|_{\Omega}^2
                        \,.
				\end{aligned}	
				\end{align}
				
				\textit{ad $I_h^2$.}
				 Using  the $\varepsilon$-Young inequality, the approximation property of $\Pi_h^{av}\colon \smash{\mathcal{S}^{1,cr}_D(\mathcal{T}_h)}\to\mathcal{S}^1_D(\mathcal{T}_h)$ (\textit{cf}.\ (\hyperlink{AV.1}{AV.1})), and \eqref{lem:efficiency.1}, for every $\varepsilon>0$, we obtain
				\begin{align}
					\begin{aligned}
						I_h^2&\leq \tfrac{1}{4\varepsilon}\,\|h_{\mathcal{T}} (f-\smash{\overline{\lambda}}_h^{cr})\|^2_{\Omega}+
						\varepsilon\,\|h_{\mathcal{T}}^{-1}(e_h-\Pi_h^{av}e_h)\|_{\Omega}^2
                        \\&\leq  \tfrac{ c_{\textit{eff}}}{4\varepsilon}\,\big\{\|\nabla_h v_h-\nabla u\|_{\Omega}^2+\mathrm{osc}_h(f)
					    	+\|\smash{\overline{\Lambda}}_h^{cr}-\Lambda\|_{*,\Omega}^2\big\}+ \varepsilon\,c_{\textit{av}}^2\,
					    	\|\nabla_h e_h\|_{\Omega}^2
					    	\,.
					\end{aligned}\label{thm:best-approx.4}
				\end{align}
				
				\textit{ad $I_h^3$.}
				Using the $\varepsilon$-Young inequality, the $H^1$-stability of $\Pi_h^{av}\colon \smash{\mathcal{S}^{1,cr}_D(\mathcal{T}_h)}\to\mathcal{S}^1_D(\mathcal{T}_h)$ (\textit{cf}.\ (\hyperlink{AV.2}{AV.2})), and \eqref{lem:efficiency.1}, for every $\varepsilon>0$, we obtain
				\begin{align}
					\begin{aligned}
					I_h^3&\leq \tfrac{1}{4\varepsilon}\,\|\nabla_h v_h-\nabla u\|_{\Omega}^2+\varepsilon\,\|\nabla \Pi_h^{av}e_h\|_{\Omega}^2\\&\leq 
                  \tfrac{1}{4\varepsilon}\,\|\nabla_h v_h-\nabla u\|_{\Omega}^2+\varepsilon\,c_{\textit{av}}^2\,\|\nabla_he_h\|_{\Omega}^2\,.
				\end{aligned}\label{thm:best-approx.7}
				\end{align}
    
				\textit{ad $I_h^4$.}
                Using  the $\varepsilon$-Young inequality and the approximation property of ${\Pi_h\colon \mathcal{L}^1(\mathcal{T}_h)\to\mathcal{L}^0(\mathcal{T}_h)}$ (\textit{cf}.\ (\hyperlink{L0.1}{L0.1})), for every $\varepsilon>0$, we obtain
				\begin{align}
					\begin{aligned}
					I_h^4&\leq \tfrac{1}{4\varepsilon}\,\mathrm{osc}_h(f)+\varepsilon\,\| h_{\mathcal{T}}^{-1}(e_h- \Pi_h e_h)\|_{\Omega}^2\\
                        &\leq \tfrac{1}{4\varepsilon}\,\mathrm{osc}_h(f)+\varepsilon\,c_{\Pi}^2\,\| \nabla_he_h\|_{\Omega}^2\,.
				\end{aligned}	\label{thm:best-approx.9}
				\end{align}

                \textit{ad $I_h^5$.}
                 Using the $\varepsilon$-Young inequality, the $H^1$-stability of $\Pi_h^{av}\colon \smash{\mathcal{S}^{1,cr}_D(\mathcal{T}_h)}\to\mathcal{S}^1_D(\mathcal{T}_h)$ (\textit{cf}.\ (\hyperlink{AV.1}{AV.1})) \& (\hyperlink{AV.2}{AV.2}),  and the discrete Poincare inequality \eqref{discrete_poincare}, for every $\varepsilon>0$, we obtain
				\begin{align}
					\begin{aligned}
					I_h^5&\leq \tfrac{1}{4\varepsilon}\,\|\smash{\overline{\Lambda}}_h^{cr}-\Lambda\|_{*,\Omega}^2+\varepsilon\,(\|\Pi_h^{av}e_h\|_{\Omega}^2+\|\nabla \Pi_h^{av}e_h\|_{\Omega}^2)\\
                        &\leq \tfrac{1}{4\varepsilon}\,\|\smash{\overline{\Lambda}}_h^{cr}-\Lambda\|_{*,\Omega}^2+\varepsilon\,(1+(c^{cr}_{P})^2)\,\| \nabla_he_h\|_{\Omega}^2\,.
				\end{aligned}	\label{thm:best-approx.10}
				\end{align}
				Combining \eqref{thm:best-approx.3}, \eqref{thm:best-approx.4}, \eqref{thm:best-approx.7}, \eqref{thm:best-approx.9}, and \eqref{thm:best-approx.10} 
				in \eqref{thm:best-approx.1}, for every $\varepsilon>0$, we conclude that
				\begin{align}\label{thm:best-approx.11}
					\begin{aligned}
				\|\nabla_h v_h-\nabla_h u_h^{cr}\|_{\Omega}^2&\leq \smash{\tfrac{3+c_{\textit{eff}}(1+c_{\textit{tr}}^2)}{4\varepsilon}}
    \,\big\{\|\nabla_h v_h-\nabla u\|_{\Omega}^2+\mathrm{osc}_h(f)
		          +\|\smash{\overline{\Lambda}}_h^{cr}-\Lambda\|_{*,\Omega}^2\big\}
          \\&\quad+ \varepsilon\,\big\{3\,c_{\textit{av}}+c_{\Pi}^2+1+(c^{cr}_{P})^2\big\}\, 
					    	\|\nabla_h e_h\|_{\Omega}^2
					    	\,,
			\end{aligned}
				\end{align}
				For $\varepsilon\coloneqq \frac{1}{2(3\,c_{\textit{av}}+c_{\Pi}^2+1+(c^{cr}_{P})^2)}>0$ and $c\coloneqq (3\,c_{\textit{av}}+c_{\Pi}^2+1+(c^{cr}_{P})^2)( 3+c_{\textit{eff}}\,(1+c_{\textit{tr}}^2))>0$~in~\eqref{thm:best-approx.11}, for every $v_h\in \mathcal{S}^{1,cr}_D(\mathcal{T}_h)$, we arrive at
				\begin{align}\label{thm:best-approx.12}
					\begin{aligned}
						\|\nabla_h v_h-\nabla_hu_h^{cr}\|_{\Omega}^2\leq  c\,\big\{\|\nabla_h v_h-\nabla u\|_{\Omega}^2+\mathrm{osc}_h(f)
					    	+\|\smash{\overline{\Lambda}}_h^{cr}-\Lambda\|_{*,\Omega}^2\big\}\,,
					\end{aligned}
				\end{align}	
                which is the claimed non-conforming efficiency estimate. 
		\end{proof}
	
  Eventually, the node-averaging quasi-interpolation operator $\Pi_h^{av} \colon\mathcal{S}^{1,cr}_D(\mathcal{T}_h)\to \mathcal{S}^1_D(\mathcal{T}_h)$ satisfies the following best-approximation result with respect to Sobolev functions.

   \begin{proposition}
   \label{lem:best-approx-inv}
                There exists a
		        constant $c>0$, depending only on the chunkiness ${\omega_0>0}$, such that  for every $v_h\in \mathcal{S}^{1,cr}_D(\mathcal{T}_h)$ and $T\in \mathcal{T}_h$, it holds
			    \begin{align*}
			    \|\nabla_h v_h-\nabla \Pi_h^{av} v_h\|_{T}^2\leq c\,\inf_{v\in H^1_D(\Omega)}{\big\{\|\nabla_h v_h-\nabla v\|_{\omega_T}^2\big\}}\,.
			    \end{align*} 
                In particular, for every $v_h\in \mathcal{S}^{1,cr}_D(\mathcal{T}_h)$, $v\in H^1_D(\Omega)$, and $T\in \mathcal{T}_h$, it holds that
                \begin{align*}
			    \|\nabla \Pi_h^{av} v_h-\nabla v\|_{T}^2\leq c\,\|\nabla_h v_h-\nabla v\|_{\omega_T}^2\,.
			    \end{align*} 
	\end{proposition}

    An essential tool in the verification of Proposition \ref{lem:best-approx-inv} is the following lemma.

    \begin{lemma}
        \label{lem:node-averaging}
        There exists a
		        constant $c>0$, depending only on the chunkiness ${\omega_0>0}$,~such~that  for every $v_h\in \mathcal{S}^{1,cr}_D(\mathcal{T}_h)$ and $T\in \mathcal{T}_h$, it holds that
			    \begin{align*}
			    \|\nabla_h v_h-\nabla \Pi_h^{av} v_h\|_{T}^2\leq c\,\sum_{S\in \mathcal{S}_h(T)\setminus \Gamma_N}{h_T\,\| h_T^{-1} \jump{v_h}_S\|^2_S}\,,
			    \end{align*} 
       where $\mathcal{S}_h(T)\coloneqq \{S\in \mathcal{S}_h\mid S\cap T\neq \emptyset\}$.
    \end{lemma}

    \begin{proof}
        Due to the (local) inverse inequality (\textit{cf}.\ \cite[Lem. 12.1]{EG21}) and the node-based~norm~equivalence (\textit{cf}.\ \cite[Prop. 12.5]{EG21}), there exists a constant $c>0$, depending only~on~the~chunkiness~$\omega_0>0$, such that
        \begin{align}\label{lem:node-averaging.1}
            \begin{aligned}
                \|\nabla_h v_h-\nabla \Pi_h^{av} v_h\|_{T}^2&\leq c\,h_T^{-1}\,\|v_h- \Pi_h^{av} v_h\|_{T}^2
                     \\&\leq c\,h_T^{d-2}\,\sum_{z\in \mathcal{N}_h\cap T}{\vert  (v_h|_T)(z)- (\Pi_h^{av} v_h)(z)\vert^2}\,.
            \end{aligned}
        \end{align}
        Next, for every $z\in \mathcal{N}_h\cap T$,  we need to distinguish the cases $z\notin\Gamma_D$ and $z\in\Gamma_D$:

        \textit{Case $z\notin\Gamma_D$.} If $z\notin\Gamma_D$,  then 
        since each $T'\in \mathcal{T}_h(z)$  can be reached from $T$ via passing through a finite number\footnote{uniformly bounded by a constant depending only on the chunkiness $\omega_0>0$.} of interior sides in  $\mathcal{S}_h^{i}(T)\coloneqq \mathcal{S}_h(T)\cap \mathcal{S}_h^{i}$,
        using \cite[(22.6)]{EG21},~we~find~that
        \begin{align}\label{lem:node-averaging.2}
            \begin{aligned}
            \vert  (v_h|_T)(z)- (\Pi_h^{av} v_h)(z)\vert^2&\leq c\,\frac{1}{\textup{card}(\mathcal{T}_h(z))}\sum_{T'\in \mathcal{T}_h(z)}{\vert (v_h|_T)(z)-(v_h|_{T'})(z)\vert^2 }
            \\&\leq c\,\sum_{S\in \mathcal{S}_h^{i}(T)}{\vert \jump{v_h}_S(z)\vert^2 }
             \\&\leq c\,\sum_{S\in \mathcal{S}_h^{i}(T)}{h_T^{1-d}\|\jump{v_h}_S\|_S^2}\,.
             \end{aligned}
        \end{align}

         \textit{Case $z\in\Gamma_D$.} If $z\in\Gamma_D$, then we need to distinguish the case that $z\in\textup{int}\,\Gamma_D$, \textit{i.e.}, $z$ lies in the relative interior of $\Gamma_D$, and $z\in\partial\Gamma_D$, \textit{i.e.}, $z$ lies in the  relative boundary of $\Gamma_D$:

         \textit{Subcase $z\in\textup{int}\,\Gamma_D$.} If $z\in\text{int}\,\Gamma_D$, then
         there exists a boundary side $S\in \mathcal{S}_h(T)\setminus\Gamma_N$ with $z\in S$ and $S\subseteq \partial T$. Thus, 
         resorting to \cite[(22.6)]{EG21}, we find that
        \begin{align}\label{lem:node-averaging.3}
            \begin{aligned}
            \vert  (v_h|_T)(z)- (\Pi_h^{av} v_h)(z)\vert^2&=\vert  (v_h|_T)(z)\vert^2
            \\&=\vert \jump{v_h}_S(z)\vert^2 
             \\&\leq c\,h_T^{1-d}\,\|\jump{v_h}_S\|_S^2\,.
             \end{aligned}
        \end{align}

        \textit{Subcase $z\in\partial\Gamma_D$.} If $z\hspace{-0.1em}\in\hspace{-0.1em}\partial\Gamma_D$, then
         there exists a boundary side $S\hspace{-0.1em}\in\hspace{-0.1em} \mathcal{S}_h(T)$ with $z\hspace{-0.1em}\in\hspace{-0.1em} S$,~$S\hspace{-0.1em}\subseteq \hspace{-0.1em}\partial T$, and either $S\hspace{-0.1em}\subseteq\hspace{-0.1em} \Gamma_D$ or $S\subseteq \Gamma_N$. If $S\subseteq \Gamma_D$, then we argue as in \eqref{lem:node-averaging.3}.
         If $S\subseteq \Gamma_N$, then there exists boundary side $S'\in \mathcal{S}_h(T)\setminus \Gamma_N$ with $z\in S'$ and an element $T'\in \mathcal{T}_h$ with $z\in T'$and~$S'\subseteq T'$. 
         If $T'=T$, then we argue as in \eqref{lem:node-averaging.3}. If $T'\neq T$, then since $T'$ can be reached from $T$ via passing through a finite number of interior sides in  $\mathcal{S}_h^{i}(T)$,
         resorting to \cite[(22.6)]{EG21},  we find that
        \begin{align}\label{lem:node-averaging.4}
            \begin{aligned}
            \vert  (v_h|_T)(z)- (\Pi_h^{av} v_h)(z)\vert^2&=\vert  (v_h|_T)(z)\vert^2\\&\leq \vert  (v_h|_{T'})(z)\vert^2
            + c\,\sum_{S\in \mathcal{S}_h^{i}(T)}{\vert \jump{v_h}_S(z)\vert^2 }
            \\&\leq c\,\vert \jump{v_h}_{S'}(z)\vert^2 + c\,\sum_{S\in \mathcal{S}_h^{i}(T)}{\vert \jump{v_h}_S(z)\vert^2 }
             \\&\leq  c\,\sum_{S\in \mathcal{S}_h(T)\setminus\Gamma_N}{h_T^{1-d}\|\jump{v_h}_S\|_S^2 }\,.
              \end{aligned}
        \end{align}
        Eventually, combining \eqref{lem:node-averaging.2}--\eqref{lem:node-averaging.4} in \eqref{lem:node-averaging.1}, we conclude 
        the claimed estimate.
    \end{proof}

    \newpage
 	\begin{proof}[Proof (of Proposition \ref{lem:best-approx-inv})]
		        Using that $\|\jump{v_h}_S\|_{L^\infty(S)}\hspace{-0.15em}\leq\hspace{-0.15em} c\,\fint_{S}{\vert \jump{v_h}_S\vert\,\mathrm{d}x}$ (\textit{cf}.\ \cite[Lem.\ 12.1]{EG21}) as well as $\vert T\vert\sim h_T\vert S\vert$ for all $ T\in\mathcal{T}_h$ and $S\in \mathcal{S}_h(T)$, where the constant $c>0$ depends only on the chunkiness $\omega_0>0$, for every $T\in \mathcal{T}_h$, we infer from Lemma \ref{lem:node-averaging} that
		       \begin{align}\label{lem:best-approx-inv.1}
		            \begin{aligned}
		                \|\nabla \Pi_h^{av} v_h-\nabla_h v_h\|_{T}^2&\leq 
		                c\,\sum_{S\in \mathcal{S}_h(T)\setminus \Gamma_N}{h_T\,\| h_T^{-1} \jump{v_h}_S\|^2_S}\\& \leq 
		                c\,\sum_{S\in \mathcal{S}_h(T)\setminus \Gamma_N}{\vert T\vert\,(\vert T\vert^{-1}\|  \jump{v_h}_S\|_{L^1(S)})^2}
		                \,.
		            \end{aligned}
		       \end{align}
		      For every $S\in\mathcal{S}_h$, we denote by $\pi_h^S\colon L^1(S)\to \mathbb{R}$, the side-wise (local) $L^2$-projection~operator onto constant functions, for every $w\in L^1(S)$ defined by ${\pi_h^Sw\coloneqq \fint_S{w\,\mathrm{d}s}}$.~Since~for~every~${w\in W^{1,1}(T)}$, where $T\in \mathcal{T}_h$ with $T\subseteq \omega_S$, due to the $L^1(S)$-stability of $\pi_h^S\colon L^1(S)\to \mathbb{R}$ and~\cite[Cor.~A.19]{kr-phi-ldg}, it holds that
		      \begin{align}\label{lem:best-approx-inv.2}
                    \begin{aligned}
		          \|w-\pi_h^S w\|_{L^1(S)}&= \|w-\Pi_hw -\pi_h^S(w-\Pi_hw)\|_{L^1(S)}
		          \\&\leq 2\,\|w-\Pi_hw \|_{L^1(S)}
		          \\&\leq c\,\|\nabla w\|_{L^1(T;\mathbb{R}^d)}\,,
                    \end{aligned}
		      \end{align}
		      where $c>0$ depends only on the~chunkiness $ \omega_0>0 $. 
              Next, let $v\in H^1_D(\Omega)$ be~fixed,~but~arbitrary. Using that $\pi_h^S\jump{v_h}_S=\jump{v}_S=0$ in $L^1(S)$ for all $S\in \mathcal{S}_h(T)\setminus \Gamma_N$ and $T\in \mathcal{T}_h$ and \eqref{lem:best-approx-inv.2}, we find that
		      \begin{align}\label{lem:best-approx-inv.2.0}
		            \begin{aligned}
		            \| \jump{v_h}_S\|_{L^1(S)}&=\| \jump{v_h-v}_S-\pi_h^S\jump{v_h-v}_S\|_{L^1(S)}
		                \\&\leq \| \nabla_h v_h-\nabla v\|_{L^1(\omega_S;\mathbb{R}^d)}\,.
		                \end{aligned}
		      \end{align}
		      Then, using~in~\eqref{lem:best-approx-inv.1}, \eqref{lem:best-approx-inv.2.0}, $\vert T\vert\sim\vert \omega_T\vert\sim \vert \omega_S\vert$ for~all~${T\in\mathcal{T}_h}$~and~${S\in\mathcal{S}_h(T)}$, where $c>0$ depends only on the chunkiness $ \omega_0>0 $, and Jensen's inequality, for every $T\in \mathcal{T}_h$,~we~deduce~that 
		       \begin{align}\label{lem:best-approx-inv.3}
		            \begin{aligned}
		          \|\nabla \Pi_h^{av} v_h-\nabla_h v_h\|_{T}^2
		                &\leq c\,\sum_{S\in \mathcal{S}_h(T)\setminus \Gamma_N}{\vert \omega_S\vert\,(\vert \omega_S\vert^{-1}\|\nabla_h v_h-\nabla v\|_{L^1(\omega_S;\mathbb{R}^d)})^2}
		                \\&\leq c\,\sum_{S\in \mathcal{S}_h(T)\setminus \Gamma_N}{\| \nabla_h v_h-\nabla v\|^2_{\omega_S}}
		                \\&\leq c\,\|\nabla_h v_h-\nabla v\|_{\omega_T}^2\,.
		               \end{aligned}
		            \end{align}
		        Eventually, taking in \eqref{lem:best-approx-inv.3} the infimum with respect to $v\in H^1_D(\Omega)$, we conclude the claimed estimate.
		    \end{proof}
 
	{\setlength{\bibsep}{0pt plus 0.0ex}\small

	\bibliographystyle{aomplain}
    \bibliography{literatur}

}
	
\end{document}